\documentclass[preprint,11pt]{imsart} 

\usepackage{amsthm, amsfonts, amssymb, amsmath, enumerate, natbib, bbm,paralist,enumitem,xr}
\usepackage{mathtools}
\mathtoolsset{showonlyrefs}
\newtheorem{thm}{Theorem}[section]
\newtheorem{lem}{Lemma}[section]
\newtheorem{lems}{Lemma}[subsection]
\newtheorem{cor}{Corollary}[section]
\newtheorem{prop}{Proposition}[section]
\newtheorem{defi}{Definition}[section]

\newcounter{myalgctr}

\newenvironment{rem}{
   \vskip1mm\indent
   \refstepcounter{myalgctr}
   \textbf{Remark \themyalgctr}
   }{\hfill$\diamond$\par}  
\numberwithin{myalgctr}{section}
\providecommand{\norm}[1]{\left\lVert#1\right\rVert}

\DeclareMathOperator*{\argmin}{arg\,min}

\DeclareRobustCommand{\rchi}{{\mathpalette\irchi\relax}}
\newcommand{\irchi}[2]{\raisebox{\depth}{$#1\chi$}}
\newcommand{\E}{\mathbb{E}}
\begin{document}
 \begin{frontmatter}
\date{} 
\title{High-dimensional CLT: Improvements, Non-uniform Extensions and Large Deviations
}
\runtitle{High-dimensional CLT}
\begin{aug}
  \author{\fnms{Arun Kumar} \snm{Kuchibhotla}\ead[label=e1]{arunku@upenn.edu, somabha@upenn.edu, dban@upenn.edu}},
  \author{\fnms{Somabha} \snm{Mukherjee}\ead[label=e2]{somabha@upenn.edu}},
  \and
  \author{\fnms{Debapratim}  \snm{Banerjee}\ead[label=e3]{dban@upenn.edu}}

  \runauthor{Kuchibhotla, Mukherjee and Banerjee}

  \affiliation{University of Pennsylvania}

  \address{Department of Statistics,The Wharton School\\ University of Pennsylvania \\ 3730 Walnut Street,\\ Jon M. Huntsman Hall\\ Philadelphia, PA 19104, USA.\\ \printead{e1}}

\end{aug}
\begin{abstract}
Central limit theorems (CLTs) for high-dimensional random vectors with dimension possibly growing with the sample size have received a lot of attention in the recent times. \cite{CCK17} proved a Berry--Esseen type result for high-dimensional averages for the class of hyperrectangles and they proved that the rate of convergence can be upper bounded by $n^{-1/6}$ upto a polynomial factor of $\log p$ (where $n$ represents the sample size and $p$ denotes the dimension). Convergence to zero of the bound requires $\log^7p = o(n)$. We improve upon their result which only requires $\log^4p = o(n)$ (in the best case). This improvement is made possible by a sharper dimension-free anti-concentration inequality for Gaussian process on a compact metric space. In addition, we prove two non-uniform variants of the high-dimensional CLT based on the large deviation and non-uniform CLT results for random variables in a Banach space by Bentkus, Ra{\v c}kauskas, and Paulauskas. We apply our results in the context of post-selection inference in linear regression and of  empirical processes. 
\end{abstract}

\end{frontmatter}


\section{Introduction}
In modern statistical applications like high dimensional estimation and multiple hypothesis testing problems, the dimension of the data is often much larger than the sample size. As can be expected from the classical asymptotic theory, the central limit theorem plays a pivotal role for inference. In this paper, we prove three variants of the high-dimensional central limit theorem. The setting we use is as follows. Consider independent mean zero random vectors $X_{1},\ldots, X_{n}\in\mathbb{R}^p$ with covariance matrices $\Sigma_i := \mathbb{E}[X_iX_i^{\top}]\in\mathbb{R}^{p\times p}$. Here $p$ is allowed to be larger than $n$. Define the scaled average 
\begin{equation}\label{eq:ScaledAverageDef}
S_{n}:= \frac{X_{1} + \cdots + X_{n}}{\sqrt{n}}\in\mathbb{R}^p.
\end{equation} 
Let $Y_i$ (for $1\le i\le n$) represent a multivariate Gaussian random vector with mean zero and variance-covariance matrix $\Sigma_i$. Define the corresponding scaled average as 
\[
U_{n,0} := \frac{Y_1 + \cdots + Y_n}{\sqrt{n}}\in\mathbb{R}^p.
\]
The problem of central limit theorem is the comparison of the probabilities $\mathbb{P}(S_n\in A)$ and $\mathbb{P}(U_{n,0}\in A)$ for $A\subseteq \mathbb{R}^p$. When $p$ is fixed (or only grows at most sublinearly in $n$) which we refer to as multivariate setting, classical results show the rate $O(p^{7/4}/n^{1/2})$; see~\cite{MR643968}, \cite{sazonov1982accuracy} and \cite{MR2144310}. The case where $p$ is allowed to grow faster than $n$, which we refer to as high-dimensional setting, has received significant interest in the recent times. The series of papers \cite{CCK13, CCK15, CCK17} have studied this problem extensively under general conditions on the random vectors when the sets $A$ are sparsely convex sets and in particular hyperrectangles. The main result of \cite{CCK17} bounds the difference $\left|\mathbb{P}\left[ S_{n} \in A  \right]- \mathbb{P}\left[ U_{n,0} \in A  \right]\right|$ uniformly over $A\in\mathcal{A}^{re}$ as a function of $n$ and $p$. Here $\mathcal{A}^{re}$ is the class of all hyperrectangles. Proposition 2.1 of \citet{CCK17} implies that
\begin{equation}\label{eq:CCKResult}
\sup_{A \in \mathcal{A}^{re}}\,\left|\mathbb{P}\left( S_{n} \in A  \right)- \mathbb{P}\left( U_{n,0} \in A  \right)\right| \le C\left(\frac{\log ^{7}(pn)}{n} \right)^{{1/6}},
\end{equation}
under certain exponential tail assumption and a constant $C$ depending on the distribution of the random vectors $X_{i},\, 1\le i\le n$. In their earlier papers, \cite{CCK13, CCK15} a special sub-class of sets $\mathcal{A}^{m} \subset \mathcal{A}^{re}$ were considered, where $\mathcal{A}^{m}$ is the class of all sets $A$ of the form $A= \{ x \in \mathbb{R}^{p} : x(j) \le a \text{ for all }  1 \le j\le p \}$. Here and throughout, we use the notation $x(j)$ for a vector $x\in\mathbb{R}^p$ to represent the $j$-th coordinate of~$x$. 

From the bound~\eqref{eq:CCKResult} we require $\log^7(pn) = o(n)$ for the difference of probabilities to converge to zero uniformly. One of the lingering questions in high-dimensional CLT is the ``correct'' exponent of $\log(pn)$ that needs to be $o(n)$ for convergence to zero. By proving a dimension-free anti-concentration inequality and comparing the proof techniques from~\cite{CCK17},~\cite{MR1015294},~\cite{bentkus2000accuracy}, we reduce the requirement to $\mu^6\log^4(ep) = o(n)$ when the sets $A$ are restricted to $l_{\infty}$ balls in $\mathbb{R}^p$. Here $\mu$ represents the median of $\|U_{n,0}\|_{\infty}$. In the best case where $\mu = O(1)$ the requirement becomes $\log^4(ep) = o(n)$ and in the worst case where $\mu = O(\sqrt{\log(ep)})$ the requirement becomes $\log^7(ep) = o(n)$ coinciding with the requirement from~\eqref{eq:CCKResult}. 

Since the dependence on the sample size, $n^{-1/6}$, in bound~\eqref{eq:CCKResult} is larger than dependence $n^{-1/2}$ that appears in multivariate Berry--Esseen bounds, the result~\eqref{eq:CCKResult} does not provide useful information when the probability $\mathbb{P}(U_{n,0}\in A)$ is smaller. This leads naturally to the question of non-uniform version of~\eqref{eq:CCKResult}. In particular, an interesting question is to find quantitative upper bound on 
\begin{equation}\label{eq:RatioLargeDeviation}
 \left|\frac{\mathbb{P}\left( S_{n} \in A^{c} \right)}{\mathbb{P}\left( U_{n,0} \in A^{c} \right)} - 1\right|,
\end{equation}
as a function of $p$ and $n$. In this paper we consider a special class of sets of the form 
\begin{equation}\label{eq:SetsA}
A= \{ x \in \mathbb{R}^{p} : -a \le x(j) \le a \text{ for all }  1\le j \le p \},
\end{equation}
and find an upper bound on~\eqref{eq:RatioLargeDeviation}. Note that sets of the form~\eqref{eq:SetsA} are $l_{\infty}$ balls. Another variant of non-uniform CLT is to consider how the difference $\left|\mathbb{P}\left( \norm{S_n}_{\infty} \le r \right)- \mathbb{P} \left( \norm{U_{n,0}}_{\infty} \le r \right)  \right|$ scales with $r$ for $r \ge 0$. To this end we prove upper bounds on 
\begin{equation}\label{eq:nonuniformintro}
\sup_{r\ge 0}\, r^{m} \left|  \mathbb{P}\left(\norm{S_{n}}_{\infty} \le r \right) - \mathbb{P}\left( \norm{Y}_{\infty} \le r \right)\right|,\quad\mbox{for}\quad m\ge 0.
\end{equation}

The first problem~\eqref{eq:RatioLargeDeviation} is well-studied in the classical multivariate setting under the name ``Cram{\'e}r-type large deviation''. We refer to the encyclopedic work~\cite{MR1171883} for a review of the extensive literature on Cram{\'e}r-type large deviation for sums of independent random variables along with extensions to multivariate random vectors. The classical result for univariate ($p = 1$) iid random variables is of the form
\begin{equation}\label{eq:LargeDeviationClassical}
\frac{\mathbb{P}(S_n > r)}{\mathbb{P}(U_{n,0} > r)} = \exp\left(\frac{Cr^3}{6\Sigma_1^{3/2}n^{1/2}}\right)\left[1 + C\left(\frac{r + 1}{\sqrt{n}}\right)\right],
\end{equation}
for $0 \le r = O(n^{1/6})$ and $C$ a constant depending on the distribution of $X_1$. See Theorem 5.23 (and Section 5.8) of~\cite{MR1353441} for details. The second problem~\eqref{eq:nonuniformintro} is usually referred to as a ``non-uniform CLT''. A result of this kind is also useful in proving convergence of moments. The classical result for univariate iid random variables of type~\eqref{eq:nonuniformintro} is given by
\begin{equation}\label{eq:NonUniformClassical}
|\mathbb{P}(S_n \le r) - \mathbb{P}(U_{n,0} \le r)| \le C(m)(1 + |r|)^{-m}\left(\frac{\mathbb{E}[|X_1|^3]}{\Sigma_1^{3/2}n^{1/2}} + \frac{\mathbb{E}[|X_1|^m]}{\Sigma_1^{m/2}n^{(m-2)/2}}\right),
\end{equation}
for all $r\in\mathbb{R}, m\ge 3$ and for some constant $C(m)$ depending only on $m$. See Theorem 5.15 and (Section 5.5) of~\cite{MR1353441} for details and other results in this direction. Also, see~\cite{MR643968} for multivariate setting. The classical results~\eqref{eq:LargeDeviationClassical} and~\eqref{eq:NonUniformClassical} provide rates that scale like $n^{-1/2}$ (as a function of $n$) in the non-uniform versions of CLT as does the classical Berry--Esseen bound. In lines with the Berry--Esseen type result~\eqref{eq:CCKResult} in high-dimensional setting, we derive rates in large deviation and non-uniform CLT with a scaling of order $n^{-1/6}$ as a function of the sample size $n$.
 
It is well-known \citep{MR823198} that the rate $n^{-1/6}$ is optimal in the central limit theorem for Banach spaces and the space $(\mathbb{R}^p, \norm{\cdot}_{\infty})$ with $p$ diverging behaves like an infinite-dimensional space. In this respect it is of particular interest to look back at the rich literature on the CLTs for Banach space valued random variables. These old and well-known large deviation and non-uniform CLTs for Banach space play a central role in the derivation of ours presented here. The basic setting for these results is as follows: Suppose $X_{1},\ldots , X_{n}$ are i.i.d. random variables taking values in a Banach space $B$ such that $\E[X_{1}]=0$ and $Y$ is a mean zero $B$-valued Gaussian random variable with same the covariance (operator) as $X_{1}$. The problem as before is the study of closeness of the distributions of $\norm{Y}$ and $\norm{S_{n}}$ where $\norm{\cdot}$ is a Banach space norm. Several results on this problem are available in \citet{bentkus2000accuracy}, \citet{paulauskas2012approximation}. {For a historical account of these results, see~\citet[p.142]{paulauskas2012approximation}}. The papers \citet{bentkus1987}, \cite{bentkus1990} and \cite{MR1162240} are of particular interest to us since they provide bounds on~\eqref{eq:RatioLargeDeviation} and~\eqref{eq:nonuniformintro} for Banach space valued random variables. In  \citet{bentkus1987} and \cite{bentkus1990} the problem of the convergence of ratio $\mathbb{P}\left(\norm{S_{n}} > r \right)/ \mathbb{P}\left(\norm{Y} > r\right)$ to $1$ was considered and it was proved that 
\[
\mathbb{P}\left(\norm{S_{n}} > r\right) = \left( 1 + \theta M_{2}(r+1) n^{-{1}/{6}} \right)\mathbb{P}\left(\norm{Y} > r\right),
\]
for $0\le r \le -1+ M_{1}n^{{1}/{6}}$ where $|\theta|\le1$, $M_{1}$ and $M_{2}$ are constants depending on the distribution of $X_{1}$. The non-uniform version of central limit theorem is also available for the Banach spaces from \citet{MR1162240}. Understanding the implication of these results for the high-dimensional case is one contribution of our paper. 
\subsection{Our contributions}\label{subsec:Contributions}
As described in the introduction, we study uniform and non-uniform variants of high dimensional central limit theorem. In the process we prove a sharper version of anti-concentration inequality for centered Gaussians. We assume that $X_{1},\ldots, X_{n} \in \mathbb{R}^{p}$ are independent random vectors with mean $0$ and covariance matrices $\Sigma_i = \mathbb{E}[X_iX_i^{\top}]\in\mathbb{R}^{p\times p}$. Let $Y_1, \ldots, Y_n$ be centered Gaussian random vectors in $\mathbb{R}^{p}$ satisfying $\mathbb{E}[Y_iY_i^{\top}] = \Sigma_i$ for $1\le i\le n$. Our results are intended for the case where $\log(p)$ grows sublinearly in the sample size $n$, although we do not explicitly make an assumption that $\log(p)$ grows with $n$. Define for $m\ge0$
\[
\Delta_n^{(m)} := \sup_{r\ge0}\,r^m\left|\mathbb{P}(\|S_{n}\|_{\infty} \le r) - \mathbb{P}(\|U_{n,0}\|_{\infty} \le r)\right|.
\]
\begin{enumerate}[label=(\roman*)]
\item One of the main ingredients of central limit theorems (both uniform and non-uniform versions) is an anti-concentration inequality which bounds $\mathbb{P}(r - \varepsilon \le \|U_{n,0}\|_{\infty} \le r + \varepsilon)$ over all $r \ge 0$ and $\varepsilon > 0$. We prove that for any $\varepsilon > 0$, $m \ge 0$
\begin{align*}
\sup_{r \ge 0}\;r^m\,\mathbb{P}(r - \varepsilon \le \|U_{n,0}\| \le r + \varepsilon) ~&\le~ \Phi_{AC,m}\varepsilon,
\end{align*}
where $\Phi_{AC,m} = O(\mu^{m+1})$ (assuming $\sigma_{\max}$ and $\sigma_{\min}^{-1}$ are of constant order) and $\mu$ denotes the median of $\|U_{n,0}\|$. Quantities $\sigma_{\max}^2$ and $\sigma_{\min}^2$ are given by the maximum and minimum variances of $U_{n,0}(j), 1\le j\le p$. This provides a refinement of Lemma 3.1 of \cite{bentkus1990} with exact constants. \cite{CCK17} (based on the result of~\cite{Naz03}) prove the above result for $m = 0$ case with $\Phi_{AC,0} = C\sqrt{\log(ep)}$ and since $\mu$ is at most of the order $\sqrt{\log(ep)}$, our result is sharper.
\item We compare the modern proof technique of~\cite{CCK17} and the classical proofs from the Banach space CLT literature~\cite{paulauskas2012approximation},~\cite{bentkus2000accuracy}; see Section~\ref{sec:Outline} for details. Based on this we improve upon the proof of~\cite{CCK17} to get better rates in high-dimensional CLT.  
\item If $X_i$ are sub-Weibull of order $\alpha$, i.e., $\|X_i(j)\|_{\psi_{\alpha}} \le K_p < \infty$ for all $1\le i\le n$, $1\le j\le p$, then
\[
\Delta_n^{(0)} \le \Theta\Phi_{AC,0}\left(\frac{\log^4(ep)}{n}\right)^{1/6} + \Theta_{\alpha}\Phi_{AC,0}K_p\frac{(\log(ep))^{1+1/\alpha}(\log n)^{1/\alpha}}{n^{1/2}}.
\]
for some constant $\Theta$ depending on the distribution of $X_1, \ldots, X_n$ that can be bounded in terms of $K_p$ and $\Theta_{\alpha}$ only depends on $\alpha$. Proposition 2.1 of~\cite{CCK17} for $\alpha = 1$ proves that $\Delta_n^{(0)} \le C(\log^7(pn)/n)^{1/6}$ which requires $\log(pn) = o(n^{1/7})$. In contrast if $\Phi_{AC,0} = O(1)$ then our result only requires $\log(ep) = o(n^{1/4})$ which is the weakest requirement till date.\label{item:ii} 
\item Under the same assumptions in~\ref{item:ii}, we have for $m\ge 0$
\[
\Delta_n^{(m)} \le \Theta\Phi_{AC,m}\left(\frac{\log^4(epn)}{n}\right)^{1/6} + \begin{cases}
0, &\mbox{if } \alpha > 1,\\
\Theta_{\alpha,m}((\log(epn))^{5/4 + 3/\alpha}/n)^{(m+1)/3},&\mbox{if }\alpha\in(0, 1].
\end{cases}
\]
More generally, for any function $\phi(\cdot)$ satisfying $\phi(x + y) \le \phi(x)\phi(y)$ our methods can be used to obtain bounds for
\[
\Delta_n^{(\phi)} := \sup_{r\ge 0}\,\phi(r)\left|\mathbb{P}(\|S_n\| \le r) - \mathbb{P}(\|U_{n,0}\| \le r)\right|.
\]
These are analogues to the results of~\cite{MR1162240}.
\item Finally, we derive a Cram{\'e}r-type large deviation in the high dimensional CLT setting. We assume that $X_1, \ldots, X_n$ are independent and identically distributed and that $\E\left[ \exp\left\{ H \norm{X_1}_{\infty} \right\} \right]< \infty$ for some $H > 0$. Under these assumptions, we prove that
\[
\mathbb{P}\left(\norm{S_{n}}_{\infty} > r \right)= (1+ 2M_{1}(r+1)n^{-1/6}) \mathbb{P}\left(\norm{Y}_{\infty} > r \right)
\]
for any $r\le -1+M_{2}n^{1/6} $.
Here $M_{1}$ and $M_{2}$ are constants depending on the distributions of $X_{1}$ and $Y$  which can be bounded by polynomials of $\log p$, under certain tail assumptions on the coordinates of $X_1$. The proof is motivated by the techniques of \citet{bentkus1987} and is modified for the high dimensional set up. The constants $M_{1}$ and $M_{2}$ are also made explicit in Theorem \ref{thm:largedeviation}. 
\end{enumerate} 
\subsection{Organization of the paper}
The paper is organized as follows. In Section \ref{sec:pre}, we define some useful notation. Section \ref{sec:result} is dedicated for our main results. In this section, we state the anti-concentration inequalities and prove refined uniform as well as non-uniform central limit theorems. In Section \ref{subsec:Application} we present an application of our results to post-selection inference where the anti-concentration constant $\Phi_{AC,0}$ can be of order much smaller than $\sqrt{\log(ep)}$ and to bounding the expectation of suprema of empirical processes over a (possibly infinite) weak VC-major function class. In Section~\ref{sec:LargeDeviation}, we prove a Cram{\'e}r-type large deviation based on the results of~\cite{bentkus1987}. In Section \ref{sec:Outline}, we present an outline of our proof of CLTs with detailed discussion on differences of proofs from other works. Finally, we conclude with a summary and future directions in Section~\ref{sec:Conclusions}. Proofs of all the results are given in the supplementary material.
\section{Preliminaries}\label{sec:pre}
\subsection{Notation and Setting}\label{subsec:Notations}
As discussed earlier, we shall consider independent random vectors $X_{1},\ldots, X_{n} \in \mathbb{R}^{p}$ with mean zero and covariance matrices $\Sigma_i, 1\le i\le n$. Let $Y_1, \ldots, Y_n \in \mathbb{R}^{p}$ denote a Gaussian random vectors with mean and covariance matching that of $X_i, 1\le i\le n$. The $l_{\infty}$ norm on $\mathbb{R}^p$ is denoted by $\norm{\cdot}$. By writing $a\lesssim b$, we mean that for some constant $C$, $a \le Cb$.
We also use the following notation throughout the paper. 
\begin{equation}\label{eq:DefDeltam}
\begin{split}
S_n &:= n^{-1/2}\left(X_1 + X_2 + \ldots + X_n\right),\\
U_{n,k} &:= n^{-1/2}(X_1 + \ldots + X_k + Y_{k+1} + \ldots + Y_n)\quad\mbox{for}\quad 0\le k\le n.
\end{split}
\end{equation}
Note that $U_{n,n} = S_n$ and hence proving the closeness (in distribution) of $U_{n,k}$ and $U_{n,0}$ for all $k$ ensures closeness of $S_n$ and $U_{n,0}$. In this regard, define for $m\ge0$
\begin{equation}\label{def:deltanm}
\delta_{n,m} := \sup_{r\ge 0}\max_{1\le k\le n}\,r^m|\mathbb{P}(\|U_{n,k}\| \le r) - \mathbb{P}(\|U_{n,0}\| \le r)|.
\end{equation}
Finally define for $1\le i\le n$ the signed measure $\zeta_i$ by
\[
\zeta_i(A) := \mathbb{P}(X_i\in A) - \mathbb{P}(Y_i\in A)\quad\mbox{for}\quad A\subseteq\mathbb{R}^p.
\]
Based on this signed measure, $|\zeta_i|$ denotes the variation of measure $\zeta_i$. It is clear that
\begin{equation}\label{eq:zetazero}
\int d\zeta_i(x) = \int x(j)d\zeta_i(x) = \int x(j)x(k)d\zeta_i(x) = 0\quad\mbox{for}\quad 1\le i\le n, 1\le j,k\le p.
\end{equation}
Define the ``weak'' third pseudo-moment as
\begin{equation}\label{eq:ThirdPseudoMoment}
L_n ~:=~ \frac{1}{n}\sum_{i=1}^n \max_{1\le j\le p}\int |x(j)|^3|\zeta_i|(dx), 
\end{equation}
and the truncated ``strong'' second pseudo-moment as
\begin{equation}\label{eq:SecondPsuedoMoment}
M_n(\phi) ~:=~ \frac{1}{n}\sum_{i=1}^n \int \|x\|^2\mathbbm{1}\{\|x\| \ge n^{1/2}\phi/\log(ep)\}|\zeta_i|(dx)\quad\mbox{for}\quad \phi > 0.
\end{equation}
These are called pseudo-moments since they are defined with respect to the variation measure and becomes zero if the distributions of $X_i$'s and $Y_i$'s are the same. Most of the results in classical multivariate setting (of~\cite{sazonov1982accuracy}) and in Banach spaces (of~\cite{MR1162240}) are derived in terms of pseudo-moments. We will present our results also in terms of the pseudo-moments defined above.  

Quantity $M_n(\phi)$ defined above is close to the one defined in~\cite{CCK17} except that they have used marginal third moment instead of second pseudo-moment. This subtlety allows us to derive better rates when random vectors only have $(2+\tau)$-moments. 

Further, set
$\mu_i := \mbox{median}(\|Y_i\|)$ and $\sigma_i^2 := \max_{1\le j\le p}\,\mbox{Var}(Y_i(j))$ for $1\le i\le n$.
Define the weighted ``weak'' third moment as
\begin{equation}\label{eq:ThirdMoment}
\bar{L}_{n,0} := \frac{1}{n}\sum_{i=1}^n (\mu_i + \sigma_i)\max_{1\le j\le p}\int |x(j)|^3|\zeta_i|(dx),
\end{equation}
and for $m > 0$,
\begin{align*}
\bar{L}_{n,m} &:= \frac{1}{n}\sum_{i=1}^n \left[\mu_i^{m + 1} + \sigma_i^{m+1}((m + 1)/e)^{(m+1)/2}\right]\max_{1\le j\le p}\int |x(j)|^3|\zeta_i|(dx).
\end{align*}
It is clear that $\bar{L}_{n,0}$ is at most of order $\sqrt{\log(ep)}$ and $\bar{L}_{n,m}$ is at most of order $(\sqrt{\log(ep)})^{m+1}$ (assuming $\sigma_{i}$'s are all of order 1). More precisely, we have
\begin{equation}\label{eq:BarLnmBound}
\bar{L}_{n,m} ~\le~ \Theta_mL_n\left(\max_{1\le i\le n}\sigma_{i}\right)^{m+1}(\log(ep))^{(m+1)/2},\quad\mbox{for all}\quad m\ge 0.
\end{equation}
Here $\Theta_m$ is a constant depending only on $m\ge0$ scaling like $m^{m/2}.$
\subsection{Tail of Gaussian Processes and Anti-concentration Inequality}
In this section, we prove various inequalities regarding the distribution function of the maximum of a Gaussian process on a compact metric space. These inequalities will lead to the sharper versions of anti-concentration inequalities and is crucial for the large deviation result derived later. The following result is a refinement of Lemma 3.1 in \citet{bentkus1990} where implicit constants were used. In the proof, we make use of a result from~\cite{GINE76}. 
\begin{thm}\label{thm:density}
Let $Y$ be a sample continuous centered Gaussian process on a compact metric space $S$ such that $\sigma_{\min}^2 \le \mathbb{E}\left[Y^2(s)\right] \le \sigma_{\max}^2$ for all $s\in S$. Let $\mu$ denote the median of $\norm{Y}$. Then the following are true:
\begin{enumerate}
\item For all $r\ge 0$,
\[
\mathbb{P}\left(\norm{Y} > r\right) \ge \frac{1}{6}\exp\left(-\frac{r^2}{\sigma_{\max}^2}\right).
\]
\item For all $r, \varepsilon \ge 0$, 
\begin{equation}\label{eq:denbddI}
\mathbb{P}\left(\norm{Y} > r - \varepsilon\right) \le 20\exp\left(\Phi_4(r + 1)\varepsilon\right)\mathbb{P}\left(\norm{Y} > r\right),
\end{equation}
where $\Phi_4$ is given by
\[
\Phi_4 := 1+ \frac{56(\mu + 1.5\sigma_{\max})(\mu + 4.1\sigma_{\max})}{\sigma_{\max}^2\sigma_{\min}^2}+ \frac{32\pi\left(2.6 \sigma_{\min}+ \mu\right)^{2} \left(\sigma_{\min}^{2} + 32 \sigma_{\min} \mu + 12 \mu^{2}\right)}{\sigma_{\min}^{6}}.
\]
\item
 For all $r\ge 0$ and $\varepsilon > 0$,
\begin{equation}\label{eq:ToProve}
\begin{split}
\mathbb{P}\left(r - \varepsilon \le \norm{Y} \le r + \varepsilon\right)\le \Phi_{2}\varepsilon(r + 1)\mathbb{P}\left(\norm{Y} > r - \varepsilon\right),
\end{split}
\end{equation}
where
\[
\Phi_{2}:= \max\left\{\frac{51(\mu + 4.1\sigma_{\max})}{\sigma_{\min}^2}, \frac{32\pi(\mu + 2.6\sigma_{\min})^2}{\sigma_{\min}^4}\right\}
\]
\end{enumerate}
\end{thm}
Observe that the set $\{1,\ldots, p\}$ is compact and discrete. So the result can be used in the high dimension case.
Theorem \ref{thm:density} is dimension-free and the dependence on the ``complexity" of $S$ appears only through the median of $\norm{Y}$. The following anti-concentration inequalities for $\norm{Y}$ can be derived as immediate corollaries of Theorem \ref{thm:density}.
\begin{thm}\label{thm:AntiConcentration}
Fix $m\ge 0$. Under the assumptions of Theorem \ref{thm:density}, we have for all $\varepsilon \ge 0$
\begin{equation}\label{eq:AntiConcentration}
\begin{split}
\sup_{r \ge 0}\;r^m\,\mathbb{P}(r - \varepsilon \le \|Y\| \le r + \varepsilon) ~&\le~ \Phi_{AC,m}\varepsilon,
\end{split}
\end{equation}
where
\begin{align*}
\Phi_{AC,m} &:= \Theta_m\frac{(\mu + \sigma_{\max})^{m+1}\sigma_{\max}^2 + (1 + \sigma_{\max})^2\sigma_{\max}^{m+2}}{\sigma_{\min}^4},
\end{align*}
with $\Theta_m$ representing a constant that depends only on $m$.
\end{thm}
The constant $\Phi_{AC,m}$ obtained in Theorem~\ref{thm:AntiConcentration}, for any $m\ge 0$, is expected to be optimal since if the coordinates of $Y$ are independent then the discussion following Corollary 2.7 of~\cite{GINE76} and Example 2 of~\cite{CCK13} imply that the density of $\|Y\|$ at points of order $\sqrt{\log p}$ is lower bounded in rate by $\sqrt{\log p}$. Hence in this case for any $m\ge 0$ as $\varepsilon \to 0$ the rate is lower bounded by $\mu^{m+1}$.
\begin{rem}{ (Comparison with~\cite{CCK15})} Theorem 3 of \cite{CCK15} implies the anti-concentration for $\norm{Y}$ and provides a dimension-free bound depending on the median of $\norm{Y}$ only under the additional assumption of $\sigma_{\max} = \sigma_{\min}$. For the general case, their bound has an additional term of $\log(1/\varepsilon)$ which makes their bound weaker than that in Theorem~\ref{thm:AntiConcentration}. In terms of the proof technique, we note that the techniques of both works are the same for $r \le 3(\mu + \sigma_{\max})$. For the case $r \ge 3(\mu + \sigma_{\max})$, \cite{CCK15} apply the Gaussian concentration inequality that leads to an extra $\log(1/\varepsilon)$ factor while we use inequality~\eqref{eq:ToProve} that leads to the sharper version. Theorem~\ref{thm:AntiConcentration} is the first result on dimension-free anti-concentration inequality for $\norm{Y}$ and hence our result readily applies to Gaussian processes on (infinite dimensional) compact metric spaces. The results of \cite{Naz03} imply an anti-concentration inequality that explicitly depends on the dimension as $\sqrt{\log p}$. Nazarov's result as proved in~\cite{chernozhukov2017detailed} cannot lead to a rate better than $\sqrt{\log(p)}$ since the covariance structure of $Y$ is completely ignored in the proof.
\end{rem}
\subsection{Smooth Approximation of the Maximum}\label{subsec:Approximation}
Starting point of almost any Berry-Esseen type result is a smoothing inequality. In our scenario, we need a smooth approximation of $\mathbbm{1}\{\|S_n\| \le r\}$ where, recall, $\|\cdot\|$ denotes the $l_{\infty}$-norm. Theorem 1 of~\cite{bentkus1990smooth} provides an infinitely differentiable (in $S_n$) approximation of this indicator with sharp bounds on the derivatives. However to get better rates of convergence in the central limit theorem a certain stability property of the derivatives is needed. This property is exactly the reason why~\cite{CCK17} were able to get better rates than those implied by Banach space CLTs; see~\cite{banerjee2018cram} for implications of Banach space CLTs. The smooth approximation (taken from~\cite{CCK13}), along with the its properties, is summarized in the following result. Define the ``softmax'' function as
\[
F_{\beta}(z) := \frac{1}{\beta}\log\sum_{j=1}^{2p} \exp\left(\beta z(j)\right)\quad\mbox{for}\quad z\in\mathbb{R}^{2p}.
\]
Also define $g_0(t) := 30\mathbbm{1}\{0 \le t \le 1\}\int_t^1 s^2(1 - s)^2ds$.
\begin{lem}\label{lem:SmoothApproximation}
Fix $r\ge0, \varepsilon > 0$ and set $\beta = 2\log(2p)/\varepsilon$. Define the function $\varphi:\mathbb{R}^p\to \mathbb{R}$ as
\[
\varphi(x) = \varphi_{r,\varepsilon}(x) = g_0\left(\frac{2(F_{\beta}(z_x - r\mathbf{1}_{2p}) - \varepsilon/2)}{\varepsilon}\right),
\]
where $z_x = (x^{\top}:-x^{\top})^{\top}$ and $\mathbf{1}_{2p}$ is the vector of $1$'s of dimension $2p$. This function $\varphi(\cdot)$ satisfies the following properties.
\begin{enumerate}
  \item It ``approximates'' the indicator of the $l_{\infty}$-ball, that is,
  \[
  \varphi(x) = \begin{cases}1, &\mbox{if }\|x\| \le r,\\
  0, &\mbox{if }\|x\| > r + \varepsilon,\end{cases}\,\mbox{ or equivalently }\,\mathbbm{1}\{\|x\| \le r\} \le \varphi(x) \le \mathbbm{1}\{\|x\| \le r + \varepsilon\}.
  \]
  \item There exists functions $D_{j}(\cdot), D_{jk}(\cdot)$ and $D_{jkl}(\cdot)$ for $1\le j,k,l\le p$ as well as constant $C_0 > 0$ such that
  \[
  |\partial_j\varphi(x)| \le D_j(x),\quad |\partial_{jk}\varphi(x)| \le D_{jk}(x),\quad |\partial_{jkl}\varphi(x)| \le D_{jkl}(x),
  \]
  where $\partial_{j}, \partial_{jk}, \partial_{jkl}$ denote the partial derivatives of $\varphi$ with respect to the indices in subscript and for all $x\in\mathbb{R}^p$,
  \[
  \sum_{j=1}^p D_{j}(x) \le C_0\varepsilon^{-1},\quad\sum_{j,k=1}^p D_{jk}(x) \le C_0\log(ep)\varepsilon^{-2},\quad \sum_{j,k,l=1}^p D_{jkl}(x) \le C_0\log^2(ep)\varepsilon^{-3}.
  \]
  \item The functions $D_j, D_{jk}, D_{jkl}$ also satisfy a ratio stability property: there exists universal constant $\mathfrak{C} > 0$ such that for all $x,w\in\mathbb{R}^p$,
  \begin{equation}\label{eq:StabilityProp}
  e^{-{\mathfrak{C}\log(ep)\|w\|_{\infty}}/{\varepsilon}} \le \frac{D_{j}(x + w)}{D_j(x)},\;\frac{D_{jk}(x + w)}{D_{jk}(x)},\;\frac{D_{jkl}(x + w)}{D_{jkl}(x)} \le e^{{\mathfrak{C}\log(ep)\|w\|_{\infty}}/{\varepsilon}}.
  \end{equation}
\end{enumerate}
\end{lem}
The smooth approximation result above is the bottleneck in attaining faster rates in CLT than those presented in the present paper. The $\log^4(ep)$ dependence in the uniform and non-uniform CLTs presented in Subsection~\ref{subsec:Contributions} comes only from the $\log(ep)$ factors in the bounds of derivatives and stability property of smooth approximation.

The approximating functions in Lemma~\ref{lem:SmoothApproximation} are constructed to work for any high-dimensional distribution; the difference between $F_{\beta}(z)$ and $\max_j z(j)$ is at most $\log(2p)/\beta$ over all vectors $z$ and this can be smaller if $z$ comes from a specific distribution. This universality can be seen clearly in the construction of~\cite{bentkus1990smooth} who defines the $\varepsilon$ approximation of $\|\cdot\|$ based on
$f_{\varepsilon}(x) = \mathbb{E}\left[\|x + \varepsilon \eta\|\right],$
where $\eta \sim N_p(0, I_p)$. However one can replace $\eta$ by other random vectors taking into account the dependence of $U_{n,0}$. A specific choice that we \emph{conjecture} works is
\[
f_{\varepsilon}(x) = \mathbb{E}[\|x + \varepsilon U_{n,0}\|].
\]
Since $\|\cdot\|$ is Lipschitz, we get that $|\|x\| - f_{\varepsilon}(x)| \le \varepsilon\mathbb{E}[\|U_{n,0}\|] = \varepsilon O(\mu)$. Since $U_{n,0}$ and $S_n$ share the same dependence structure, we only need to bound the derivatives of $f_{\varepsilon}$ at $U_{n,j}$ which would lead to better rates in CLT using the proofs here; see Section~\ref{sec:Outline} for details. 
\section{Main Results}\label{sec:result}
We are now ready to state the main results of this paper. The proofs of all the results in this section are given in the supplementary material. The sketches of these proofs are presented, for readers' convenience, in Section~\ref{sec:Outline}. Recall the notation $\delta_{n,m}$ from ~\eqref{def:deltanm} in Section~\ref{subsec:Notations}. Also recall $M_n(\cdot)$ and quantity $\bar{L}_{n,0}$ are defined in~\eqref{eq:SecondPsuedoMoment} and~\eqref{eq:ThirdMoment}, respectively. The quantity $L_n$ (in~\eqref{eq:ThirdPseudoMoment}) denotes the ``weak'' third pseudo-moment and if $L_n = 0$ then $\delta_{n,0} = \delta_{n,m} = 0$ for any $m \ge 0$. For this reason, we assume $L_n > 0$. Let $\Phi_{AC,m}$ denote the anti-concentration constant in Theorem~\ref{thm:AntiConcentration} for random vector $U_{n,0}$.
\subsection{Uniform CLT}\label{sec:UniformCLT}
\begin{thm}\label{thm:UniformCLT}
For independent random vectors $X_1, \ldots, X_n\in\mathbb{R}^p$, we have
\begin{align*}
\delta_{n,0} ~&\le~ 4\Phi_{AC,0}\varepsilon_n + \frac{2C_0\log(ep)M_n(\varepsilon_n)}{\varepsilon_n^{2}} + \frac{\log^{1/3}(ep)\bar{L}_{n,0}}{n^{1/3}L_n^{4/3}(2e^{5\mathfrak{C}}C_0)^{1/3}},
\end{align*}
where $\varepsilon_n = (2e^{2\mathfrak{C}}C_0\log^2(ep)L_n)^{1/3}/n^{1/6}$.
\end{thm}

Theorem~\ref{thm:UniformCLT} is qualitatively the same as Theorem 2.1 of~\cite{CCK17} for $l_{\infty}$-balls. More importantly, note that if $\Phi_{AC,0} \asymp \sqrt{\log(ep)}$ then Theorem~\ref{thm:UniformCLT} has a dominating term of order $(\log^7(ep)/n)^{1/6}$ and hence the result above is as good as Theorem 2.1 of~\cite{CCK17}. The last term in the bound of $\delta_{n,0}$ is of lower order compared to the first term. From inequality~\eqref{eq:BarLnmBound}, we obtain that
\[
\frac{\log^{1/3}(ep)\bar{L}_{n,0}}{n^{1/3}L_n^{4/3}(2e^{5\mathfrak{C}}C_0)^{1/3}} ~=~ O\left(\frac{(\log(ep))^{5/6}}{(nL_n)^{1/3}}\right) = O\left(\frac{(\log(ep))^{5/2}}{nL_n}\right)^{1/3},
\]
which is dominated by the first term which is at least of order $(\log^4(ep)/n)^{1/6}$. Further the quantity $M_n(\varepsilon_n)$ (defined in~\eqref{eq:SecondPsuedoMoment}) is exactly what appears in the classic Lindeberg condition. 
\paragraph{Rates under $(2 + \tau)$-moments:} 
Recall that $M_n(\phi)$ only involves truncated second moment instead of third moment used in~\cite{CCK17}. This subtle difference allows for deriving better rates when $\|X_i\|$ have $(2 + \tau)$-moments. In this case the choice of $\varepsilon_n$ in Theorem~\ref{thm:UniformCLT} is not the right choice since $M_n(\varepsilon_n)/\varepsilon_n^2$ does not converge to zero; we need to choose $\varepsilon_n$ larger. Even though our results involve third ``moments'' $L_n, \bar{L}_{n,0}$, from the proof (in particular~\eqref{eq:SomeToAllIntegral} in the supplement) it can be seen that the integral in the definition of $L_n, \bar{L}_{n,0}$ can be changed to integral over a truncated set; see Appendix~\ref{app:TauLe1} and Step E in Section~\ref{sec:Outline} for details. Very recently~\cite{sun2019gaussian} considered CLT under purely $(2 + \tau)$-moments using Lindeberg method but no explicit rates were provided. In the following we provide details for $\tau \ge 1$ and the calculations for $\tau < 1$ are provided in Appendix~\ref{app:TauLe1}. Set
\[
\nu_{2 + \tau}^{2 + \tau} := \frac{1}{n}\sum_{i=1}^n \int \|x\|^{2 + \tau}|\zeta_i|(dx)\quad\Rightarrow\quad M_n(\varepsilon) \le \frac{\nu_{2+\tau}^{2 + \tau}}{(n^{1/2}\varepsilon/\log(ep))^{\tau}} = \nu_{2+\tau}^{2 + \tau}\left(\frac{\log(ep)}{n^{1/2}\varepsilon}\right)^{\tau}.
\]
The proof of Theorem~\ref{thm:UniformCLT} actually proves a bound that holds for all $\varepsilon > 0$ and we now choose $\varepsilon = \varepsilon_n := \max\{(2e^{2\mathfrak{C}}C_0\log^2(ep)L_n)^{1/3}/n^{1/6},\,r_n\nu_{2+\tau}((\log(ep))^{1 + \tau}/n^{\tau/2})^{1/(2 + \tau)}\}$ for some $r_n \ge 1$. This choice implies that
\[
\frac{e^{2\mathfrak{C}}C_0\log^2(ep)L_n}{n^{1/2}\varepsilon_n^{3}} \le \frac{1}{2}\quad\mbox{and}\quad \frac{C_0\log(ep)M_n(\varepsilon_n)}{\varepsilon_n^{2}} \le \frac{C_0}{r_n^{2 + \tau}}.
\] 
Following the proof of Theorem~\ref{thm:UniformCLT}, for any $r_n\ge1$, we get
\[
\delta_{n,0} ~\lesssim~ \frac{1}{r_n^{2 + \tau}}
+ \Phi_{AC,0}\frac{(\log^2(ep)L_n)^{1/3}}{n^{1/6}} + \Phi_{AC,0}\frac{r_n\nu_{2+\tau}(\log(ep))^{(1 + \tau)/(2 + \tau)}}{n^{\tau/2(2 + \tau)}}.
\]
Here $a\wedge b = \min\{a,b\}$ and $a\vee b = \max\{a,b\}$. Minimizing over $r_n \ge 1$, we get
\begin{equation}\label{eq:Thm31Implication}
\delta_{n,0} \lesssim 
\begin{cases}
\left[\Phi_{AC,0}L_{n,\tau}^{1/(2+\tau)} + \left(\Phi_{AC,0}\nu_{2+\tau}\right)^{(2+\tau)/(3+\tau)}\right]\frac{(\log(ep))^{(\tau+1)/(\tau+2)}}{n^{\tau/(6+2\tau)}}, &\mbox{if }\tau < 1,\\
&\vspace{-0.2in}\\
\Phi_{AC,0}\frac{(\log^2(ep)L_n)^{1/3}}{n^{1/6}} + (\Phi_{AC,0}\nu_{2+\tau})^{(2 + \tau)/(3 + \tau)}\frac{(\log(ep))^{(\tau+1)/(\tau+3)}}{n^{\tau/(6+2\tau)}}, &\mbox{if }\tau \ge 1,
\end{cases}
\end{equation}
where $L_{n,\tau} = n^{-1}\sum_{j=1}^n \max_{1\le j_1\le p}\int |x(j_1)|^{2+\tau}|\zeta_j|(dx)$. As mentioned before the proof for $\tau < 1$ is given in Appendix~\ref{app:TauLe1}. The dependence on sample size obtained above is better than the one obtained in Proposition 2.1 of~\cite{CCK17} for all $\tau > 0$. In particular for $\tau = 1$, we got the dependence $n^{-1/8}$ whereas~\cite{CCK17} obtained $n^{-1/9}$ on the sample size. 
We note here that following the proof of Theorem 2.1 of~\cite{bentkus2000accuracy} it is possible to get $n^{-1/6}$ dependence on $n$ whenever $\tau \ge 1$ as shown in the following proposition. This is similar to the case of multivariate CLT where the rate of convergence scales as $n^{-1/2}$ whenever the random vectors have more than three moments; i.e., after some number of moments the rate stabilizes. The following result is proved based on Theorem 2.1 of~\cite{bentkus2000accuracy} and does not use the stability property~\eqref{eq:StabilityProp}.
\begin{prop}\label{prop:BentkusProof}
For i.i.d. random vectors $X_1,\ldots,X_n$,
\[
\delta_{n,0} \le \frac{1}{2^{n-1}} + 8\Phi_{AC,0}\nu_3\frac{\left(C_0\log^2(ep)\right)^{1/3}}{n^{1/6}}~.
\]
\end{prop}
It is clear that the rate obtained in Proposition~\ref{prop:BentkusProof} is sharper than the one obtained in~\eqref{eq:Thm31Implication} for $1 \le \tau \le 1.5$. Note that for $\tau = 1.5$, $\tau/(6 + 2\tau) = 1/6$ and~\eqref{eq:Thm31Implication} leads to $n^{-1/6}$ rate. A notable difference between the rates is that while Theorem~\ref{thm:UniformCLT} decouples the main rate $\Phi_{AC,0}(\log^2(ep))^{1/3}/n^{1/6}$ and the term depending on moments $\nu_{2+\tau}$ (but somehow sub-optimally), Proposition~\ref{prop:BentkusProof} puts these two terms together which implies that Proposition~\ref{prop:BentkusProof} does not lead to the ``right'' requirement between $\log(p)$ and $n$ in case random vectors have exponential tails. This discussion raises an important point: What is the ``right way'' of proving rates in the high-dimensional CLT? The following result combines the proof techniques of Theorem~\ref{thm:UniformCLT} and Proposition~\ref{prop:BentkusProof} in the ``right way'' so that the rate scales like $n^{-1/6}$ for all $\tau \ge 1$ (as Proposition~\ref{prop:BentkusProof}) and becomes the same as~\eqref{eq:Thm31Implication} for $\tau\to\infty$.
\begin{thm}\label{thm:OptimalCLT} For i.i.d. random vectors $X_1,\ldots,X_n$, and for any $\tau \ge 1$,
\[
\delta_{n,0} ~\lesssim~ \frac{1}{2^{n}} + \Phi_{AC,0}\left(\frac{L_n^2\log^4(ep)}{n}\right)^{1/6} + \Phi_{AC,0}\nu_{2+\tau}\frac{(\log(ep))^{(\tau+1)/(\tau+2)}}{n^{\tau/(4 + 2\tau)}},
\]
where $\lesssim$ hides only universal constants.
\end{thm}
The assumption of i.i.d. random vectors in Proposition~\ref{prop:BentkusProof} and Theorem~\ref{thm:OptimalCLT} can be relaxed to observations $X_1,\ldots,X_n$ having the same covariance matrix. It is clear that Theorem~\ref{thm:OptimalCLT} leads to better rates than~\eqref{eq:Thm31Implication} for all $\tau \ge 1$ since $\tau/(6+2\tau) < \tau/(4+2\tau)$. 
\subsection{Non-uniform CLT}\label{sec:NonUniformCLT}
As an extension of Theorem~\ref{thm:UniformCLT}, we have the following non-uniform version of CLT. Since the statement is exact with explicit constants, it is cumbersome. 
\begin{thm}\label{thm:NonUniformPolyCLT}
For independent $\mathbb{R}^p$ random vectors $X_1,\ldots,X_n$, for any $m, r_{n,m} \in (0,\infty)$, 
\begin{align*}
\delta_{n,m} &\le 2^{2m^2/3 + 8m/3 +1}\varepsilon_{n}^{m+1}\Phi_{AC,0} + (2^{2m/3+1} + 2)\Phi_{AC,m}\varepsilon_n + {2C_0 \log(ep)r_{n,m}^m \varepsilon_n^{-2}M_n(2^{2m/3}\varepsilon_n)}\\ 
&\quad+ \frac{r_{n,m}^m\bar{L}_{n,m}}{L_n 2^{m}e^{\mathfrak{C}}}\left(\frac{\log(ep)}{n(2^{2m+1}e^{2\mathfrak{C}}C_0L_n)}\right)^{(m+1)/3} + \sup_{r\ge r_{n,m}}\,\max_{0\le k\le n} r^m\mathbb{P}(\|U_{n,k}\| \ge r). 
\end{align*}
Here $\varepsilon_n$ is same as the one defined in Theorem~\ref{thm:UniformCLT}. 
\end{thm}
Theorem~\ref{thm:NonUniformPolyCLT} does not get the correct dependence on the sample size $n$ when random vectors $X_1,\ldots,X_n$ only have $(2+\tau)$-moments as the calculation for Theorem~\ref{thm:UniformCLT} show. We now present an improvement for i.i.d. random vectors which can be seen as a non-uniform extension of Theorem~\ref{thm:OptimalCLT}.
\begin{thm}\label{thm:OptimalNonUniformCLT}
Fix $m\ge 1$ and $\tau \ge m$. If $X_1,\ldots,X_n, n\ge 1$ are i.i.d. random vectors, then
\[
\delta_{n,m} \le 2^{m/2}\left(\frac{\nu_m}{2^{n/2}}\right)^m + 2^{2 + 2m}\varepsilon_n^{m} + 2.4\Phi_{AC,m}\varepsilon_n,
\]
where
\[
\varepsilon_n := \max\left\{\frac{(2^{2 + 3m/2}C_0e^{\mathfrak{C}}L_n\log^2(ep))^{1/3}}{n^{1/6}},\,\frac{\nu_{2+\tau}\left(2^{3+5m/2}C_0\log^{\tau+1}(ep)\right)^{1/(2+\tau)}}{n^{\tau/(4 + 2\tau)}}\right\}.
\]
\end{thm}
Theorem~\ref{thm:OptimalNonUniformCLT} always leads to $n^{-1/6}$ dependence on the sample size unlike Theorem~\ref{thm:NonUniformPolyCLT}. 
\paragraph{Comments on the Proof Technique.} The proofs of uniform and non-uniform CLTs in Banach space literature are based on Lindeberg method and smooth approximation; see, for example,~\cite{MR1162240} and~\cite{bentkus2000accuracy}. Motivated by the proof technique of~\cite{CCK17} who introduced the stability property, we combine Lindeberg method with a minor twist from the proof of~\cite{CCK17} to prove all the above results; see Section~\ref{sec:Outline} for a detailed outline of how our proofs differ from the ones in the above mentioned references.

Recently,~\citet[Proposition 2.3]{koike2019high} proved for sub-Gaussian random variables a rate of $(\log^6(ep)/n)^{1/6}$ for $\delta_{n,0}$ which could also be weaker than our result depending on $\Phi_{AC,0}$. It maybe possible that the methods in that paper could sharpen our result. We refrained from using their proof techniques for simplicity.
\subsection{Corollaries for Sub-Weibull Random Vectors}\label{sec:CorollariesWeibull}
In this section, we provide simplified results under the assumption that the coordinates of random vectors $X_1, \ldots, X_n$ are sub-Weibull. For a simplification of uniform CLTs, we choose a suitable $\tau$ in~\eqref{eq:Thm31Implication} and Theorem~\ref{thm:OptimalCLT}. For a simplification of Theorem~\ref{thm:NonUniformPolyCLT}, we simplify $M_n(\varepsilon_n)$ and $\bar{L}_{n,m}$ under sub-Weibull tails. Similar simplification can be done from Theorem~\ref{thm:UniformCLT} but this leads to sub-optimal dependence on $\log(ep)$ in the second order term. 

Recall that $M_n(\varepsilon_n)$ and $\bar{L}_{n,m}$ are defined in terms of the variation measures $\zeta_j$ which we bound using the sum of the measures for simplicity. To state the simplified results, we introduce Orlicz norms.
\begin{defi}
Let $X$ be a real-valued random variable and $\psi : [0,\infty) \mapsto [0,\infty)$ be a non-decreasing function with $\psi(0) = 0$. Then, we define
$$\norm{X}_\psi = \inf \left\{c > 0 :~\mathbb{E}\psi\left({|X|}/{c}\right) \leq 1\right\},$$ where the infimum over the empty set is taken to be $\infty$.
\end{defi}
Usually the definition of Orlicz ``norm'', $\|\cdot\|_{\psi}$, includes convexity assumption of $\psi$ which we did not include since we also work with non-convex $\psi$ below. It follows from Jensen's inequality, that when $\psi$ is a non-decreasing, convex function, $\norm{\cdot}_\psi$ is a norm on the set of random variables $X$ for which $\norm{X}_\psi < \infty$. The commonly used Orlicz norms are derived from $\psi_\alpha(x) := \exp(x^\alpha) - 1,$ for $\alpha \geq 1$, which are obviously increasing and convex. For $0 < \alpha < 1$, $\psi_\alpha$ is not convex, and $\norm{X}_{\psi_\alpha}$ is not a norm, but a quasinorm.
A random variable $X$ is called sub-exponential if $\norm{X}_{\psi_1} < \infty$, and a random variable $X$  is called sub-Gaussian if $\norm{X}_{\psi_2} < \infty$. 

Recall that $X(j)$ represents the $j$-th coordinate of $X\in\mathbb{R}^p$ and set
\[
\sigma_{\max}^2 := \max_{1\le j \le p} \mbox{Var}(U_{n,0}(j)),\quad\mbox{and}\quad\sigma_{\min}^2 := \min_{1\le j \le p} \mbox{Var}(U_{n,0}(j)).
\]
Throughout the following corollaries, $\Theta$ stands for a universal constant that does not depend on $p, n$ or any of the other distributional properties. $\Theta$ with subscripts (such as $\Theta_{\alpha}$) represents constants that only depend on those subscripts.
\begin{cor}[Uniform CLT]\label{cor:UniformCLT}
Suppose there exists a constant $1\le K_p < \infty$ such that
\begin{equation}\label{eq:OrliczNorm}
\max_{1\le i\le n}\max_{1\le j\le p}\,\|X_i(j)\|_{\psi_{\alpha}} ~\le~ K_p,
\end{equation}
for some $0 < \alpha \le 2$. 
If $\log(n^{3/2}/(K_p\Theta_{\alpha}\Phi_{AC,0}(\log(ep))^{2+1/\alpha})) \ge 4$ then
\begin{equation}\label{eq:BadRateImplication}
\delta_{n,0} \le \Theta\Phi_{AC,0}\frac{(L_n\log^2(ep))^{1/3}}{n^{1/6}} + \Theta_{\alpha}K_p\Phi_{AC,0}\frac{(\log(ep))^{1 + 1/\alpha}}{n^{1/2}}\log^{1/\alpha}\left(\frac{\Theta_{\alpha}\Phi_{AC,0}^{-1}n^{3/2}}{(\log(ep))^{2+1/\alpha}}\right).
	\end{equation}
Instead if $\alpha\log(n/\log(ep)) \ge 3$ and $X_1,\ldots,X_n$ share the same covariance matrix, then
\begin{equation}\label{eq:TheoremOptImplicate}
\delta_{n,0} \le \Theta\Phi_{AC,0}\frac{(L_n\log^2(ep))^{1/3}}{n^{1/6}} + \Theta_{\alpha}K_p\Phi_{AC,0}\frac{(\log(ep))^{1 + 1/\alpha}}{n^{1/2}}\log^{1/\alpha}\left(\frac{n}{\log(ep)}\right).
\end{equation}
\end{cor}
\begin{proof} 
Note that assumption~\eqref{eq:OrliczNorm} implies that $\nu_{2+\tau} \le K_p\Theta_{\alpha}(2+\tau)^{1/\alpha}(\log(ep))^{1/\alpha}$ for all $\tau \ge 1$. Then~\eqref{eq:BadRateImplication} follows from~\eqref{eq:Thm31Implication} by taking $\tau+3 = \log(n^{3/2}/(K_p\Theta_{\alpha}\Phi_{AC,0}(\log(ep))^{2+1/\alpha}))$ which is possible for $\tau \ge 1$ since the right hand side is assumed to be larger than $4$. Further~\eqref{eq:TheoremOptImplicate} follows from Theorem~\ref{thm:OptimalCLT} by taking $\tau + 2 = \alpha\log(n/\log(ep))$ (which is possible for $\tau\ge1$ since $\alpha\log(n/\log(ep)) \ge 3$).
\end{proof} 
Simplifying a little further, the bounds on $\delta_{n,0}$ can be written as
\[
\delta_{n,0} ~\lesssim~ \Phi_{AC,0}\frac{(L_n\log^2(ep))^{1/3}}{n^{1/6}} + K_p\Phi_{AC,0}\frac{(\log(ep))^{1 + 1/\alpha}(\log n)^{1/\alpha}}{n^{1/2}}.
\]
The following corollary is obtained by controlling $M_n(\varepsilon_n)$ and $\bar{L}_{n,m}$ in Theorem~\ref{thm:NonUniformPolyCLT} for sub-Weibull random vectors. By choosing an appropriate of $\tau$ in Theorem~\ref{thm:OptimalNonUniformCLT} we get a much simpler form under the assumption that $X_1,\ldots,X_n$ share the same covariance structure.
\begin{cor}[Non-uniform CLT]\label{cor:NonUniformPolyCLT}
Fix $m\ge 0$. If $n ~\ge~ \Theta K_p^3L_n^{-1}(2e\log(ep))^{1 + 3/\alpha}$, then we have (i) for $1 < \alpha \le 2$,
\begin{align*}
\delta_{n,m} &\le \Theta_m\Phi_{AC,m}\left(\frac{L_n^2\log^4(ep)}{n}\right)^{1/6} + \Theta_{\alpha,m}\left(\frac{K_p^{(2m + 1)\alpha}\log^4(epn)}{nL_n^{(m+1)\alpha/3}}\right)^{1/(\alpha - 1)}\\ &\qquad+ \Theta_mK_p^m\sigma_{\max}^{m+1}\left(\frac{\log^4(epn)}{nL_n}\right)^{(m+1)/3} + \Theta_m\frac{K_p^{m+2}}{\sigma_{\max}^2n^{2/3}},
\end{align*}
and (ii) for $0 < \alpha \le 1$,
\begin{align*}
\delta_{n,m} &\le \Theta_m\Phi_{AC,m}\left(\frac{L_n^2\log^4(ep)}{n}\right)^{1/6} + \Theta_{\alpha,m}\frac{K_p^{3 + m}}{L_n}\left(\frac{K_p^3(\log(epn))^{5/4 + 3/\alpha}}{nL_n}\right)^{12/\alpha + 2m}\\ &\qquad+ \Theta_mK_p^m\sigma_{\max}^{m+1}\left(\frac{(\log(epn))^{1 + 3/\alpha}}{nL_n}\right)^{(m+1)/3} + \Theta_m\frac{K_p^{m+2}}{\sigma_{\max}^2n^{2/3}}.
\end{align*}
Instead if $\alpha\log(2^{3+5m/2}n/\log(ep)) \ge m+2$ and $X_1,\ldots,X_n$ share the same covariance matrix, then
\begin{equation}\label{eq:NonUniformOptimal}
\delta_{n,m} \le 2^{m/2}\left(\frac{\nu_m}{2^{n/2}}\right)^m + 2^{2+2m}\varepsilon_n^m + 2.4\Phi_{AC,m}\varepsilon_n,
\end{equation}
where
\[
\varepsilon_n := \max\left\{\frac{(2^{2+3m/2}C_0e^{\mathfrak{C}}L_n\log^2(ep))^{1/3}}{n^{1/6}}, \Theta_{\alpha,m}(e\alpha)^{1/\alpha}\frac{(\log(ep))^{1 + 1/\alpha}}{n^{1/2}}\log^{1/\alpha}\left(\frac{2C_0n}{\log(ep)}\right)\right\}.
\]
\end{cor}
\noindent Inequality~\eqref{eq:NonUniformOptimal} follows by taking $2 + \tau = \alpha\log(2^{3+5m/2}C_0n/\log(ep))$ in $\varepsilon_n$ of Theorem~\ref{thm:OptimalNonUniformCLT}.
\begin{rem}\,(Even more simplified rates)
The bounds in Corollary~\ref{cor:NonUniformPolyCLT} are finite sample and show explicit dependence on $L_n, K_p, \sigma_{\max}$ and other distributionals constants. If $\max\{K_p, \sigma_{\max}, L_n^{-1}\} = O(1)$ then the bounds in Corollary~\ref{cor:NonUniformPolyCLT} can simply be written as
\[
\delta_{n,m} ~\lesssim~ \Phi_{AC,m}\frac{(L_n\log^2(ep))^{1/3}}{n^{1/6}} + \begin{cases} 0,&\mbox{if }\alpha > 1,\\
((\log(ep))^{5/4 + 3/\alpha}/n)^{(m+1)/3},&\mbox{if }\alpha\in(0, 1].\end{cases}
\]
The implication from~\eqref{eq:NonUniformOptimal} can be simply written as
\[
\delta_{n,m} \lesssim \left(\frac{\nu_m}{2^{n/2}}\right)^m + \Phi_{AC,m}\frac{(L_n\log^2(ep))^{1/3}}{n^{1/6}} + \Phi_{AC,m}\frac{(\log(ep))^{1 + 1/\alpha}(\log n)^{1/\alpha}}{n^{1/2}}.
\]
\end{rem}
\begin{rem}{ (Convergence of Moments)}\label{Rem:MomentConvergence}
Theorems~\ref{thm:UniformCLT} and~\ref{thm:NonUniformPolyCLT} are useful in proving convergence of $m$-th moment of $\norm{S_{n}}$ to that $\norm{U_{n,0}}$ at an $n^{-1/6}$ rate (up to factors depending on $\log(ep)$). To prove an explicit bound, note that for $m \ge 1$:
\begin{align*}
\left|\E\left[ \norm{S_{n}}^{m} \right]- \E\left[ \norm{U_{n,0}}^{m} \right]\right| &= \left|\int_{0}^{\infty} m r^{m-1} \left(\mathbb{P}\left(\norm{S_{n}} \ge r \right)- \mathbb{P}\left( \norm{U_{n,0}} \ge r\right)\right)dr \right|\nonumber\\
& \le m\int_{0}^{1} \delta_{n,0} dr + \int_{1}^{\infty} \frac{m}{r^{1+\beta}} \delta_{n,m+\beta} dr
\le m\delta_{n,0}(r) + \frac{m}{\beta} \delta_{n,m+\beta},
\end{align*}
for any $\beta \ge 0$. The second term in the right hand side above can be bounded using Theorem~\ref{thm:NonUniformPolyCLT} or~\ref{thm:OptimalNonUniformCLT}. Taking the main terms in the bounds for $\delta_{n,0}, \delta_{n,m+\beta}$ and using $\Phi_{AC,m}\lesssim \mu^{m+1}$ from Theorem~\ref{thm:AntiConcentration}, we can write the bound on moment difference as
\[
\left|\E\left[ \norm{S_{n}}^{m} \right]- \E\left[ \norm{U_{n,0}}^{m} \right]\right| \lesssim m\mu\frac{(L_n\log^2(ep))^{1/3}}{n^{1/6}} + \frac{m}{\beta}\mu^{m+\beta+1}\frac{(L_n\log^2(ep))^{1/3}}{n^{1/6}}.
\] 
Note that we assumed here $\sigma_{\min}^{-1} = O(1)$. Choosing $\beta = 1/\log(\mu)$, we get
\begin{align*}
\left|\E\left[ \norm{S_{n}}^{m} \right]- \E\left[ \norm{U_{n,0}}^{m} \right]\right| &\lesssim m\mu\frac{(L_n\log^2(ep))^{1/3}}{n^{1/6}} + {m}\log(\mu)\mu^{m+1}\frac{(L_n\log^2(ep))^{1/3}}{n^{1/6}}.
\end{align*}
Since the median and $(\mathbb{E}[\|U_{n,0}\|^{\ell}])^{1/\ell}, \ell \ge 1$ are of the same order, we get that for $m\ge1$
\[
\mathbb{E}\left[\|S_n\|^m\right] ~=~ \left(1 + O\left(\frac{\mu\log(\mu)(L_n\log^2(ep))^{1/3}}{n^{1/6}}\right)\right)\mathbb{E}\left[\|U_{n,0}\|^m\right]. 
\]
Hence our results imply that the moments of $\|S_n\|$ match the moments of $\|U_{n,0}\|$ up to a lower order term if $\mu\log(\mu)(L_n\log^2(ep))^{1/3} = o(n^{1/6})$.
\end{rem}
\section{Applications}\label{subsec:Application}
In this section, we present two applications. The first is to post-selection inference and shows the impact of dimension-free anti-concentration constant. The second is to maximal inequalities for empirical processes that shows the importance of non-uniform CLT.  
\subsection{Many Approximate Means and Post-selection Inference}\label{subsec:PoSI}
In this section, we use the ``many approximate means'' (MAM) framework of~\cite{belloni2018high} for post-selection inference (PoSI). In this PoSI problem, we show scenarios where $\Phi_{AC,0}$ (or $\mu$) grows almost \emph{like} a constant. 

The MAM framework is as follows. Suppose we have a parameter $\theta_0 = (\theta_0(1), \ldots, \theta_0(p))^{\top}$ and an estimator $\hat{\theta} = (\hat{\theta}(1), \ldots, \hat{\theta}(p))^{\top}$ of parameter $\theta_0$ that has an approximate linear form:
\begin{equation}\label{eq:InfluExpansion}
n^{1/2}(\hat{\theta} - \theta_0) ~=~ \frac{1}{\sqrt{n}}\sum_{i=1}^n \psi(Z_i) + R_n,
\end{equation}
where $\psi(\cdot) = (\psi_1(\cdot), \ldots, \psi_p(\cdot))^{\top}\in\mathbb{R}^p$ and $R_n = (R_n(1), \ldots, R_n(p))^{\top}\in\mathbb{R}^p$. Here $Z_1, \ldots, Z_n$ are independent random variables based on which $\hat{\theta}$ is constructed. The function $\psi_j(\cdot), 1\le j\le p$ represents the influence function for estimator $\hat{\theta}(j)$. An estimator $\hat{\theta}$ satisfying~\eqref{eq:InfluExpansion} is called asymptotically linear and most $M$-estimators (commonly used) satisfy this assumption. Based on the linear approximation~\eqref{eq:InfluExpansion} and the anti-concentration result Theorem~\ref{thm:AntiConcentration}, we have the following result. Define
\[
\bar{\Delta}_n(r) := \left|\mathbb{P}(\sqrt{n}\,\|\hat{\theta} - \theta_0\| \le r) - \mathbb{P}(\|Y^{\psi}\| \le r)\right|,
\]
where
$Y^{\psi} ~\sim~ N_p\left(0, n^{-1}\sum_{i=1}^n \mathbb{E}[\psi(Z_i)\psi^{\top}(Z_i)]\right).$
Let $\Phi_{AC,m}^{\psi}$ denote the anti-concentration constant from Theorem~\ref{thm:AntiConcentration} for $Y^{\psi}$. Set
$S_n^{\psi} := n^{-1/2}\sum_{i=1}^n \psi(Z_i).$
\begin{prop}\label{prop:MAMNonUniform}
For any $\delta > 0$ and $m\ge 0$, we have
\[
r^m\bar{\Delta}_n(r) \le 2\left[(5/4)^{m}\Delta_{n,m}^{\psi} + 5^m\Phi_{AC,0}^{\psi}\delta^{m+1}\right] + \Phi_{AC,m}^{\psi}\delta + r^m\mathbb{P}(\|R_n\| > \delta),
\]
where
$\Delta_{n,m}^{\psi} := \sup_{r\ge0}\,r^m|\mathbb{P}(\|S_n^{\psi}\| \le r) - \mathbb{P}(\|Y^{\psi}\| \le r)|.$
\end{prop}
This result is a slight generalization of Theorem 2.1 of~\cite{belloni2018high} for $l_{\infty}$-balls. To derive a proper non-uniform CLT result from Proposition~\ref{prop:MAMNonUniform}, one needs to choose $\delta$ depending on $r$ and apriori bound moments of $\|R_n\|$. For the case $m = 0$ (which we focus on from now), the \emph{proof} of Proposition~\ref{prop:MAMNonUniform} implies
\[
\bar{\Delta}_n(r) \le \Delta_{n,m}^{\psi} + \Phi_{AC,0}^{\psi}\delta + \mathbb{P}(\|R_n\| > \delta).
\]
Hence $\bar{\Delta}_n(r)$ converges to zero if $\|R_n\| = o_p(r_n)$ such that $r_n\Phi_{AC,0} = o(1)$ as $n, p \to \infty$.

We now describe the framework of post-selection inference. Suppose $Z_i := (X_i^{\top}, Y_i)^{\top}\in\mathbb{R}^d\times\mathbb{R}$ for $1\le i\le n$ represent regression data with $d$-dimensional covariates $X_i$. For any subset $M\subseteq\{1,2,\ldots,d\}$, let $X_{i,M}$ denote a subvector of $X_i$ with indices in $M$. Based on a loss function $\ell(\cdot, \cdot)$ define the regression ``slope'' estimator $\hat{\beta}_M$ as
\[
\hat{\beta}_M := \argmin_{\theta\in\mathbb{R}^{|M|}}\,\frac{1}{n}\sum_{i=1}^n \ell(X_{i,M}^{\top}\theta, Y_i).
\]
The target for the estimator $\hat{\beta}_M$ is given by
$\beta_M := \argmin_{\theta\in\mathbb{R}^{|M|}}\,n^{-1}\sum_{i=1}^n \mathbb{E}[\ell(X_{i,M}^{\top}\theta, Y_i)].$
Some examples of loss functions are related to linear regression $\ell(u, v) = (u - v)^2/2$, logistic regression $\ell(u, v) = uv - \log(1 + e^u)$, Poisson regression $\ell(u, v) = uv - \exp(u)$. 

As is often done in practical data analysis, suppose we choose a model $\hat{M}$ based on the data $\{Z_i:\,1\le i\le n\}$. The problem of post-selection inference refers to the statistical inference for the (random) target $\beta_{\hat{M}}$. In particular, PoSI problem refers to construction of confidence regions $\{\hat{\mathcal{R}}_M:\,M\subseteq\{1,2,\ldots,d\}\}$ (depending on $\alpha\in[0, 1]$) such that
\begin{equation}\label{eq:PoSI}
\liminf_{n\to\infty}\,\mathbb{P}\left(\beta_{\hat{M}}\in\hat{\mathcal{R}}_{\hat{M}}\right)~\ge~ 1 - \alpha,
\end{equation}
holds for \underline{any} randomly selected model $\hat{M}$. It was proved in~\citet[Theorem 3.1]{kuchibhotla2018valid} that the PoSI problem~\eqref{eq:PoSI} is equivalent to the simultaneous inference problem
\begin{equation}\label{eq:Simultaneous}
\liminf_{n\to\infty}\,\mathbb{P}\left(\bigcap_{M}\,\left\{\beta_M \in \hat{\mathcal{R}}_M\right\}\right) ~\ge~ 1 - \alpha,
\end{equation}
where the intersection is taken over all $M\subseteq\{1,2,\ldots,d\}$. A straightforward construction of such simultaneous confidence regions can be based on finding quantiles of the statistic
\begin{equation}\label{eq:maxt}
\max_{M}\max_{1 \le j\le |M|}\,\left|\frac{\sqrt{n}(\hat{\beta}_{M}(j) - \beta_M(j))}{\hat{\sigma}_M(j)}\right|,
\end{equation}
where $v(j)$, for a vector $v$, represents the $j$-th coordinate of $v$ and $\hat{\sigma}_M(j)$ represents an estimator of the standard deviation of $\sqrt{n}(\hat{\beta}_M(j) - \beta_M(j))$. In order to apply Proposition~\ref{prop:MAMNonUniform} for finding quantiles of the statistic~\eqref{eq:maxt}, we need to prove a linear approximation result such as~\eqref{eq:InfluExpansion}. In case of linear regression ($\ell(u, v) = (u - v)^2/2$), it was proved in~\cite{kuchibhotla2018model} that for any $1\le k\le d$,
\[
\max_{|M|\le k}\,\left\|\hat{\beta}_M - \beta_M + \frac{1}{n}\sum_{i=1}^n \Omega_M^{-1}X_{i,M}(Y_i - X_{i,M}^{\top}\beta_M)\right\|_2 = O_p\left(\frac{k\log(ed/k)}{n}\right),
\]
where $\Omega_M := n^{-1}\sum_{i=1}^n \mathbb{E}[X_{i,M}X_{i,M}^{\top}].$ Extensions for general $\ell$ are available in that paper. This result is proved under certain tail assumptions on the observations and holds both for random and fixed covariates. For linear regression with fixed covariates, we have
\[
\hat{\beta}_M - \beta_M = \left(\frac{1}{n}\sum_{i=1}^n x_{i,M}x_{i,M}^{\top}\right)^{-1}\frac{1}{n}\sum_{i=1}^n x_{i,M}(Y_i - \mathbb{E}[Y_i]) = \frac{1}{n}\sum_{i=1}^n \Omega_M^{-1}x_{i,M}(Y_i - \mathbb{E}[Y_i]),
\]
where we write $x_i$ to note fixed covariates and the variance of $n^{1/2}(\hat{\beta}_M - \beta_M)$ is given by
\begin{equation}\label{eq:VarianceM}
\mbox{Var}\left(\sqrt{n}(\hat{\beta}_M - \beta_M)\right) ~:=~ \frac{1}{n}\sum_{i=1}^n \left(\Omega_M^{-1}x_{i,M}\right)\left(\Omega_M^{-1}x_{i,M}\right)^{\top}\mbox{Var}(Y_i).
\end{equation}
Set
$\psi_{j,M}(x_i, Y_i) = {(Y_i - \mathbb{E}[Y_i])(\Omega_M^{-1}x_{i,M})(j)}/{\sigma_M(j)},$
with $\sigma_M(j)$ representing the $j$-th diagonal element of the variance matrix~\eqref{eq:VarianceM}. Define
\[
{\Delta}_{\texttt{PoSI}} := \sup_{r\ge0}\,\left|\mathbb{P}\Bigg(\max_{\substack{|M|\le k,\\1\le j\le |M|}}\,\left|\frac{1}{\sqrt{n}}\sum_{i=1}^n \psi_{j,M}(x_i, Y_i)\right| \le r\Bigg) - \mathbb{P}\Bigg(\max_{\substack{|M|\le k,\\1\le j\le |M|}}\,|G_{j,M}| \le r\Bigg)\right|,
\]
where $(G_{j,M})_{j,M}$ has a multivariate normal distribution such that
\[
\mbox{Cov}(G_{j,M}, G_{j', M'}) = \mbox{Cov}\left(\frac{1}{\sqrt{n}}\sum_{i=1}^n \psi_{j,M}(x_i, Y_i), \frac{1}{\sqrt{n}}\sum_{i=1}^n \psi_{j',M'}(x_i, Y_i)\right).
\]
By taking the maximum over $|M|\le k$ in $\Delta_{\texttt{PoSI}}$, we are restricting the final selected model to have cardinality at most $k$.
Without loss of generality, we can take
$G_{j,M} = n^{-1/2}\sum_{i=1}^n {g_i(\Omega_M^{-1}x_{i,M})(j)}/{\sigma_M(j)}$, for $g_i\overset{ind}{\sim} N(0, \mbox{Var}(Y_i))$.
To bound $\Delta_{\texttt{PoSI}}$ from our results, we need the median $\mu_{\texttt{PoSI}}$:
\[
\mu_{\texttt{PoSI}} := \mbox{median}\left(\max_{|M|\le k}\max_{1\le j\le |M|}|G_{j,M}|\right).
\] 
Note that the Gaussian vector $(G_{j,M})_{j,M}$ has dimension $p\asymp (ed/k)^k$ and hence $\mu_{\texttt{PoSI}} = O(\sqrt{k\log(ed/k)})$. Theorem~\ref{thm:UniformCLT} implies that
$\Delta_{\texttt{PoSI}} = O(\mu_{\texttt{PoSI}}\left({k^4\log^4(ed/k)}/{n}\right)^{1/6})$. Using the pessimistic bound on $\mu_{\texttt{PoSI}}$, we get that
\begin{equation}\label{eq:ChernPoSI}
\Delta_{\texttt{PoSI}} = O(1)\left({k^7\log^7(ed/k)}/{n}\right)^{1/6},
\end{equation}
and this requires $k\log(ed/k) = o(n^{1/7}).$ In words, this means that the selected model cardinality is at most $o(n^{1/7})$. Using Theorem~\ref{thm:UniformCLT}, we can get a better requirement on $k$ if we have better bounds for~$\mu_{\texttt{PoSI}}$. It was proved in~\citet[Proposition 5.5]{berk2013valid} that if the covariates are orthogonal, that is, 
\begin{equation}\label{eq:Orthogonal}\textstyle
n^{-1}\sum_{i=1}^n x_ix_i^{\top} = I_{d},
\end{equation}
then $\mu_{\texttt{PoSI}} = O(\sqrt{\log d})$ for any $1\le k\le d$. This result was further improved by~\citet[Theorem 3.3]{bachoc2018post} in that if there exists a $\kappa\in[0, 1)$ such that for all $|M| \le k$
\begin{equation}\label{eq:RIP}\textstyle
(1 - \kappa)\|\theta\|^2_2 \le n^{-1}\sum_{i=1}^n (x_{i,M}^{\top}\theta)^2 \le (1 + \kappa)\|\theta\|_2^2\quad\mbox{for all}\quad \theta\in\mathbb{R}^{|M|},
\end{equation}
holds then 
$\mu_{\texttt{PoSI}} \le \sqrt{2\log(2d)} + C(\kappa)\kappa\sqrt{2k\log(6d/k)},$
for a function $C(\cdot)$ satisfying $C(\delta) \to 1$ as $\delta \to 0$. Condition~\eqref{eq:RIP} is called the restricted isometry property (RIP) and is (trivially) satisfied with $\kappa = 0$ if the covariates are orthogonal~\eqref{eq:Orthogonal}. Hence if $\kappa\sqrt{k}$ converges to zero then also $\mu_{\texttt{PoSI}} = O(\sqrt{\log d})$. 
Hence under the RIP condition~\eqref{eq:RIP} with $\kappa\sqrt{k}\to0$, we get from Theorem~\ref{thm:UniformCLT} that
$\Delta_{\texttt{PoSI}} ~=~ O(\sqrt{\log(ed)}\left({k^4\log^4(ed/k)}/{n}\right)^{1/6}).$
This result only requires $k(\log(ed))^{7/4} = o(n^{1/4})$ which is much weaker than that implied by~\eqref{eq:ChernPoSI}. To get a better perspective take $k = d$ and for this case, we obtain
\begin{equation}\label{eq:RIPPoSI}
\Delta_{\texttt{PoSI}} ~=~ O(1)\left({d^4\log^3(d)}/{n}\right)^{1/6},
\end{equation}
which converges to zero if $d^4 = o(n/\log^3n).$ On the other hand, the Berry-Esseen bound (\cite{MR2144310}) for the linear regression estimator on the full model $M_{\texttt{full}} = \{1,2,\ldots,d\}$ (no simultaneity involved) requires $d^{3.5} = o(n)$ which is close to the requirement from~\eqref{eq:RIPPoSI}.
\subsection{Maximal Inequalities for Empirical Processes}\label{subsec:MaximalIneq}
In this section we consider application of our non-uniform CLT result for the case of suprema of empirical processes. Suppose $(\xi_i, X_i), 1\le i\le n$ are independent random variables in the measurable space $\mathbb{R}\times\rchi$ and let $\mathcal{F}$ be a class of functions from $\rchi$ to $\mathbb{R}$ such that $\mathbb{E}[\xi_i|X_i] = 0$ for $1\le i\le n$. We consider the problem of bounding the moments of
\[\textstyle
Z_n(\mathcal{F}) ~:=~ \sup_{f\in\mathcal{F}}\,\left|n^{-1/2}\sum_{i=1}^n \xi_if(X_i)\right|.
\]
Maximal inequalities that bound $\mathbb{E}[Z_n(\mathcal{F})]$ are of great importance in the study of $M$-estimators and empirical risk minimizers; see~\cite{baraud2016bounding}. In the study of least squares estimator obtained by minimizing the quadratic loss over $f\in\mathcal{F}$ the rate is essentially determined by the behavior of $\mathbb{E}[Z_n(\mathcal{F}')]$ over subsets $\mathcal{F}'$ of $\mathcal{F}$; see~\citet[Lemma 3.1]{van2017concentration}. Rates of convergence when $\xi_i$ only have finite number of moments has received some interest recently. Even though $\mathcal{F}$ is usually countable and infinite, our results can be used in case of (weak) VC-major classes when $\xi_i$ only have 3 moments. The calculation in this section could be used to prove rates for multiple isotonic regression when the errors are dependent on covariates and only have three moments; see~\citet[page 24]{han2019global}.

A class $\mathcal{F}$ is {\bf weak VC-major} with dimension $d\ge1$ if $d$ is the smallest integer $k\ge1$ such that for all $u\in\mathbb{R}$, the class
$\mathcal{C}_u(\mathcal{F}) := \{\{x\in\rchi:\,f(x) > u\}:f\in\mathcal{F}\}$
is a VC-class of subsets of $\rchi$ with dimension not larger than $k$; see~\citet[Section 2.6]{VdvW96} for details on VC classes. It can be proved (\citet[Eq. (2.5)]{baraud2016bounding}) that for a weak VC-major class of dimension $d$, $\log|\mathcal{E}_u(x_1,\ldots,x_n)| \le d\log(2en/d)$ where
$\mathcal{E}_u(x_1,\ldots,x_n) := \{\{i=1,\ldots,n:x_i\in C\}:C\in\mathcal{C}_u(\mathcal{F})\}.$
If $f(x)\in[0,1]$ for all $x\in\rchi$, $f\in\mathcal{F}$ and $\mathcal{F}$ is a weak VC-major class of dimension $d$, then~\citet[Section 3.3]{baraud2016bounding} shows
\[
Z_n(\mathcal{F}) ~\le~ \int_0^1 \sup_{C\in\mathcal{C}_u(\mathcal{F})}\left|\frac{1}{\sqrt{n}}\sum_{i=1}^n \xi_i\mathbbm{1}\{X_i\in C\}\right|du = \int_0^1 \sup_{a\in\mathcal{A}_u}\left|\frac{1}{\sqrt{n}}\sum_{i=1}^n \xi_ia_i\right|du, 
\]
where $\mathcal{A}_u = \{a = (\mathbbm{1}\{X_1\in C\},\ldots,\mathbbm{1}\{X_n\in C\})\in\{0, 1\}^n:C\in\mathcal{C}_u(\mathcal{F})\}$.
From the assumption of weak VC-major class in $\mathcal{F}$, the supremum over $C\in\mathcal{C}_u(\mathcal{F})$ is actually a finite maximum over $a\in\mathcal{A}_u$ which allows us to apply our uniform and non-uniform CLTs (by conditioning on $X_1,\ldots,X_n$ so that $\mathcal{A}_u$ can be treated non-random). To state the final bound, note that after conditioning on $X_1,\ldots,X_n$ the ``limiting'' Gaussian vector can be taken to as $(G_a)_{a\in\mathcal{A}_u}$ where $G_a := n^{-1/2}\sum_{i=1}^n g_i\sigma(X_i)a_i$, for $\sigma^2(X_i) := \mathbb{E}[\xi^2_i|X_i]$ and $g_i\overset{iid}{\sim} N(0, 1).$ Set $\mathcal{X}_n = \{X_1,\ldots,X_n\}.$ If $\sup_x\mathbb{E}[\xi_i^3|X_i = x] < \infty$, then Remark~\ref{Rem:MomentConvergence} implies that
\[
\mathbb{E}_{|\mathcal{X}_n}\left[\sup_{a\in\mathcal{A}_u}\left|\frac{1}{\sqrt{n}}\sum_{i=1}^n \xi_ia_i\right|\right] = \mu_u + O\left(1\right)\mu_u^2\log(\mu_u)\frac{(L_{n,u}d^2\log^2(en/d))^{1/3}}{n^{1/6}},
\]
where $\mu_u := \mathbb{E}_{|\mathcal{X}_n}\left[\sup_{a\in\mathcal{A}_u}\left|n^{-1/2}\sum_{i=1}^n g_i\sigma(X_i)a_i\right|\right]$. Here $\mathbb{E}_{|\mathcal{X}_n}$ represents the expectation conditional on $\mathcal{X}_n$. Taking expectations with respect to $X_1,\ldots,X_n$, we get if $d = O(1)$,
\[
\mathbb{E}\left[Z_n(\mathcal{F})\right] ~\le~ (1 + o(1))\int_0^1 \mathbb{E}\left[\sup_{C\in\mathcal{C}_u(\mathcal{F})}\left|\frac{1}{\sqrt{n}} \sum_{i=1}^n g_i\sigma(X_i)\mathbbm{1}\{X_i\in C\}\right|\right]du.
\]
Note that the expectation on the right hand side is the Gaussian mean width of sets $\mathcal{A}_u$. The right hand side can be bounded using several existing results; see, e.g.,~\citet[Chapters 2,10]{talagrand2014upper},~\citet[Theorem 1]{han2019}.  
\section{Cram{\'e}r-type Large Deviation}\label{sec:LargeDeviation}
In this section, we prove a Cram{\'e}r-type large deviation result for $\norm{S_n}$. A version of this result appeared in \citet[Theorem 1]{bentkus1987} for the case of Banach space valued random variables. In the following result, we make the dependences on distributional constants precise. For this result, we assume that the observations $X_1, \ldots, X_n$ are i.i.d. and write the variation measure as $\zeta$ instead of $\zeta_i$. Set $C_2 = C_0\log(ep), C_3 = C_0\log^2(ep)$. We assume the following condition for the result: there exists $H \in (0, \infty)$ such that $\int \exp(H\|x\|)|\zeta|(dx) \le 4.$ We also need the following quantities. Set $B := 2(1 + \sigma_{\max}^{-2})H^{-1}$, and
\begin{align*}
\tilde{\Pi}_n &:= 4 + 12\Phi_{AC,0}{(8C_3L_ne^{\mathfrak{C}})^{{1}/{3}}} + \frac{20\Phi_{AC,0}\log(ep)\log(8C_0n)}{Hn^{1/3}} + \frac{5.1\log(ep)}{C_0n^{5/6}},\\
\Pi &:= \max\left\{\tilde{\Pi}, (132\Phi_2)^{4/3}, \frac{19C_3L_ne^{\mathfrak{C}}}{\Phi_4}, ({37C_3L_ne^{\mathfrak{C}}})^{{4}/{7}}, \left(\frac{24C_2}{\Phi_4^5H^2}\right)^{4/11}\right\},\\
M &:= \max\left\{2\Pi,\left(112\Phi_2 + 83C_3L_n\right)^{\frac{4}{3}}, \frac{(48C_2)^{10/23}}{(\Phi_4^{32}H^{20})^{1/23}}, 36\left(C_3L_n\Phi_2\right)^{\frac{2}{3}}, \frac{(124C_3L_n)^2}{(\mu+1)^{17/8}n^{5/16}}\right\}.
\end{align*}
The quantity $M$ determines how big the ratio~\eqref{eq:RatioLargeDeviation} is relative to $(r+1)n^{-1/6}$.
\begin{thm}(Cram{\'e}r-type large deviation)\label{thm:largedeviation}
Under the setting above, we have for $n\ge4,$
$$\left|\frac{\mathbb{P}\left(\norm{S_n} > r\right)}{\mathbb{P}\left(\norm{Y} > r\right)} - 1\right| \leq 1.02M(r+1)n^{-{1}/{6}},$$
for all $r$ such that $(r+1)n^{-1/6} \le \mathfrak{B}_0\exp\left(-3M^{1/4}(\mu+1)^{-17/16}\log(en)n^{-5/32}\right)$, where
\begin{align*}
\mathfrak{B}_0 &:= \min\left\{\frac{1}{3(\Phi_4^4\Pi)^{1/3}}, \frac{1}{4(\Phi_4^4M)^{4/15}}, \frac{(\log(en))^{-1/3}}{2(\Phi_4B)^{1/3}}, \frac{\Pi^{1/9}(\log(en))^{-4/9}}{2(B\log(ep))^{4/9}}, \frac{M^{1/8}(\log(en))^{-1/2}}{(6\mathfrak{C}B\log(ep))^{1/2}}\right\}.
\end{align*}
\end{thm}
The following corollary is obtained for sub-Weibull random vectors using $\Phi_2 \le \Theta\mu^2 \le \Theta\log(ep)$ and $\Phi_4 \le \Theta\mu^4 \le \Theta\log^2(ep)$. For examples where $\mu$ is of smaller order than $\sqrt{\log(ep)}$ (as in post-selection inference example; Section~\ref{subsec:PoSI}) better rates follow easily. It is noteworthy that the rates in the corollary does not depend on $\alpha\in[1,2]$.
\begin{cor}\label{cor:Theorem31}
Suppose that assumption~\eqref{eq:OrliczNorm} holds
for some $1 \leq \alpha \leq 2$.
Then, there exist positive constants $\Theta_1,\Theta_2\in(0,\infty)$ depending on $K_p, {\sigma}_{\min}$ and ${\sigma}_{\max}$, such that 
$$\left|{\mathbb{P}\left(\norm{S_n} > r\right)}/{\mathbb{P}\left(\norm{Y} > r\right)} - 1\right| \leq \Theta_1 (\log(ep))^{8/3}(r+1)n^{-1/6},$$
for all $n \geq (\log(ep))^{64/15}(\log(en))^{32/5}(\mu+1)^{-34/5}$ and $(r+1)n^{-1/6} \leq \Theta_2(\log(ep + n))^{-28/9}$.
\end{cor}
The corollary requires $n$ to be slightly larger than $\log^4(ep)$ which becomes relaxed if $\mu \ge \Theta\sqrt{\log(ep)}$. Note that $n\ge \log^4(ep)$ is required for the bound on $\delta_{n,0}$ to be less than 1.
\section{Outline of the Proofs}\label{sec:Outline}
The following represents the main steps in our proofs of uniform and non-uniform CLTs. 
The outline for the large deviation result is given at the end of this section. 
We also detail the differences between the proofs of~\cite{CCK17} and~\cite{MR1162240},~\cite{bentkus2000accuracy}. We define and control for $r\ge0$,
\[
\Delta_{n,k}(r) := |\mathbb{P}(\|U_{n,k}\| \le r) - \mathbb{P}(\|U_{n,0}\| \le r)|.
\]
Recall that $\delta_{n,0} = \sup_{r\ge0}\max_{1\le k\le n}\Delta_{n,k}(r)$ and $\delta_{n,m} = \sup_{r\ge0}\max_{1\le k\le n}r^m\Delta_{n,k}(r)$.

\noindent {\bf Step A: Smoothing Inequality.} The first step in all Berry-Esseen type results is a smoothing inequality which replaces the probabilities by expectations of smooth approximations of indicators. By Lemma~5.1.1 of~\cite{paulauskas2012approximation}, we have
\begin{equation}\label{eq:Smoothing}
\Delta_{n,k}(r) \le \max_{\ell = 1,2}|\mathbb{E}\left[\varphi_{\ell}(U_{n,k}) - \varphi_{\ell}(U_{n,0})\right]| + \mathbb{P}(r - \varepsilon \le \|U_{n,0}\| \le r + \varepsilon),
\end{equation}
where $\varphi_1(x) := \varphi_{r,\varepsilon}(x)$ and $\varphi_2(x) := \varphi_{r-\varepsilon,\varepsilon}(x)$ are given by Lemma~\ref{lem:SmoothApproximation}. The second term in~\eqref{eq:Smoothing} is controlled by Theorem~\ref{thm:AntiConcentration} for anti-concentration. 

\vspace{0.1in}

\noindent {\bf Step B: Lindeberg Replacement.} Write $\varphi(\cdot)$ for both $\varphi_1(\cdot)$ and $\varphi_2(\cdot)$ and fix any $k\ge1$. To bound the first term in~\eqref{eq:Smoothing}, we use Lindeberg method.
\begin{align*}
|\mathbb{E}[\varphi(U_{n,k}) - \varphi(U_{n,0})]| &\le \sum_{j=1}^k |\mathbb{E}[\varphi(U_{n,j-1}) - \varphi(U_{n,j})]|
= \sum_{j=1}^k \left|\int \mathbb{E}[\varphi(W_{n,j} + n^{-1/2}x)]\zeta_j(dx)\right|,
\end{align*}  
where $W_{n,j} = n^{-1/2}(X_1+\ldots+X_{j-1}+Y_{j+1}+\ldots+Y_n)$. The last equality follows from $U_{n,j} = W_{n,j} + n^{-1/2}X_j$ and $U_{n,j-1} = W_{n,j} + n^{-1/2}Y_j$. Using the Taylor expansion of $\varphi$ (possible because of smoothness) and $\mathbb{E}[X_j] = \mathbb{E}[Y_j] = 0$, $\mathbb{E}[X_jX_j^{\top}] = \mathbb{E}[Y_jY_j^{\top}]$, we get
\begin{align*}
\left|\int \mathbb{E}[\varphi(W_{n,j} + n^{-1/2}x)]\zeta_j(dx)\right| &= \left|\int \mathbb{E}[\mbox{Rem}_n(W_{n,j}, x)]\zeta_j(dx)\right|\\
&\le \int \mathbb{E}[|\mbox{Rem}_n(W_{n,j}, x)|]|\zeta_j|(dx) ~=:~ I_j,
\end{align*}
where
\[
\mbox{Rem}_n(y, x) = \varphi(y + xn^{-1/2}) - \varphi(y) - \frac{1}{\sqrt{n}}\sum_{j=1}^p x(j)\partial_j\varphi(y) - \frac{1}{2n}\sum_{j,k=1}^p x(j)x(k)\partial_{jk}\varphi(y). 
\]
\noindent {\bf Step C: Splitting the Integral.} Based on Step B, it remains to bound the integral of $\mbox{Rem}_{n}(W_{n,j},x)$ with respect to $\zeta_j$. Using the mean value theorem and the bound on derivatives in Lemma~\ref{lem:SmoothApproximation}, we have
\begin{equation}\label{eq:RemainderBound}
|\mbox{Rem}_{n}(W_{n,j},x)| \le \min\left\{\frac{C_0\log^2(ep)\varepsilon^{-3}\|x\|^3}{6n^{3/2}}, \frac{C_0\log(ep)\varepsilon^{-2}\|x\|^2}{n}\right\}.
\end{equation}
Since the support of $\varphi_1(\cdot), \varphi_2(\cdot)$ is contained in $\{x: \|x\| \in [r-\varepsilon, r+\varepsilon]\}$, the bound above on the remainder can be multiplied by
$\mathbbm{1}\{|\|W_{n,j}\| - r| \le \varepsilon + n^{-1/2}\|x\|\}.$
Noting that the change point (where the minimum changes from first to second term) in the minimum is at order $n^{1/2}\varepsilon/\log(ep)$, we split the integral of remainder into two parts. Set $\mathcal{E} := \{x\in\mathbb{R}^p:\,\|x\| \le n^{1/2}\varepsilon/\log(ep)\}$ and this yields
\[
I_j = \int_{\mathcal{E}} \mathbb{E}[|\mbox{Rem}_n(W_{n,j}, x)|]|\zeta_j|(dx) + \int_{\mathcal{E}^c} \mathbb{E}[|\mbox{Rem}_n(W_{n,j}, x)|]|\zeta_j|(dx) ~=:~ I_j^{(1)} + I_j^{(2)}.
\]
It is intuitively clear that $I_j^{(1)}$ contributes to the main rate term in the bound and $I_j^{(2)}$ would be a second order term since it is an integral over a tail end of the distribution. 

In the classical proof of~\cite{bentkus2000accuracy}, the splitting of integral is done at $\|x\| \le n^{1/2}\varepsilon$ and in the modern proof of~\cite{CCK17} the splitting is done as above.

\vspace{0.1in}

\noindent {\bf Step D: Controlling $I_j^{(2)}$.} \cite{CCK17} bound $I_j^{(2)}$ as
\[
I_j^{(2)} \le \int_{\mathcal{E}^c} \frac{C_0\log^2(ep)\varepsilon^{-3}\|x\|^3}{6n^{3/2}}|\zeta_j|(dx).
\]
This is sub-optimal since from~\eqref{eq:RemainderBound} the second term in the minimum is a better upper bound. Hence in Theorem~\ref{thm:UniformCLT}, we bound $I_j^{(2)}$ using
\begin{equation}\label{eq:OurIj2}
I_j^{(2)} \le \int_{\mathcal{E}^c} \frac{C_0\log(ep)\varepsilon^{-2}\|x\|^2}{n}|\zeta_j|(dx).
\end{equation}
Summing over $1\le j\le n$ leads to $M_n(\varepsilon)$ in Theorem~\ref{thm:UniformCLT} and leads to better rates under $(2+\tau)$-moments of $\|X_j\|$.

In the classical proof of~\cite{bentkus2000accuracy}, $I_j^{(2)}$ is bounded using
\[
I_j^{(2)} \le \int_{\mathcal{E}^c} \frac{C_0\log(ep)\varepsilon^{-2}\|x\|^2}{n}\mathbb{P}(r - \varepsilon - n^{-1/2}\|x\| \le \|W_{n,j}\| \le r + \varepsilon + n^{-1/2}\|x\|)|\zeta_j|(dx).
\]
Note that this bound is better than~\eqref{eq:OurIj2} since the probability is bounded by $1$. Since $W_{n,j}$ is related to $U_{n-1,j-1}\sqrt{1 - 1/n}$, we can use the definition of $\delta_{n-1,0}$ (or $\delta_{n-1,m}$ for non-uniform CLT) to obtain
\begin{align*}
&\mathbb{P}(r - \varepsilon - n^{-1/2}\|x\| \le \|W_{n,j}\| \le r + \varepsilon + n^{-1/2}\|x\|)\\
&\quad\le \mathbb{P}(\sqrt{n/(n-1)}(r - \varepsilon - n^{-1/2}\|x\|) \le \|U_{n,0}\| \le \sqrt{n/(n-1)}(r + \varepsilon + n^{-1/2}\|x\|)) + 2\delta_{n-1,0}\\
&\quad\le \Phi_{AC,0}(\varepsilon + n^{-1/2}\|x\|)\sqrt{n/(n-1)} + 2\delta_{n-1,0}.
\end{align*}
Here the last inequality follows from Theorem~\ref{thm:AntiConcentration}. This is used in the proof of Theorem~\ref{thm:UniformCLT} to get the ``right'' dependence on $\tau$ for the finite moment case. To make the recursion work, we need the anti-concentration inequality to not change with sample size and for this reason i.i.d. assumption is introduced in Proposition~\ref{prop:BentkusProof} and Theorem~\ref{thm:OptimalCLT}. In case of non-uniform CLT we need to split $\mathcal{E}^c$ further in order to avoid dividing by zero when using $\delta_{n-1,m}$.

\vspace{0.1in}

\noindent {\bf Step E: Controlling $I_j^{(1)}$.} The simpler approach in bound $I_j^{(1)}$ is in the classical proof of~\cite{bentkus2000accuracy} that uses
\[
I_j^{(1)} \le \int_{\mathcal{E}} \frac{C_0\log^2(ep)\varepsilon^{-3}\|x\|^3}{6n^{3/2}}\mathbb{P}(r - \varepsilon - n^{-1/2}\|x\| \le \|W_{n,j}\| \le r + \varepsilon + n^{-1/2}\|x\|)|\zeta_j|(dx).
\]
Since $\|x\| \le n^{1/2}\varepsilon/\log(ep)$ for $x\in\mathcal{E}$, the probability can be bounded by $\mathbb{P}(r - 2\varepsilon \le \|W_{n,j}\| \le r + 2\varepsilon)$ (which does not involve $x$ anymore). Now as in Step D, we can relate this probability to $\delta_{n-1,0}$ and $\Phi_{AC,0}\varepsilon$. This results in $\int \|x\|^3|\zeta_j|(dx) = \nu_3^3$ in the final bound in Proposition~\ref{prop:BentkusProof} and leads to a sub-optimal rate in case $X_j$ have exponential tails.

A better way to control $I_j^{(1)}$ from~\cite{CCK17} uses
\begin{align*}
&\left|\mbox{Rem}_n(W_{n,j},x)\right|\\ 
&= \frac{1}{2}\left|\sum_{j_1,j_2,j_3 = 1}^p \frac{x(j_1)x(j_2)x(j_3)}{n^{3/2}}\int_0^1 (1 - t)^2\partial_{j_1j_2j_3}(W_{n,j} + tn^{-1/2}x)dt\right|\\
&\le \sum_{j_1,j_2,j_3 = 1}^p \frac{|x(j_1)x(j_2)x(j_3)|}{2n^{3/2}}\int\displaylimits_0^1 (1 - t)^2D_{j_1j_2j_3}\left(W_{n,j} + \frac{tx}{\sqrt{n}}\right)\mathbbm{1}\left\{|\|W_{n,j}\| - r| \le \varepsilon + \frac{\|x\|}{\sqrt{n}}\right\}dt
\end{align*}
By stability property~\eqref{eq:StabilityProp} of $D_{j_1j_2j_3}$ along with $\|x\|/\sqrt{n} \le \varepsilon/\log(ep)$ for $x\in\mathcal{E}$, we get
\[
e^{-\mathfrak{C}}D_{j_1j_2j_3}(W_{n,j}) ~\le~ D_{j_1j_2j_3}(W_{n,j} + txn^{-1/2}) ~\le~ e^{\mathfrak{C}}D_{j_1j_2j_3}(W_{n,j}).  
\]
This implies
\[
I_j^{(1)} \le \frac{e^{\mathfrak{C}}}{2n^{3/2}}\sum_{j_1,j_2,j_3=1}^p \int_{\mathcal{E}} |x(j_1)x(j_2)x(j_3)||\zeta_j|(dx)\mathbb{E}[D_{j_1j_2j_3}(W_{n,j})\mathbbm{1}\{|\|W_{n,j}\| - r| \le 2\varepsilon\}].
\]
By H{\"o}lder's inequality,
\[
\int_{\mathcal{E}} |x(j_1)x(j_2)x(j_3)||\zeta_j|(dx) \le \max_{1\le j_1\le p} \int_{\mathcal{E}} |x(j_1)|^3|\zeta_j|(dx).
\]
If $X_j$ has $(2 + \tau)$-moments for some $\tau \ge 1$, then the integral over $\mathcal{E}$ can be replaced by integral over $\mathbb{R}^p$ which leads to $L_n$ (when summed over $1\le j\le n$). In case $X_j$ has only $(2 + \tau)$-moments for some $\tau < 1$, then using $\mathcal{E} = \{\|x\| \le n^{1/2}\varepsilon/\log(ep)\}$, we can write
\[
\int_{\mathcal{E}} |x(j_1)|^3|\zeta_j|(dx) \le \left(\frac{n^{1/2}\varepsilon}{\log(ep)}\right)^{1 - \tau}\int |x(j_1)|^{2+\tau}|\zeta_j|(dx),
\] 
which would lead to a version of $L_n$ only involving ``weak'' $(2+\tau)$-moment rather than the ``weak'' third moment. Getting back to $I_j^{(1)}$, the above discussion leads to 
\begin{equation}\label{eq:Ij1BoundPrelim}
I_j^{(1)} \le \frac{e^{\mathfrak{C}}}{2n^{3/2}}\max_{1\le j_1\le p}\int_{\mathcal{E}} |x(j_1)|^3|\zeta_j|(dx)\sum_{j_1,j_2,j_3 = 1}^p \mathbb{E}[D_{j_1j_2j_3}(W_{n,j})\mathbbm{1}\{|\|W_{n,j}\| - r| \le 2\varepsilon\}].
\end{equation}
There are two ways to bound the right hand side. First, the way used in the proof of~\cite{CCK17} is to split the indicator inside the expectation by adding $Y_j\in\mathcal{E}$ and $Y_j\in\mathcal{E}^c$:
\begin{align*}
\mathbb{E}\left[D_{j_1j_2j_3}(W_{n,j})\mathbbm{1}\{|\|W_{n,j}\| - r| \le 2\varepsilon\}\right] &= \mathbb{E}\left[D_{j_1j_2j_3}(W_{n,j})\mathbbm{1}\{|\|W_{n,j}\| - r| \le 2\varepsilon, Y_j\in\mathcal{E}\}\right]\\
&\qquad+ \mathbb{E}\left[D_{j_1j_2j_3}(W_{n,j})\mathbbm{1}\{|\|W_{n,j}\| - r| \le 2\varepsilon, Y_j\in\mathcal{E}^c\}\right].
\end{align*}
The first term on the right hand side can be linked using the stability property~\eqref{eq:StabilityProp} to $\mathbb{E}[D_{j_1j_2j_3}(U_{n,j-1})\mathbbm{1}\{|\|U_{n,j-1}\| - r| \le 3\varepsilon\}]$. Further using the bound on derivative this can bounded by $C_3\varepsilon^{-3}\mathbb{P}(r - 3\varepsilon \le \|U_{n,j-1}\| \le r + 3\varepsilon)$ that can linked to $\delta_{n,0} + \Phi_{AC,0}\varepsilon$. The second term is, anyways, small since it involves $Y_j\in\mathcal{E}^c$ which is a small probability event since $Y_j$ has Gaussian tails.

The calculations after~\eqref{eq:Ij1BoundPrelim} are used in the proof of~\cite{CCK17} to get a bound in terms of $\delta_{n,0}$ and solve the inequality for $\delta_{n,0}$. This is what we followed in the proof of Theorem~\ref{thm:UniformCLT} but using the ``right'' bound on $I_j^{(2)}$.

For the proof of Theorem~\ref{thm:OptimalCLT}, we proceed from~\eqref{eq:Ij1BoundPrelim} by using
\[
\sum_{j_1,j_2,j_3 = 1}^p \mathbb{E}[D_{j_1j_2j_3}(W_{n,j})\mathbbm{1}\{|\|W_{n,j}\| - r| \le 2\varepsilon\}] \le C_3\varepsilon^{-3}\mathbb{P}(r - 2\varepsilon \le \|W_{n,j}\| \le r + 2\varepsilon),
\] 
where the right hand side can be bounded in terms of $\delta_{n-1,0}$ and anti-concentration term as shown above in the control of $I_j^{(2)}.$

Combining the bounds and solving recursions (when needed) completes the proofs of all three versions of uniform CLT. By appropriate (minor) modifications mentioned in the outline above, the two versions of non-uniform CLT also follow.
\paragraph{Sketch of the proof of Theorem \ref{thm:largedeviation}.} The proof of Theorem \ref{thm:largedeviation} relies upon Lindeberg method and a refined induction argument. 
The starting point for the proof is Step B of the outline. To control $I_j$ we split it into two parts depending on whether $\{\|x\| \le n^{1/2}f_n(r)\}$ or not for some function $f_n(\cdot)$ of $r$. The case when $\norm{x} \ge n^{1/2}f_n(r)$ the integral is bounded using $\int \exp(H\|x\|)|\zeta|(dx) \le 4$ and Markov's inequality by:
\begin{equation}\label{eq:remainderlarge}
\frac{8\beta C_2\varepsilon^{-2}}{nH^2}\exp(-Hn^{1/2}f_n(r)/2)
\le \frac{8C_2\varepsilon^{-2}\beta}{n^{2+\Phi_1}H^2\Phi_0}\mathbb{P}(\|Y\| > r).
\end{equation}
 The integral corresponding to the case when $\norm{x}$ is smaller than $n^{1/2}f_n(r)$ is bounded by 
\begin{equation}\label{eq:remainderII}
\frac{C_3\varepsilon^{-3}e^{\mathfrak{C}f_n(r)\log(ep)/\varepsilon}L_n}{n^{3/2}} \mathbb{P}\left(\bar{a}_n(r) \le \|W_{n,j}\| \le \bar{b}_n(r)\right),
\end{equation}
for suitably chosen $\bar{a}_{n}(r)$ and $\bar{b}_{n}(r)$.
Since $W_{n,j}\stackrel{d}{=} \left((n-1)/n\right)^{1/2}U_{n-1,j}$, we get 
\begin{align}
\Delta_{n}(r) &\le \mathbb{P}(r - \varepsilon \le \|Y\| \le r + \varepsilon) + \frac{8C_2\varepsilon^{-2}\beta}{\Phi_0H^2n^{1+\Phi_1}}\mathbb{P}(\|Y\| > r)\label{eq:lemmaA21}\\
&\qquad+ \frac{C_3\varepsilon^{-3}L_ne^{\mathfrak{C}f_n(r)\log(ep)/\varepsilon}}{6n^{1/2}}\max_{0\le k\le n}\mathbb{P}(a_n(r) \le \|U_{n-1,k}\| \le b_n(r))
\end{align}
for some $a_n(r)$ and $b_n(r)$.
Our next task is to bound the right hand side of \eqref{eq:lemmaA21} in term of $\mathbb{P}\left[\norm{Y} >q r  \right]$. 
 We now inductively use a bound on $\mathbb{P}\left[\|U_{n-1,j}\|\ge a_{n}(r)  \right]$ in terms of $\mathbb{P}\left[ \norm{Y} >q r \right]$ to get a bound for \eqref{eq:lemmaA21}. Finally bounding $\mathbb{P}\left[r-\varepsilon \le \norm{Y} \le r+\varepsilon  \right]$ in terms of $\mathbb{P}\left[ \norm{Y} >q r \right]$ using Theorem \ref{thm:density} and summing this with the bounds for \eqref{eq:remainderlarge}, \eqref{eq:remainderII} we obtain the following result. For any $1\le k \le n$,
$|\mathbb{P}\left( \norm{U_{n,k}}> r \right)/\mathbb{P}\left(\norm{Y}> r\right) - 1| \le \Pi T_{n,r}^{1/4}$,
for all $r \in \mathbb{R}$ satisfying
\[
T_{n,r} ~\le~ \min\left\{\frac{1}{16\Phi_4^4\Pi}, \frac{1}{2\Phi_4B\log(en)}, \frac{\Pi^{1/3}}{(B\log(en)\log(ep))^{4/3}}\right\}, 
\]  
where $T_{n,r}=(r+1)^3n^{-{1/2}}$. One can find the details in Lemma~\ref{lem:sixtpt}.
Since $T_{n,r}^{1/4}= (r+1)^{3/4}n^{-1/8}$, Lemma \ref{lem:sixtpt} gives a large deviation with rate $n^{-1/8}$ (for all $n\ge1$). However this rate can be modified to rate $n^{-1/6}$ by a more refined induction argument. This is done in detail in the proof of Theorem~\ref{thm:largedeviation}.
\section{Summary and Future Directions}\label{sec:Conclusions}
In this paper, we proved non-uniform central limit theorems and large deviations for scaled averages of independent high-dimensional random vectors based on dimension-free anti-concentration inequalities. We further illustrated the usefulness of these results in the context of post-selection inference for linear regression and in bounding the expectation of suprema of empirical processes. All the proofs are based on Lindeberg method which was an integral tool in Banach space CLTs. Using the stability property introduced in~\cite{CCK17}, we obtained refinements for uniform as well as non-uniform CLTs. It should be mentioned here that we credit~\cite{bentkus1987},~\cite{bentkus1990} for the proof of Theorem~\ref{thm:largedeviation}. In comparison to~\cite{CCK17}, we mention that our setting is restrictive in the sense that we consider $l_{\infty}$ balls while \cite{CCK17} consider general sparsely convex sets. Once an anti-concentration result such as Theorem~\ref{thm:AntiConcentration}, it is fairly easy to follow our proofs to extend the results which we hope to pursue in the future along with the discussion of ``best'' anti-concentration inequalities.

In this work, we have presented the results in the simpler setting with independent random vectors. Lindeberg method is well-known for its robustness to independence assumptions and hence extensions to the case of dependent random vectors (in particular martingales) form an interesting future direction. Apart from the application in PoSI, other areas of interest in terms of applications are bootstrap and high-dimensional vectors with a specified group structure. Our results can be used to obtain CLTs and large deviations to empirical processes which forms an interesting direction. See \cite{MR1115160} and \cite{MR3262461} for some results.
\appendix
\section{Uniform CLT: for $(2+\tau)$-moments with $\tau < 1$}\label{app:TauLe1}
Following the outline in (Step E of) Section~\ref{sec:Outline} for $\tau \le 1$, we get for any $\varepsilon > 0$:
\begin{align*}
\delta_{n,0} &\le \frac{e^{2\mathfrak{C}}C_3\varepsilon^{-3}L_{n,\tau}}{2n^{3/2}}\left(\frac{n^{1/2}\varepsilon}{\log(ep)}\right)^{1 - \tau}\left[2\Phi_{AC,0}\varepsilon + \delta_{n,0}\right] + \frac{C_2M_n(\varepsilon)}{\varepsilon^{2}} + \Phi_{AC,0\varepsilon}\\
&\quad+ \frac{e^{\mathfrak{C}}C_3\varepsilon^{-3}}{2n^{3/2}}\left(\frac{n^{1/2}\varepsilon}{\log(ep)}\right)^{1 - \tau}\sum_{j=1}^n \max_{1\le j_1\le p}\int |x(j_1)|^{2+\tau}|\zeta_j|(dx)\mathbb{P}(\|Y_j\| > n^{1/2}\varepsilon/\log(ep)),
\end{align*}
where $L_{n,\tau} = n^{-1}\sum_{j=1}^n \max_{1\le j_1\le p}\int |x(j_1)|^{2+\tau}|\zeta_j|(dx)$. It is also clear that
\[
M_n(\varepsilon) \le \nu_{2+\tau}^{2+\tau}\left(\frac{\log(ep)}{n^{1/2}\varepsilon}\right)^{2+\tau}.
\]
Set $\varepsilon = \varepsilon_n$ so that
\[
\frac{e^{2\mathfrak{C}}C_3L_{n,\tau}}{2n^{3/2}}\left(\frac{n^{1/2}\varepsilon}{\log(ep)}\right)^{1-\tau} \le \frac{1}{2}\quad\mbox{and}\quad \frac{C_2\nu_{2+\tau}^{2+\tau}\log^{\tau}(ep)}{n^{\tau/2}\varepsilon^{2+\tau}} \le \frac{1}{r_n^{2+\tau}},
\]
for some $r_n\ge1$. This implies that
\[
\delta_{n,0} \lesssim \Phi_{AC,0}\varepsilon + \frac{1}{r_n^{2+\tau}} \lesssim \Phi_{AC,0}\frac{(L_{n,\tau}\log^2(ep))^{1/(2+\tau)}}{n^{1/2}(\log(ep))^{(1-\tau)/(2+\tau)}} + r_n\Phi_{AC,0}\nu_{2+\tau}\frac{(\log(ep))^{(\tau+1)/(\tau+2)}}{n^{\tau/(4+2\tau)}} + \frac{1}{r_n^{2+\tau}}.
\]
Choosing the ``optimal'' order of $r_n$ by equating the last two terms, we get
\begin{align*}
\delta_{n,0}~&\lesssim~ \left[\frac{\Phi_{AC,0}L_{n,\tau}^{1/(2+\tau)}}{n^{1/2}} + \frac{\left(\Phi_{AC,0}\nu_{2+\tau}\right)^{(\tau+2)/(\tau+3)}}{n^{\tau/(6+2\tau)}}\right](\log(ep))^{(\tau+1)/(\tau+2)}\\
~&\lesssim~ \left[\Phi_{AC,0}L_{n,\tau}^{1/(2+\tau)} + \left(\Phi_{AC,0}\nu_{2+\tau}\right)^{(\tau+2)/(\tau+3)}\right]\frac{(\log(ep))^{(\tau+1)/(\tau+2)}}{n^{\tau/(6+2\tau)}}.
\end{align*}
The last inequality follows since $n^{\tau/(6+2\tau)} \le n^{1/8} \le n^{1/2}$.
\noeqref{eq:zetazero}
\section*{Acknowledgment}
We would like to thank Prof. Jian Ding for comments that led to an improved presentation. 
\bibliographystyle{apalike}
\bibliography{HDCLT}
\newpage
\begin{center}
\Large\textbf{Supplementary Material}
\end{center}
\setcounter{section}{0}
\renewcommand{\thesection}{S.\arabic{section}}
\section{Proof of Theorem \ref{thm:density} and Theorem~\ref{thm:AntiConcentration}}\label{appsec:AntiConcentration}
\subsection{A preliminary lemma}
We start with a lemma about anti-concentration. A similar version of this lemma can be found in \cite{GINE76}.
\begin{lem}[Lemma 2.5 of \cite{GINE76}]\label{lem:AntiConc2}
Let $Z$ be a centered sample continuous Gaussian process on a compact metric space $S$ such that $\mathbb{E}\left[Z^2(s)\right] \ge \sigma^2 > 0$ for every $s\in S$. 
Then, for $\varepsilon \le \sigma/2$ and $\lambda >0$, 
\[
\mathbb{P}\left(\lambda - \varepsilon \le \sup_{s\in S}\, \left|Z(s) \right| \le \lambda + \varepsilon\right) \le 2\varepsilon {K}(\lambda).
\]
 where
\begin{equation}\label{eq:BoundGine}
K(\lambda):= 2\sigma^{-1}(2.6 + \lambda/\sigma).
\end{equation}
\end{lem}
\begin{proof}
We shall actually prove that 
\begin{equation}\label{eq:bddgine}
\mathbb{P}\left(\lambda - \varepsilon \le \sup_{s\in S}\, Z(s)  \le \lambda + \varepsilon\right) \le \varepsilon {K}(\lambda).
\end{equation}
After proving \eqref{eq:bddgine} the result follows since 
\[
\left\{ \lambda - \varepsilon \le \sup_{s\in S}\, \left|Z(s) \right| \le \lambda + \varepsilon \right\} \subseteq \left\{ \lambda - \varepsilon \le \sup_{s\in S}\, Z(s)  \le \lambda + \varepsilon \right\} \cup \left\{ \lambda - \varepsilon \le \sup_{s\in S}\, -Z(s)  \le \lambda + \varepsilon\right\}.
\]
Now observe that 
\begin{equation}
\begin{split}
\mathbb{P}\left(\lambda - \varepsilon \le \sup_{s\in S}\, Z(s)  \le \lambda + \varepsilon\right)&= \mathbb{P}\left( -\varepsilon \le \sup_{s\in S}\, \left(Z(s)-\lambda\right)  \le \varepsilon \right)\\
&= \mathbb{P}\left( \cup_{s\in S}\left\{ -\varepsilon \le \left(Z(s)-\lambda\right)\right\} \cap \left(\cap_{s \in S} \left\{\left(Z(s)-\lambda\right)  \le \varepsilon \right\}\right)  \right)  \\
& \le \mathbb{P}\left( \bigcup_{s \in S} \left\{-\frac{\varepsilon}{\sigma} \le \frac{Z(s)-\lambda}{\sigma(s)}\right\}\cap \bigcap_{s \in S}\left\{\frac{Z(s)-\lambda}{\sigma(s)} \le \frac{\varepsilon}{\sigma}  \right\} \right)\\
& = \mathbb{P}\left( -\frac{\varepsilon}{\sigma} \le \sup_{s \in S} \frac{Z(s)-\lambda}{\sigma(s)} \le \frac{\varepsilon}{\sigma} \right)\\
&= \mathbb{P}\left( \frac{\lambda-\varepsilon }{\sigma} \le \sup_{s \in S} \frac{Z(s)-\lambda}{\sigma(s)}+\frac{\lambda}{\sigma} \le \frac{\varepsilon+ \lambda}{\sigma} \right)
\end{split}
\end{equation}
Observe that the process $\left( \left( Z(s)-\lambda \right)/\sigma(s) + \lambda/\sigma \right)$ has non negative mean and variance identical to $1$. So from the arguments from \citet{GINE76}, we have 
\begin{equation}
\begin{split}
\mathbb{P}\left(\lambda - \varepsilon \le \sup_{s\in S}\, Z(s)  \le \lambda + \varepsilon\right) \le \int_{\sigma^{-1}(\lambda-\varepsilon)}^{\sigma^{-1}(\lambda +\varepsilon)} f(x) dx 
\end{split}
\end{equation}
where 
\[
f(x) \le \left\{  
\begin{array}{ll}
2.6, & \text{if} \quad x \le 1,\\
x+ x^{-1}, & \text{otherwiswe}. 
\end{array}
\right.
\]
Now consider the following two cases $(1)~ \lambda \le \sigma/2$, $(2)~ \sigma/2 \le \lambda$. In case $(1)$, $\sigma^{-1}(\lambda +\varepsilon) \le 1$ so, using $f(x)\le 2.6$, we have 
\[
\int_{\sigma^{-1}(\lambda-\varepsilon)}^{\sigma^{-1}(\lambda +\varepsilon)} f(x) dx \le 5.2 \sigma^{-1}\varepsilon.
\]
In case $(2)$, we have $f(x) \le 2.6 +x $ and so, 
\[
\int_{\sigma^{-1}(\lambda-\varepsilon)}^{\sigma^{-1}(\lambda +\varepsilon)} f(x) dx \le 5.2 \sigma^{-1} \varepsilon + \frac{1}{2\sigma^2}\left[(\lambda+\varepsilon)^{2}-\left( \lambda -\varepsilon \right)^2  \right]= 5.2 \sigma^{-1}\varepsilon + \frac{2\lambda \varepsilon}{\sigma^2} =\varepsilon K(\lambda).
\]
This completes the proof.
\end{proof}
\subsection{Proof of Theorem \ref{thm:density}}

\textbf{Proof of part (1):}
Let $t_0\in S$ be the index such that 
\[
\sigma_{\max}^2 = \mbox{Var}(Y(t_0)) = \sup_{t\in S}\mbox{Var}\left(Y(t)\right).
\]
It is clear that
\[
\mathbb{P}\left(\norm{Y} >q r\right) \ge \mathbb{P}\left(|Y(t_0)| > r\right).
\]
Since $Y(t_0)\sim N(0, \sigma_{\max}^2)$, by Mill's ratio~\citep{MR0005558} we get that:
\[
\mathbb{P}\left(|Y(t_0)| > r\right) = \mathbb{P}\left(\frac{|Y(t_0)|}{\sigma_{\max}} \ge \frac{r}{\sigma_{\max}}\right) \ge \frac{(r/\sigma_{\max})}{1 + (r/\sigma_{\max})^2}\frac{1}{\sqrt{2\pi}}\exp\left(-\frac{r^2}{2\sigma_{\max}^2}\right).
\]
If $r \ge \sigma_{\max}$, then we claim that
\begin{equation}\label{auxlem2pqmills}
\frac{(r/\sigma_{\max})}{1 + (r/\sigma_{\max})^2}\frac{1}{\sqrt{2\pi}}\exp\left(-\frac{r^2}{2\sigma_{\max}^2}\right) \ge \frac{1}{2\sqrt{2\pi}}\exp\left(-\frac{r^2}{\sigma_{\max}^2}\right).
\end{equation}
Since $r/\sigma_{\max} \geq 1$, we have:
\[
\exp\left(\frac{r^2}{2\sigma_{\max}^2}\right) \geq 1 + \frac{r^2}{2\sigma_{\max}^2} \geq \frac{1 + (r/\sigma_{\max})^2}{2r/\sigma_{\max}}~,
\]
proving~\eqref{auxlem2pqmills}.
Thus, for $r \ge \sigma_{\max}$,
\[
\mathbb{P}\left(|Y(t_0)| > r\right) \ge \frac{1}{2\sqrt{2\pi}}\exp\left(-\frac{r^2}{\sigma_{\max}^2}\right).
\]
If $r \le \sigma_{\max}$, then
\begin{equation}
\mathbb{P}\left(|Y(t_0)| > r\right) \ge \mathbb{P}\left(|Y(t_0)| \ge \sigma_{\max}\right) = 2Q(-1)
> \frac{1}{4} \geq \frac{1}{4}\exp\left(-\frac{r^2}{\sigma_{\max}^2}\right)
\end{equation}
where $Q(\cdot)$ is the distribution function of a standard Gaussian random variable.
Therefore, for all $r \ge 0$,
\[
\mathbb{P}\left(\norm{Y} >q r\right) \ge \min\left\{\frac{1}{4}~,~\frac{1}{2\sqrt{2\pi}}\right\}\exp\left(-\frac{r^2}{\sigma_{\max}^2}\right) = \frac{1}{2\sqrt{2\pi}}\exp\left(-\frac{r^2}{\sigma_{\max}^2}\right).
\]
Since $1/(2\sqrt{2\pi}) > 1/6$, the result follows.\qed
\\
\textbf{Proof of part $(2)$:}
Let $Q(x)= \mathbb{P}\left[ Z\le x \right]$ where $Z$ is a standard Gaussian random variable and we define $\Psi(x)=1-Q(x)$. Ehrhard's inequality implies that for all convex Borel sets $A, C \subset B$,
\begin{equation}\label{eq:Ehrhard}
Q^{-1}\left[ \mathbb{P}\left[ Y \in \alpha A+ \beta C \right] \right] \ge \alpha Q^{-1} \left[ \mathbb{P}\left[ Y \in A \right] \right] + \beta Q^{-1}\left[ \mathbb{P}\left[ Y \in C \right] \right]
\end{equation}
when $\alpha, \beta\ge 0$ and $\alpha + \beta =1$. See, for example, \citet[Theorem 3.1]{MR1957087} and \citet[Section 3]{MR2030108}. Denote $q(r):= Q^{-1}\left[\mathbb{P}\left[  \norm{Y} \le r \right]\right]$. We now show that $q(\cdot)$ is a concave function.
Fix $\alpha, \beta >0$ such that $\alpha + \beta =1$ and any $r,s\in\mathbb{R}^+$. From the definition of $q(\cdot)$ 
\[
q(\alpha r + \beta s)= Q^{-1}\left[ \mathbb{P}\left[  \norm{Y} \le \alpha r + \beta s \right] \right].
\]
Observe that if $x \in \alpha B_{ \norm{\cdot}}(0,r) + \beta B_{\norm{\cdot}}(0,s)$, then $ \norm{x}  \le \alpha r + \beta s$ from triangle inequality. Hence 
\[
q(\alpha r + \beta s) \ge Q^{-1}\left[ \mathbb{P}\left[ Y \in \alpha B_{ \norm{\cdot}}(0,r) + \beta B_{ \norm{\cdot} }(0,s) \right] \right].
\]
Now the concavity of $q(\cdot)$ follows from \eqref{eq:Ehrhard}. Let $\mu$ be the median of $ \norm{Y} $ or equivalently $q(\mu)=0$. Set $d_1 = 4\sigma_{\min}^{-1}(2.6 + \mu\sigma_{\min}^{-1})$. From Lemma~\ref{lem:AntiConc2}, we get that,
\begin{equation}\label{eq:AntiConc2}
\mathbb{P}\left(\mu - \delta \le \norm{Y} \le \mu + \delta\right) \le d_1\delta,\quad\mbox{for all $\delta \le \sigma_{\min}/2$.}
\end{equation}
In order to prove \eqref{eq:denbddI} we at first prove that it is enough to assume the following five conditions: 
\begin{center}
\begin{inparaenum}[(i)]
\item $\varepsilon \le {r}/{4}$.
\item 
$
\mu+ \varepsilon \le r- \varepsilon.
$
\item
$\mu \le r-\varepsilon$.
\item  
$q(r - \varepsilon) \ge 1.$
\item 
$\mu + 1/{(3d_1)} \le r.$
\end{inparaenum}
\end{center}

From the proof of part $(1)$, we have $\mathbb{P}\left[  \norm{Y}  > r\right] \ge (2\sqrt{2\pi })^{-1}\exp\left(-{r^2}/{\sigma_{\max}^2}\right)$ for all $r>0$. Here $\sigma_{\max}$ is the maximum variance of the coordinates. Now we verify the conditions one by one.
\\
\textbf{Condition $(i)$:} Suppose $\varepsilon > {r}/{4}$. Then
\begin{equation}\label{eq:Condition1Fail}
\begin{split}
\mathbb{P}\left[  \norm{Y} >q r-\varepsilon \right]\le 1 &\le {2\sqrt{2\pi }}\exp\left(\frac{r^2}{\sigma_{\max}^2}\right) \mathbb{P}\left[ \norm{Y} >q r \right]\\&\le {2\sqrt{2\pi }} \exp\left( \frac{4 r\varepsilon}{\sigma_{\max}^2} \right) \mathbb{P}\left[ \norm{Y} >q r \right].
\end{split}
\end{equation}
\textbf{Condition $(ii)$:} Suppose $\varepsilon \le {r}/{4}$ and $\mu+ \varepsilon > r- \varepsilon$. We divide this case in the following two sub cases. Observe that in this case $r\le 2\mu$. \\
\textbf{Sub case (a):} Here we assume $\varepsilon\ge {1}/{(6d_1)}$. In this case, we have 
\begin{equation}
\begin{split}
&\varepsilon r \ge \frac{r}{6d_1} \ge \frac{r^2}{12d_1\mu} \quad\Rightarrow\quad \frac{r^2}{\sigma_{\max}^2} \le {\left(\frac{12d_1\mu}{\sigma_{\max}^2}\right)}\varepsilon r \le c_{2}\varepsilon(r+1),
\end{split}
\end{equation} 
for $c_2 = 12d_1\mu\sigma_{\max}^{-2}$. So,
\begin{equation}\label{eq:Condition2aFail}
\begin{split}
\mathbb{P}\left(  \norm{Y} >q r-\varepsilon \right)\le 1 &\le {2\sqrt{2\pi }}\exp\left(\frac{r^2}{\sigma_{\max}^2}\right) \mathbb{P}\left(  \norm{Y} >q r \right)\\
&\le {2\sqrt{2\pi }}\exp\left( c_{2}\varepsilon(r+1) \right) \mathbb{P}\left(  \norm{Y} >q r \right).
\end{split}
\end{equation}
\textbf{Sub case (b):} Here we assume $\varepsilon < 1/{(6d_1)}$. As a consequence, we have 
\[
r \le \mu+ 2 \varepsilon \le \mu + \frac{1}{3d_1}.
\]
Now noting that $2\sigma_{\min}^{-1} \le 3d_1$, we get from~\eqref{eq:AntiConc2} that 
\[
\mathbb{P}\left[ \mu\le  \norm{Y} \le \mu+  \frac{1}{3d_1} \right]\le \frac{1}{3}.
\]
As a consequence, 
\begin{align*}
\mathbb{P}\left[  \norm{Y}  > r\right] &= 1 - \mathbb{P}\left(\norm{Y} \le r\right)\\ &\ge 1 - \mathbb{P}\left(\norm{Y} \le \mu\right) - \mathbb{P}\left(\mu \le \norm{Y} \le \mu + 1/(3d_1)\right)\ge 1 - \frac{1}{2} - \frac{1}{3} = \frac{1}{6}.
\end{align*}
So 
\begin{equation}\label{eq:COndition2bFail}
\mathbb{P}\left[ \norm{Y}  > r-\varepsilon\right]~\le~ 1 ~\le~ 6 \mathbb{P}\left[  \norm{Y} >q r \right] ~\le~ 6\exp\left(\varepsilon(r+1)\right)\mathbb{P}\left(\norm{Y}> r\right).
\end{equation}
\textbf{Condition $(iii)$:} If $(iii)$ fails then $(ii)$ fails which is covered above.\\
\textbf{Condition $(iv)$:} If $(iv)$ fails then, $q(r-\varepsilon)\le 1$. Let $\tilde{r}:=q^{-1}(1)$. We have $r \le \tilde{r} +\varepsilon$ and using $\varepsilon \le r/4$, we get $r\le 4\tilde{r}/3$. From Lemma 3.1 of \cite{MR2814399}, we get
\[
\mathbb{P}\left(\norm{Y} >q \mu + \sigma_{\max}\Psi^{-1}(\Psi(1)/2)\right) \le \Psi(1) = \mathbb{P}\left(\norm{Y} >q \tilde{r}\right).
\]
The last equality above follows from the definition of $\tilde{r}$. Thus, $\tilde{r} \le \mu + 1.5\sigma_{\max}.$ We now divide the case in two sub cases. \\
\textbf{Sub case (a):} Here we assume $\varepsilon \ge 1/(10\tilde{d})$, where $\tilde{d} := 4\sigma_{\min}^{-1}(2.6 + \tilde{r}\sigma_{\min}^{-1})$. In this case we have 
\[
\frac{r^{2}}{\sigma_{\max}^{2}} \le \frac{r}{\sigma_{\max}^2}\left(\frac{4\tilde{r}}{3}\right) \le r\varepsilon\left(\frac{14\tilde{r}\tilde{d}}{\sigma_{\max}^2}\right) \le c_3\varepsilon(r + 1),
\]
for $c_3 := 14\sigma_{\max}^{-2}\tilde{r}\tilde{d}$. So 
\begin{equation}\label{eq:Condition4aFail}
\begin{split}
\mathbb{P}\left[  \norm{Y} >q r-\varepsilon \right]\le 1 &\le {2\sqrt{2\pi }}\exp\left(\frac{r^2}{\sigma_{\max}^2}\right)\mathbb{P}\left[  \norm{Y} >q r \right]\\&\le {2\sqrt{2\pi }}\exp\left( c_{3}\varepsilon(r+1) \right) \mathbb{P}\left[  \norm{Y} >q r \right].
\end{split}
\end{equation}
\textbf{Sub case (b):} If $\varepsilon \le 1/(10\tilde{d})$, then we  have 
\[
\mathbb{P}\left[ \norm{Y} >q r \right] \ge 1 - \mathbb{P}\left(\norm{Y} \le \tilde{r}\right) - \mathbb{P}\left(\tilde{r} \le \norm{Y} \le \tilde{r} + 1/(10\tilde{d})\right) \ge \Psi(1)-0.1 \ge 0.05.
\]
So 
\begin{equation}\label{eq:Condition4bFail}
\mathbb{P}\left[ \norm{Y} >q r-\varepsilon \right] \le 1 \le 20 \mathbb{P}\left[ \norm{Y} >q r \right] \le 20\exp\left(r\varepsilon\right)\mathbb{P}\left(\norm{Y} >q r\right).
\end{equation}
\textbf{Condition $(v)$:} If condition $(v)$ fails, then
\[
r \le \mu + 1/(3d_1).
\]
Since $2\sigma_{\min}^{-1} \le 3d_1$, we get $1/(3d_1) \le \sigma_{\min}/2$ and so, by the bound~\eqref{eq:AntiConc2}, we get
\begin{align*}
\mathbb{P}\left(\norm{Y} \le r\right) &~\le~ \mathbb{P}\left(\norm{Y} \le \mu + \frac{1}{3d_1}\right)\\
&~=~ \mathbb{P}\left(\norm{Y} \le \mu\right) + \mathbb{P}\left(\mu \le \norm{Y} \le \mu + \frac{1}{3d_1}\right)\\
&~=~ \frac{1}{2} + \mathbb{P}\left(\mu \le \norm{Y} \le \mu + \frac{1}{3d_1}\right) \le \frac{1}{2} + \frac{1}{3}.
\end{align*}
Thus, $6\mathbb{P}\left(\norm{Y} >q r\right) \ge 1$ and so,
\begin{equation}\label{eq:Condition5Fail}
\mathbb{P}\left(\norm{Y} >q r - \varepsilon\right) \le 1 \le 
6\mathbb{P}\left( \norm{Y} >q r \right) \le 6\exp\left(r\varepsilon\right)\mathbb{P}\left(\norm{Y} >q r\right).
\end{equation}
Combining inequalities~\eqref{eq:Condition1Fail}, \eqref{eq:Condition2aFail}, \eqref{eq:COndition2bFail}, \eqref{eq:Condition4aFail}, \eqref{eq:Condition4bFail} and~\eqref{eq:Condition5Fail}, we get
\begin{equation}\label{eq:ConditionaFail}
\mathbb{P}\left(\norm{Y} >q r - \varepsilon\right) \le 20\exp\left(\Phi_4\varepsilon(r + 1)\right)\mathbb{P}\left(\norm{Y} >q r\right),
\end{equation}
where
\[
\Phi_4 := \max\left\{1, \frac{56(\mu + 1.5\sigma_{\max})(\mu + 4.1\sigma_{\max})}{\sigma_{\max}^2\sigma_{\min}^2},\frac{4}{\sigma_{\max}^2}\right\}
\]
Now that the result is proved if one of conditions $(i)-(v)$ fail, we proceed to proving the result under all the conditions $(i)-(v)$. Take $\alpha= {\varepsilon}{(r-\mu)^{-1}}$ and $\beta = 1-\alpha$. From condition $(iii)$ we have $0\le\alpha\le 1$. Observe that $\alpha \mu + \beta r = r-\varepsilon$. Since $q$ is concave, we have 
\begin{equation*}
\begin{split}
&q(r-\varepsilon)= q(\alpha \mu + \beta r) \ge \alpha q(\mu) + \beta q(r) = \beta q(r)\quad\Rightarrow\quad q(r) \le \beta^{-1}q(r - \varepsilon).
\end{split}
\end{equation*}
Using this inequality along with the definition of $q(\cdot)$, we get that
\begin{equation}\label{eq:prepsbddI}
\begin{split}
\mathbb{P}\left(\norm{Y} >q r - \varepsilon\right) &= \Psi\left(q(r - \varepsilon)\right)\\ &= \frac{\mathbb{P}\left(\norm{Y} >q r\right)}{\Psi(q(r))}\Psi(q(r - \varepsilon)) \le \frac{\Psi(q(r - \varepsilon))}{\Psi(\beta^{-1}q((r - \varepsilon)))}\mathbb{P}\left(\norm{Y} >q r\right).
\end{split}
\end{equation}
Note from Mill's ratio \citep{MR0005558} that for $t \ge 1,$
\begin{equation}\label{eq:mill'sratio}
\frac{1}{2t}\frac{1}{\sqrt{2\pi}}\exp\left(-\frac{t^2}{2}\right) \le \frac{t}{1+t^{2}}\frac{1}{\sqrt{2\pi}} \exp\left( -\frac{t^2}{2} \right) ~\le~ \Psi(t) ~\le~ \frac{1}{t} \frac{1}{\sqrt{2\pi}} \exp\left( -\frac{t^2}{2} \right)
\end{equation}
From condition $(iv)$, $q(r-\varepsilon)\ge 1$ and so, $\beta^{-1}q(r - \varepsilon) \ge 1$. As a consequence 
\begin{equation}\label{eq:prepsbddII}
\begin{split}
\frac{\Psi(q(r-\varepsilon))}{ \Psi\left( \beta^{-1}q(r-\varepsilon) \right)} & \le \frac{2}{\beta}\frac{ \exp\left( -q^2(r-\varepsilon)/2 \right)}{\exp\left( -\beta^{-2}{q(r-\varepsilon)^2}/{2} \right)} = \frac{2}{\beta}\exp\left( \frac{q^2(r-\varepsilon)}{2} \left( \frac{1}{\beta^2}-1 \right) \right).
\end{split}
\end{equation}
Now we prove that $\beta \ge {1}/{2}$. We have 
\[
\beta - \frac{1}{2} = \frac{1}{2}- \frac{\varepsilon}{r-\mu}=\frac{r-\mu-2\varepsilon}{2(r-\mu)} \ge 0
\]
from condition $(ii)$ and $(iii)$. Plugging this in \eqref{eq:prepsbddII} we have 
\begin{equation}\label{eq:RatioBound}
\frac{\Psi(q(r-\varepsilon))}{ \Psi\left( \beta^{-1}{q(r-\varepsilon)} \right)} \le 4 \exp\left( \frac{q^2(r-\varepsilon)}{2} \left( \frac{1}{\beta^2}-1 \right) \right).
\end{equation}
So it is enough to bound
\[
\exp\left( \frac{q^2(r-\varepsilon)}{2} \left( \frac{1}{\beta^2}-1 \right) \right).
\]
From condition $(iii)$ we know $r-\varepsilon\ge \mu$. Since $q$ is a concave function, $\mu>0$ and $q(\mu)=0$, we have 
\begin{equation}\label{eq:ConcaveQBound}
q(r-\varepsilon)\le q'(\mu) (r-\varepsilon)\le q'(\mu)r.
\end{equation}
Since $Q(q(r)) = \mathbb{P}\left(\norm{Y} \le r\right),$ we get that
\[
Q'(q(r))q'(r) = \frac{d}{dr}\mathbb{P}\left(\norm{Y} \le r\right)
\]
Substituting $r = \mu$ in this equation, we get
\[
\frac{1}{\sqrt{2\pi}}q'(\mu) = \lim_{\delta\to 0}\,\frac{1}{2\delta}\mathbb{P}\left(\mu - \delta \le \norm{Y} \le \mu + \delta\right) \le \frac{d_1}{2},
\]
where the last inequality follows from inequality~\eqref{eq:AntiConc2}. This implies that $q'(\mu) \le \sqrt{2\pi}d_1/2$. Thus using~\eqref{eq:ConcaveQBound}, we obtain
\begin{equation}\label{eq:prepsbddIII}
\begin{split}
{q^2(r-\varepsilon)} \left( \frac{1}{\beta^2}-1 \right)&\le \frac{1}{2}\pi d_1^2r^{2}\left(\frac{1-\beta^{2}}{\beta^{2}}\right)\le 2\pi d_1^2r^{2}\left(1-\beta^{2}\right)\\
&= 2\pi d_1^2r^{2}\left( 1-1 -\frac{\varepsilon^{2}}{(r-\mu)^{2}}+ 2\frac{\varepsilon}{r-\mu} \right) \le  4\pi d_1^2r\varepsilon\left(\frac{r}{r-\mu}\right).
\end{split}
\end{equation}
Note from condition $(v)$ that
\begin{equation}
\begin{split}
\frac{r}{r-\mu}= 1+ \frac{\mu}{r-\mu}\le 1+ 3d_1\mu.
\end{split}
\end{equation}
Thus,
\[
\frac{q^2(r-\varepsilon)}{2} \left( \frac{1}{\beta^2}-1 \right) \le 2\pi d_1^2(1 + 3d_1\mu)\varepsilon(r + 1).
\]
Plugging this in~\eqref{eq:RatioBound}, we have 
\[
\mathbb{P}\left[ \norm{Y} >q r-\varepsilon \right] \le 4 \exp\left( 2\pi d_1^2(1 + 3d_1\mu)(r+1)\varepsilon \right) \mathbb{P}\left[ \norm{Y} >q r \right].
\]
Combining this with inequality~\eqref{eq:ConditionaFail} that holds if one of conditions $(i)-(v)$ fail, we get
\[
\mathbb{P}\left(\norm{Y} >q r - \varepsilon\right) \le 20\exp\left(\Phi_4(r + 1)\varepsilon\right)\mathbb{P}\left(\norm{Y} >q r\right),
\]
where $\Phi_4$ is redefined as
\[
\Phi_4 := \max\left\{1,\, \frac{56(\mu + 1.5\sigma_{\max})(\mu + 4.1\sigma_{\max})}{\sigma_{\max}^2\sigma_{\min}^2},\, 2\pi d_1^2(1 + 3d_1\mu), \frac{4}{\sigma_{\max}^{2}}\right\}.
\]
Since
\[
1 + 3d_1\mu = 1+ 12 \sigma_{\min}^{-1}(2.6 + \mu \sigma_{\min}^{-1})\mu \le 1 + 32 \sigma_{\min}^{-1} \mu + 12 \mu^{2} \sigma_{\min}^{-2},
\]
the result follows. \qed

\noindent 
\textbf{Proof of part $(3)$:}
We follow the notation from the proof of part $(2)$ and consider two cases:
\begin{center}
\begin{inparaenum}[(i)]
\item $q({r - \varepsilon}) \le 1$, \quad and\quad
\item 
$
q(r - \varepsilon) \ge 1.
$
\end{inparaenum}
\end{center}
Under case $(i)$, $\mathbb{P}(\norm{Y} >q r - \varepsilon) \ge \Psi(1)$ and $r - \varepsilon \le \tilde{r} \le \mu + 1.5\sigma_{\max}$. Here $\tilde{r} = q^{-1}(1)$. Recall the function $K(\cdot)$ from Lemma~\ref{lem:AntiConc2}. Also, from Lemma~\ref{lem:AntiConc2}, it follows that for $\varepsilon \le \sigma_{\min}/4$,
\begin{align*}
\mathbb{P}\left(r - \varepsilon \le \norm{Y} \le r + \varepsilon\right) &\le \mathbb{P}\left((r - \varepsilon) - 2\varepsilon \le \norm{Y} \le (r - \varepsilon) + 2\varepsilon\right)\\ &\le K(r - \varepsilon)4\varepsilon \le 4K(\mu + 1.5\sigma_{\max})\varepsilon(r + 1)\\ &\le 4K(\mu + 1.5\sigma_{\max})\varepsilon(r + 1)\frac{\mathbb{P}\left(\norm{Y} >q r - \varepsilon\right)}{\Psi(1)}.
\end{align*}
Thus, for $r$ satisfying $q(r - \varepsilon) \le 1$ and $\varepsilon \le \sigma_{\min}/4$,
\begin{equation}\label{eq:CaseIa}
\mathbb{P}\left(r - \varepsilon \le \norm{Y} \le r + \varepsilon\right) \le \frac{4K(\mu + 1.5\sigma_{\max})}{\Psi(1)}\varepsilon(r + 1)\mathbb{P}\left(\norm{Y} >q r - \varepsilon\right).
\end{equation}
If $\varepsilon \ge \sigma_{\min}/4$, then
\begin{equation}\label{eq:CaseIb}
\mathbb{P}\left(r - \varepsilon \le \norm{Y} \le r + \varepsilon\right) \le \mathbb{P}\left(\norm{Y} >q r - \varepsilon\right) \le \frac{4\varepsilon(r + 1)}{\sigma_{\min}}\mathbb{P}\left(\norm{Y} >q r - \varepsilon\right).
\end{equation}

In order to verify~\eqref{eq:ToProve} under case $(ii)$, note that for any $z \ge 0,$
\[
Q(q(z)) = \mathbb{P}\left(\norm{Y} \le z\right)\quad\Rightarrow Q'(q(z))q'(z) = p(z),
\]
where $p(z)$ represents the density of $\norm{Y}$. Since $Q$ represents the distribution function of a standard normal random variable, we get
\[
p(z) = \frac{q'(z)}{\sqrt{2\pi}}\exp\left(-\frac{q^2(z)}{2}\right).
\]
So,
\begin{align*}
\mathbb{P}\left(\norm{Y} >q r - \varepsilon\right) &= \int_{r - \varepsilon}^{\infty} \frac{q'(z)}{\sqrt{2\pi}}\exp\left(-\frac{q^2(z)}{2}\right)dz\\
&= \int_{q(r - \varepsilon)}^{\infty} \frac{1}{\sqrt{2\pi}}\exp\left(-\frac{y^2}{2}\right)dy\\
&{\ge} \frac{1}{2\sqrt{2\pi}\,q(r - \varepsilon)}\exp\left(-\frac{q^2(r - \varepsilon)}{2}\right),
\end{align*}
where the last inequality follows from Mill's ratio and the fact that under case $(ii)$, $q(r - \varepsilon) \ge 1$. Since $q(\cdot)$ is increasing,
\[
q(r - \varepsilon) \ge 1\quad\Rightarrow\quad r - \varepsilon \ge q^{-1}(1) \ge q^{-1}(0) = \mu.
\]
Since $q(\cdot)$ is concave, this implies that $q'(r - \varepsilon) \le q'(\mu)$. Thus, for all $z \ge r - \varepsilon$,
\begin{align*}
\frac{p(z)}{q'(\mu)} &= \frac{q'(z)}{\sqrt{2\pi}q'(\mu)}\exp\left(-\frac{q^2(z)}{2}\right)\\ &\le \frac{1}{\sqrt{2\pi}}\exp\left(-\frac{q^2(z)}{2}\right)~\le~ 2q(r - \varepsilon)\mathbb{P}\left(\norm{Y} >q r - \varepsilon\right).
\end{align*}
Summarizing the inequalities above, we obtain
\[
p(z) \le 2q'(\mu)q(r - \varepsilon)\mathbb{P}\left(\norm{Y} >q r - \varepsilon\right).
\]
Using concavity of $q(\cdot)$ and the fact $q(\mu) = 0$, we get $q(z) \le q'(\mu)(z - \mu) \le q'(\mu)z$ and so,
\[
p(z) \le 2\left(q'(\mu)\right)^2z\mathbb{P}\left(\norm{Y} >q r - \varepsilon\right),\quad\mbox{for all}\quad z \ge r - \varepsilon.
\]
Observe now that
\begin{equation}\label{eq:CaseII}
\begin{split}
\mathbb{P}\left(r - \varepsilon \le \norm{Y} \le r + \varepsilon\right) &= \int_{r - \varepsilon}^{r + \varepsilon} p(z) dz\le 2\left(q'(\mu)\right)^2\mathbb{P}\left(\norm{Y} >q r - \varepsilon\right)\left(2r\varepsilon\right)\\ 
&\le 4\left(q'(\mu)\right)^2\varepsilon(r + 1)\mathbb{P}\left(\norm{Y} >q r - \varepsilon\right)
\end{split}
\end{equation}
Combining inequalities~\eqref{eq:CaseIa}, \eqref{eq:CaseIb} and~\eqref{eq:CaseII}, we get for all $r \ge 0$ and $\varepsilon \ge 0$,
\[
\mathbb{P}\left(r - \varepsilon \le \norm{Y} \le r + \varepsilon\right) \le \max\left\{\frac{4K(\mu + 1.5\sigma_{\max})}{\Psi(1)}, \frac{4}{\sigma_{\min}}, 4(q'(\mu))^2\right\}\varepsilon(r + 1)\mathbb{P}\left(\norm{Y} >q r - \varepsilon\right). 
\]
Since $q'(\mu) = \sqrt{2\pi}p(\mu) \le \sqrt{2\pi}K(\mu)$, and using the form of $K(\lambda)$ from Lemma~\ref{lem:AntiConc2} the result follows. \qed
\subsection{Proof of Theorem~\ref{thm:AntiConcentration}}
Note that it is enough to prove
\begin{equation}\label{eq:LimSupStatement}
\limsup_{\varepsilon \to 0}\,\frac{r^m\mathbb{P}(r - \varepsilon \le \|Y\| \le r + \varepsilon)}{\varepsilon} \le \Phi_{AC, m}\quad\mbox{for all}\quad m\ge 0, r \ge 0.
\end{equation}
Take $\varepsilon \le \min\{\sigma_{\min}/2, \mu/4\}$. By Lemma S.1.1, we get
\[
\mathbb{P}(r - \varepsilon \le \|Y\| \le r + \varepsilon) ~\le~ \frac{4\varepsilon}{\sigma_{\min}^2}(2.6\sigma_{\min} + r).
\]
If $r \le 3(\mu + \sigma_{\max})$ then for any $m\ge 0$
\begin{equation}\label{eq:Case1}
\frac{r^m\mathbb{P}(r - \varepsilon \le \|Y\| \le r + \varepsilon)}{\varepsilon} ~\le~ \frac{3^{m+1}(\mu + \sigma_{\max})^m(2\sigma_{\max} + \mu)}{\sigma_{\min}^2} \le \frac{3^{m+1}(\mu + 2\sigma_{\max})^{m+1}}{\sigma_{\min}^2}.
\end{equation}
If $r > 3(\mu + \sigma_{\max})$, then from $\varepsilon \le \min\{\sigma_{\min}/2, \mu/4\}$ we get that
\begin{align*}
r - \varepsilon - \mu &\ge r/3 + 2r/3 - 5\mu/4\\ 
&> r/3 + 2\mu + 2\sigma_{\max} - 5\mu/4 > r/3 + 3(\mu + \sigma_{\max})/4.
\end{align*}
Therefore,
\begin{align*}
\mathbb{P}(\|Y\| > r - \varepsilon) &= \mathbb{P}(\|Y\| - \mu > r - \varepsilon - \mu)\\
&\le \mathbb{P}(\|Y\| - \mu > r/3 + 3(\mu + \sigma_{\max})/4)\\
&\le 2\exp\left(-\frac{(r/3 + 3(\mu + \sigma_{\max})/4)^2}{2\sigma_{\max}^2}\right),
\end{align*}
where the last inequality follows from Lemma 3.1 of~\cite{MR2814399}. We know from Theorem 2.1 part (3) that
\[
\mathbb{P}(r - \varepsilon \le \|Y\| \le r + \varepsilon) \le \Phi_2\varepsilon(r + 1)\mathbb{P}(\|Y\| > r - \varepsilon),
\]
and hence
\begin{align*}
&\frac{r^m\mathbb{P}(r - \varepsilon \le \|Y\| \le r + \varepsilon)}{\varepsilon}\\ 
&\qquad\le r^m\Phi_2(r + 1)\mathbb{P}\left(\|Y\| - \mu > r/3 + 3(\mu + \sigma_{\max})/4\right)\\
&\qquad\le 2\Phi_2r^m(r + 1)\exp\left(-\frac{r^2}{18\sigma_{\max}^2}\right)\exp\left(-\frac{9(\mu + \sigma_{\max})^2}{32\sigma_{\max}^2}\right)\\
&\qquad\le 2\left[\Phi_2\exp\left(-\frac{9(\mu + \sigma_{\max})^2}{32\sigma_{\max}^2}\right)\right]\times\left[r^m(r + 1)\exp\left(-\frac{r^2}{18\sigma_{\max}^2}\right)\right].
\end{align*}
Since $\Phi_2 \le C\max\{(\mu + \sigma_{\max})/\sigma_{\min}^2, (\mu + \sigma_{\max})^2/\sigma_{\min}^4\}$ for some universal constant $0 < C < \infty$, we get that
\[
\Phi_2\exp\left(-\frac{9(\mu + \sigma_{\max})^2}{32\sigma_{\max}^2}\right) \le \frac{C\sigma_{\max}^2(1 + \sigma_{\min})}{\sigma_{\min}^4},
\]
for some other universal constant $C > 0$. Further a similar argument implies
\[
r^m(r + 1)\exp\left(-\frac{r^2}{18\sigma_{\max}^2}\right) \le C_m(\sigma_{\max}^m + \sigma_{\max}^{m+1}),
\]
for some constant $0 < C_m < \infty$ depending only on $m$. Therefore if $r > 3(\mu + \sigma_{\max})$ then
\[
\frac{r^m\mathbb{P}(r - \varepsilon \le \|Y\| \le r + \varepsilon)}{\varepsilon} \le C_m\frac{\sigma_{\max}^{m+2}(1 + \sigma_{\max})^2}{\sigma_{\min}^4},
\]
for some constant $C_m > 0$ depending only on $m$. Combining this inequality with~\eqref{eq:Case1}, we get
\[
\frac{r^m\mathbb{P}(r - \varepsilon \le \|Y\| \le r + \varepsilon)}{\varepsilon} \le \max\left\{\frac{3^{m+1}(\mu + 2\sigma_{\max})^{m+1}}{\sigma_{\min}^2},\,C_m\frac{\sigma_{\max}^{m+2}(1 + \sigma_{\max})^2}{\sigma_{\min}^4}\right\},
\]
for $\varepsilon \le \min\{\sigma_{\min}/2, \mu/4\}$. This completes the proof of~\eqref{eq:LimSupStatement} and of the result.
\section{Proof of Lemma~\ref{lem:SmoothApproximation}}
Define for any $x\in\mathbb{R}^p$, $z_x = (x^{\top}\,:\,-x^{\top})^{\top}$. The softmax function satisfies
\[
\max_j |x(j)| \le F_{\beta}(z_x) \le \max_j |x(j)| + \frac{\log(2p)}{\beta}\quad\mbox{for any}\quad z\in\mathbb{R}^p.
\]
Thus if $|x(j)| \le r$ for all $1\le j\le p$ then
\[
F_{\beta}(z_x - r\mathbf{1}_{2p}) \le \max_{1\le j\le p}\max\{x(j) - r, -x(j) - r\} + \frac{\log(2p)}{\beta} \le \frac{\log(2p)}{\beta} = \frac{\varepsilon}{2},
\]
and hence $F_{\beta}(z_x - r\mathbf{1}_{2p}) - {\varepsilon}/{2} \le 0$ which in turn implies that
\[
\varphi(x) = 1\quad\mbox{since}\quad g_0(t) = 1\quad\mbox{for}\quad t \le 0.
\]
On the other hand if $|x(j)| > r + \varepsilon$ for some $1\le j\le p$ then
\[
F_{\beta}(z_x - r\mathbf{1}_{2p}) \ge \max_{1\le j\le p}\max\{x(j) - r, -x(j) - r\} > \varepsilon\quad\Rightarrow\quad \frac{2}{\varepsilon}\left(F_{\beta}(z_x - r\mathbf{1}_{2p}) - \frac{\varepsilon}{2}\right) \ge 1,
\]
which in turn implies that
\[
\varphi(x) = 0\quad\mbox{since}\quad g_0(t) = 0\quad\mbox{for}\quad t \ge 1.
\]
This completes the proof of first statement. The remaining two statements can be easily verified by direct calculation. See~\citet[Appendix A]{CCK13} for details.
\section{Proofs of Results in Section~\ref{sec:UniformCLT} and~\ref{sec:NonUniformCLT}}
\subsection{Preliminary Results}
\begin{lem}\label{lem:RecursiveIneq}
Suppose $\{\rchi_n\}$ be a sequence such that
\[
\rchi_n \le \kappa\rchi_{n-1} + a_n\quad\mbox{for}\quad n \ge 2,
\]
for some $\kappa < 1$. If $a_{n-1}/a_n \le \bar{C}$ for all $n\ge2$ for a constant $\bar{C}$ such that $\kappa \bar{C} < 1$, then
\[
\rchi_n \le \kappa^{n-1}\rchi_1 + a_n\left(\frac{1}{1 - \kappa \bar{C}}\right).
\] 
\end{lem}
\begin{proof}
Hypothesize that $\rchi_n \le \rchi_1\kappa^{n-1} + Ca_n$ for some $C = (1 - \kappa \bar{C})^{-1}$. This is trivially true for $n = 1$. Suppose true up to $n-1$. Then for $n$,
\begin{align*}
\rchi_n &\le \kappa(\rchi_1\kappa^{n-2} + Ca_{n-1}) + a_n = \rchi_1\kappa^{n-1} + \kappa Ca_{n-1} + a_n\\ 
&\le \rchi_1\kappa^{n-1} + (\kappa C\bar{C} + 1)a_n = \rchi_1\kappa^{n-1} + a_n(1 + \kappa\bar{C}(1 - \kappa \bar{C})^{-1})\\
&= \rchi_1\kappa^{n-1} + Ca_n.
\end{align*}
This completes the proof.
\end{proof}
\noindent Define
\[
\Delta_{n,m}(r) = \max_{1\le k\le n}\,r^m|\mathbb{P}(\|U_{n,k}\| \le r) - \mathbb{P}(\|U_{n,0}\| \le r)|.
\]
In this section, we prove a general bound for $\Delta_{n,m}(r)$ for a fixed $r$. The uniform and non-uniform CLTs follow by massaging this result.
\begin{lem}\label{lem:GeneralmChern}
Fix $r, \varepsilon > 0$ and $m\ge 0$ such that $r \ge 4\varepsilon$ if $m > 0$. Then
\begin{align*}
\Delta_{n,m}(r) &\le \frac{e^{2\mathfrak{C}}C_3\varepsilon^{-3}\left[{2\Phi_{AC, m}\varepsilon} + {4^{m}\delta_{n,m}}\right]L_n}{n^{1/2}}\\
&\quad+ \frac{e^{\mathfrak{C}}C_3\varepsilon^{-3}r^m}{2n^{3/2}}\sum_{j=1}^n\max_{1\le j_1\le p}\int |x(j_1)|^3|\zeta_j|(dx)\mathbb{P}(\|Y_j\| > n^{1/2}\varepsilon/\log(ep))\\
&\quad+ {C_2\varepsilon^{-2}r^m}M_n(\varepsilon) + \Phi_{AC,m}\varepsilon,
\end{align*}
where $C_2 := C_0\log(ep)$ and $C_3 := C_0\log^2(ep)$.
\end{lem}
\begin{proof}
By the use of smoothing lemma, we get
\begin{equation}\label{eq:SmoothingIneqChern}
\Delta_{n,m}(r) \le \max_{1\le k\le n}\max_{j = 1,2}\,r^m|\mathbb{E}\left[\varphi_{j}(U_{n,k}) - \varphi_j(U_{n,0})\right]| + r^m\mathbb{P}\left(r - \varepsilon \le \|U_{n,0}\| \le r + \varepsilon\right),
\end{equation}
where $\varphi_1(x) = \varphi_{r,\varepsilon}(x)$ and $\varphi_2(x) = \varphi_{r-\varepsilon, \varepsilon}(x)$. From now we write $\varphi$ to represent either $\varphi_1$ or $\varphi_2$. It is clear that for any $1\le k\le n$, we have
\begin{align*}
|\mathbb{E}[\varphi(U_{n,k}) - \varphi(U_{n,0})]| &\le \sum_{j=1}^k \left|\mathbb{E}\left[\varphi(U_{n,j}) - \varphi(U_{n,j-1})\right]\right|\\
&\le \sum_{j=1}^k \left|\mathbb{E}\left[\varphi(W_{n,j} + n^{-1/2}X_j) - \varphi(W_{n,j} + n^{-1/2}Y_j)\right]\right|\\
&\le \sum_{j=1}^k \left|\int \mathbb{E}[\varphi(W_{n,j} + n^{-1/2}x)d\zeta_j(x)]\right|.
\end{align*}
By Taylor series expansion, we get that
\[
\varphi(W_{n,j} + n^{-1/2}x) = \varphi(W_{n,j}) + n^{-1/2}x^{\top}\nabla \varphi(W_{n,j}) + \frac{1}{2n}x^{\top}\nabla_2\varphi(W_{n,j})x + \mbox{Rem}_n(W_{n,j}, x).
\]
Combining above inequalities with Equation~\eqref{eq:zetazero}
\begin{equation}\label{eq:FirstDecompositionChern}
\begin{split}
|\mathbb{E}[\varphi(U_{n,k}) - \varphi(U_{n,0})]| &\le \sum_{j = 1}^k \left|\int \mathbb{E}\left[\mbox{Rem}_n(W_{n,j}, x)\right]d\zeta_j(x)\right|\\
&\le \sum_{j=1}^k \left|\mathbb{E}\int \mathbbm{1}\{x\in\mathcal{E}\}\mbox{Rem}_n(W_{n,j}, x)d\zeta_j(x)\right|\\
&\qquad+ \sum_{j=1}^k \left|\mathbb{E}\int \mathbbm{1}\{x\in\mathcal{E}^c\}\mbox{Rem}_n(W_{n,j}, x)d\zeta_j(x)\right| = \mathbf{I} + \mathbf{II},
\end{split}
\end{equation}
where $\mathcal{E} := \left\{\|x\|_{\infty} \le {n^{1/2}\varepsilon}/{\log(ep)}\right\}.$ To bound $\mathbf{I}$, observe that
\[
\mbox{Rem}_n(W_{n,j}, x) = \frac{1}{2}\sum_{j_1,j_2,j_3 = 1}^p \frac{x(j_1)x(j_2)x(j_3)}{n^{3/2}}\int_0^1 (1 - t)^2\partial_{j_1,j_2,j_3}\varphi(W_{n,j} + tn^{-1/2}x)dt.
\]
Noting the support of $\varphi$ is $\{r - \varepsilon \le \|x\| \le r + \varepsilon\}$, $|\partial_{j_1,j_2,j_3}\varphi(W_{n,j} + tn^{-1/2}x)| \le D_{j_1j_2j_3}(W_{n,j} + tn^{-1/2}x)$ and the stability property of $D_{j_1j_2j_3}$ we get that
\begin{align*}
&|\mbox{Rem}_n(W_{n,j}, x)|\mathbbm{1}\{x\in\mathcal{E}\}\\ 
&\le \frac{1}{2n^{3/2}}\sum_{j_1,j_2,j_3 = 1}^p {|x(j_1)x(j_2)x(j_3)|}\int_0^1 D_{j_1j_2j_3}(W_{n,j} + tn^{-1/2}x)\mathbbm{1}\{|\|W_{n,j} + tn^{-1/2}x\| - r| \le \varepsilon, x\in\mathcal{E}\}dt\\
&\le \frac{e^\mathfrak{C}}{2n^{3/2}}\sum_{j_1,j_2,j_3 = 1}^p {|x(j_1)x(j_2)x(j_3)|}D_{j_1j_2j_3}(W_{n,j})\mathbbm{1}\{|\|W_{n,j}\| - r| \le \varepsilon + \varepsilon/\log(e p), x\in\mathcal{E}\},
\end{align*}
which implies that
\begin{equation}\label{eq:SomeToAllIntegral}
\begin{split}
\mathbf{I} &\le \frac{e^{\mathfrak{C}}}{2n^{3/2}}\sum_{\substack{1\le j\le k,\\1\le j_1,j_2,j_3 \le p}} \int_{x\in\mathcal{E}} |x(j_1)x(j_2)x(j_3)||\zeta_j|(dx)\mathbb{E}\left[D_{j_1j_2j_3}(W_{n,j})\mathbbm{1}\{|\|W_{n,j}\| - r| \le 2\varepsilon\}\right]\\
&\le \frac{e^{\mathfrak{C}}}{2n^{3/2}}\sum_{\substack{1\le j\le k,\\1\le j_1,j_2,j_3 \le p}} \max_{1\le j_4\le p}\int |x(j_4)|^3|\zeta_j|(dx)\mathbb{E}\left[D_{j_1j_2j_3}(W_{n,j})\mathbbm{1}\{|\|W_{n,j}\| - r| \le 2\varepsilon\}\right]\\
&\le \frac{e^{\mathfrak{C}}}{2n^{3/2}}\sum_{\substack{1\le j\le k}} \max_{1\le j_1\le p}\int |x(j_1)|^3|\zeta_j|(dx)\sum_{1\le j_1,j_2,j_3\le p}\mathbb{E}\left[D_{j_1j_2j_3}(W_{n,j})\mathbbm{1}\{|\|W_{n,j}\| - r| \le 2\varepsilon\}\right].
\end{split}
\end{equation}
We will now relate the right hand side in terms of $\Delta_{n,m}(r)$. Observe that for any $1\le j\le k,$
\begin{align*}
\mathbb{E}\left[D_{j_1j_2j_3}(W_{n,j})\mathbbm{1}\{|\|W_{n,j}\| - r| \le 2\varepsilon\}\right] &= \mathbb{E}\left[D_{j_1j_2j_3}(W_{n,j})\mathbbm{1}\{|\|W_{n,j}\| - r| \le 2\varepsilon, Y_j\in\mathcal{E}\}\right]\\
&\qquad+ \mathbb{E}\left[D_{j_1j_2j_3}(W_{n,j})\mathbbm{1}\{|\|W_{n,j}\| - r| \le 2\varepsilon, Y_j\in\mathcal{E}^c\}\right].
\end{align*}
If $Y_j\in\mathcal{E}$ and $r - 2\varepsilon \le \|W_{n,j}\| \le r + 2\varepsilon$ then $r - 3\varepsilon \le \|U_{n,j-1}\| \le r + 3\varepsilon,$
and
\[
e^{-\mathfrak{C}}D_{j_1j_2j_3}(W_{n,j} + n^{-1/2}Y_j) \le D_{j_1j_2j_3}(W_{n,j}) \le e^{\mathfrak{C}}D_{j_1j_2j_3}(W_{n,j} + n^{-1/2}Y_j).
\]
Since $W_{n,j} + n^{-1/2}Y_j = U_{n,j-1}$, we get
\begin{align*}
&\mathbb{E}\left[D_{j_1j_2j_3}(W_{n,j})\mathbbm{1}\{|\|W_{n,j}\| - r| \le 2\varepsilon, Y_j\in\mathcal{E}\}\right]\\ 
&\qquad\le e^{\mathfrak{C}}\mathbb{E}\left[D_{j_1j_2j_3}(U_{n,j-1})\mathbbm{1}\{r - 3\varepsilon \le \|U_{n,j-1}\| \le r + 3\varepsilon\}\right]\\
&\qquad\le e^{\mathfrak{C}}\mathbb{E}\left[D_{j_1j_2j_3}(U_{n,j-1})\mathbbm{1}\{r - 3\varepsilon \le \|U_{n,j-1}\| \le r + 3\varepsilon\}\right]
\end{align*}
Summing over $1\le j_1, j_2, j_3 \le p$ we get
\begin{align*}
&\sum_{1\le j_1,j_2,j_3\le p}\mathbb{E}\left[D_{j_1j_2j_3}(W_{n,j})\mathbbm{1}\{|\|W_{n,j}\| - r| \le 2\varepsilon, Y_j\in\mathcal{E}\}\right]\\
&\quad\le e^{\mathfrak{C}}\sum_{1\le j_1,j_2,j_3 \le p} \mathbb{E}\left[D_{j_1j_2j_3}(U_{n,j-1})\mathbbm{1}\{r - 3\varepsilon \le \|U_{n,j-1}\| \le r + 3\varepsilon\}\right]\\
&\quad\le e^{\mathfrak{C}}\max_{z}\sum_{1\le j_1,j_2,j_3\le p}D_{j_1j_2j_3}(z)\mathbb{P}\left(r - 3\varepsilon \le \|U_{n,j-1}\| \le r + 3\varepsilon\right)\\
&\quad\le e^{\mathfrak{C}}C_3\varepsilon^{-3}\mathbb{P}\left(r - 3\varepsilon \le \|U_{n,j-1}\| \le r + 3\varepsilon\right).
\end{align*}
The probability on the right hand side can be approximated by $\mathbb{P}(r - 3\varepsilon \le \|U_{n,0}\| \le r + 3\varepsilon)$ using the definition of $\delta_{n,m}$:
\begin{align*}
\mathbb{P}(r - 3\varepsilon \le \|U_{n,j-1}\| \le r + 3\varepsilon) &\le \mathbb{P}\left(r - 3\varepsilon\le \|U_{n,0}\| \le r + 3\varepsilon\right)\\ &\qquad+ \delta_{n,m}[(r - 3\varepsilon)^{-m} + (r + 3\varepsilon)^{-m}]\\
&\quad\le \frac{3\Phi_{AC,m}\varepsilon}{r^m} + \frac{2\delta_{n,m}}{(r - 3\varepsilon)^m} ~\le~ \frac{3\Phi_{AC,m}\varepsilon}{r^m} + \frac{4^{m}2\delta_{n,m}}{r^m}.
\end{align*}
To bound $\mathbb{E}[D_{j_1j_2j_3}(W_{n,j})\mathbbm{1}\{|\|W_{n,j}\| - r| \le 2\varepsilon, Y_j\in\mathcal{E}^c\}]$, we note that
\begin{align*}
&\sum_{1\le j_1,j_2,j_3 \le p}\mathbb{E}[D_{j_1j_2j_3}(W_{n,j})\mathbbm{1}\{|\|W_{n,j}\| - r| \le 2\varepsilon, Y_j\in\mathcal{E}^c\}]\\
&\qquad\le \max_z\sum_{1\le j_1,j_2,j_3\le p} D_{j_1j_2j_3}(z)\mathbb{P}(Y_j\in \mathcal{E}^c) ~\le~ C_3\varepsilon^{-3}\mathbb{P}(Y_j\in\mathcal{E}^c). 
\end{align*}
Combining the bounds above, we get
\begin{align*}
\mathbf{I} &\le \frac{e^{2\mathfrak{C}}C_3\varepsilon^{-3}}{n^{3/2}}\left[\frac{2\Phi_{AC, m}\varepsilon}{r^m} + \frac{4^{m}\delta_{n,m}}{r^m}\right]\sum_{1\le j\le k}\max_{1\le j_1\le p}\int|x(j_1)|^3|\zeta_{j}|(dx)\\
&\quad+ \frac{e^{\mathfrak{C}}C_3\varepsilon^{-3}}{2n^{3/2}}\sum_{1\le j\le k}\max_{1\le j_1\le p}\int |x(j_1)|^3|\zeta_j|(dx)\mathbb{P}(Y_j\in\mathcal{E}^c).
\end{align*}
We now bound $\mathbf{II}$ in~\eqref{eq:FirstDecompositionChern}. For this, we use the bound
\[
\left|\mbox{Rem}_n(W_{n,j}, x)\right| \le \frac{C_2\varepsilon^{-2}\|x\|^2}{n},
\]
and get
\begin{align*}
\mathbf{II} &\le \sum_{j=1}^k \mathbb{E}\left[\int \frac{C_2\varepsilon^{-2}\|x\|^2}{n}\mathbbm{1}\{x\in\mathcal{E}^c\}|\zeta_j|(dx)\right]\\
&\le \frac{C_2\varepsilon^{-2}}{n}\sum_{j=1}^k \int \|x\|^2\mathbbm{1}\{x\in\mathcal{E}^c\}|\zeta_j|(dx).
\end{align*}
Collecting the bounds above and recalling the smoothing inequality~\eqref{eq:SmoothingIneqChern}, we get
\begin{align*}
\Delta_{n,m}(r) &\le \frac{e^{2\mathfrak{C}}C_3\varepsilon^{-3}\left[{2\Phi_{AC, m}\varepsilon} + {4^{m}\delta_{n,m}}\right]}{n^{3/2}}\sum_{j=1}^n\max_{1\le j_1\le p}\int|x(j_1)|^3|\zeta_{j}|(dx)\\
&\quad+ \frac{e^{\mathfrak{C}}C_3\varepsilon^{-3}r^m}{2n^{3/2}}\sum_{j=1}^n\max_{1\le j_1\le p}\int |x(j_1)|^3|\zeta_j|(dx)\mathbb{P}(Y_j\in\mathcal{E}^c)\\
&\quad+ \frac{C_2\varepsilon^{-2}r^m}{n}\sum_{j=1}^n \int \|x\|^2\mathbbm{1}\{x\in\mathcal{E}^c\}|\zeta_j|(dx) + \Phi_{AC,m}\varepsilon.
\end{align*}
The result is proved.
\end{proof}
\subsection{Proof of Theorem~\ref{thm:UniformCLT}}
From Lemma~\ref{lem:GeneralmChern}, we get
\begin{align*}
\Delta_{n,m}(r) &\le \frac{e^{2\mathfrak{C}}C_3\varepsilon^{-3}\left[{2\Phi_{AC, 0}\varepsilon} + {\delta_{n,m}}\right]L_n}{n^{1/2}}\\ &\quad+ \frac{e^{\mathfrak{C}}C_3\varepsilon^{-3}}{2n^{3/2}}\sum_{j=1}^n\max_{1\le j_1\le p}\int |x(j_1)|^3|\zeta_j|(dx)\mathbb{P}(\|Y_j\| > n^{1/2}\varepsilon/\log(ep))\\
&\quad+ {C_2\varepsilon^{-2}}M_n(\varepsilon) + \Phi_{AC,0}\varepsilon.
\end{align*}
Since $\delta_{n,m} = \sup_{r \ge 0} \Delta_{n,m}(r)$, we get
\begin{align*}
\delta_{n,m} &\le \frac{e^{2\mathfrak{C}}C_3\varepsilon^{-3}\left[{2\Phi_{AC, 0}\varepsilon} + {\delta_{n,m}}\right]L_n}{n^{1/2}} + {C_2\varepsilon^{-2}}M_n(\varepsilon) + \Phi_{AC,0}\varepsilon\\
&\quad+ \frac{e^{\mathfrak{C}}C_3\varepsilon^{-3}}{2n^{3/2}}\sum_{j=1}^n\max_{1\le j_1\le p}\int |x(j_1)|^3|\zeta_j|(dx)\mathbb{P}(\|Y_j\| > n^{1/2}\varepsilon/\log(ep)).
\end{align*}
Choosing $\varepsilon = \varepsilon_n$ such that
\[
\frac{e^{2\mathfrak{C}}C_3\varepsilon^{-3}_nL_n}{n^{1/2}} = \frac{1}{2}\quad\Leftrightarrow\quad \varepsilon_n = \left(\frac{2e^{2\mathfrak{C}}C_3L_n}{n^{1/2}}\right)^{1/3},
\]
we get
\begin{align*}
\delta_{n,m} &\le \frac{\delta_{n,m}}{2} + \frac{2e^{2\mathfrak{C}}C_3\Phi_{AC,0}\varepsilon_n^{-2}L_n}{n^{1/2}} + C_2\varepsilon^{-2}_nM_n(\varepsilon_n) + \Phi_{AC,0}\varepsilon_n\\ 
&\qquad+ \frac{1}{4e^{\mathfrak{C}}nL_n}\sum_{j=1}^n\max_{1\le j_1\le p}\int |x(j_1)|^3|\zeta_j|(dx)\mathbb{P}\left(\|Y_j\| > n^{1/2}\varepsilon_n/\log(ep)\right)\\
&\le \frac{1}{2}\delta_{n,m} + \Phi_{AC,0}\varepsilon_n + C_2\varepsilon^{-2}_nM_n(\varepsilon_n) + \Phi_{AC,0}\varepsilon_n\\ 
&\qquad+ \frac{1}{4e^{\mathfrak{C}}nL_n}\sum_{j=1}^n\max_{1\le j_1\le p}\int |x(j_1)|^3|\zeta_j|(dx)\mathbb{P}\left(\|Y_j\| > n^{1/2}\varepsilon_n/\log(ep)\right).
\end{align*}
Simplifying this inequality, we get
\begin{equation}\label{eq:DeltanmIneqChern}
\begin{split}
\delta_{n,m} &\le 4\Phi_{AC,0}\varepsilon_n + 2C_2\varepsilon_n^{-2}M_n(\varepsilon_n)\\ 
&\qquad+ \frac{1}{2e^{\mathfrak{C}}nL_n}\sum_{j=1}^n\max_{1\le j_1\le p}\int |x(j_1)|^3|\zeta_j|(dx)\mathbb{P}\left(\|Y_j\| > n^{1/2}\varepsilon_n/\log(ep)\right).
\end{split}
\end{equation}
To bound the last term, consider two cases:
\[
(i)\;\frac{n^{1/2}\varepsilon_n}{\log(ep)} \le 2\mu_j\quad\mbox{and}\quad (ii)\;\frac{n^{1/2}\varepsilon_n}{\log(ep)} > 2\mu_j.
\]
\begin{enumerate}
  \item In case $(i)$, we have
  \[
  1 \le \frac{2\mu_j\log^{1/3}(ep)}{n^{1/3}(2e^{2\mathfrak{C}}C_0L_n)^{1/3}}.
  \]
  This implies that
  \begin{equation}\label{eq:ProbBound}
  \mathbb{P}(\|Y_j\| > n^{1/2}\varepsilon_n/\log(ep)) \le 1 \le \frac{2\mu_j\log^{1/3}(ep)}{n^{1/3}(2e^{2\mathfrak{C}}C_0L_n)^{1/3}}.
  \end{equation}
  \item In case $(ii)$, we have
  \begin{equation}\label{eq:MujIneq}
  \frac{n^{1/2}\varepsilon_n}{\log(ep)} \ge \frac{n^{1/2}\varepsilon_n}{2\log(ep)} + \mu_j,
  \end{equation}
  and so using $e^{x} \ge \sqrt{2x}$ (or equivalently $e^{2x} \ge 2x$), we get
  \begin{align*}
  \mathbb{P}\left(\|Y_j\| > n^{1/2}\varepsilon_n/\log(ep)\right) &\le \mathbb{P}\left(\|Y_j\| \ge \mu_j + \frac{n^{1/2}\varepsilon_n}{2\log(ep)}\right)\\
  &\le \exp\left(-\frac{n\varepsilon_n^2}{8\log^2(ep)\sigma_j^2}\right)\\
  &\le \exp\left(-\frac{n^{2/3}(2e^{2\mathfrak{C}}C_0L_n)^{2/3}}{8\sigma_j^2\log^{2/3}(ep)}\right)\\
  &\le 2^{-1/2}\left(\frac{8\sigma_j^2\log^{2/3}(ep)}{n^{2/3}(2e^{2\mathfrak{C}}C_0L_n)^{2/3}}\right)^{1/2}~\le~ \frac{2\sigma_j\log^{1/3}(ep)}{n^{1/3}(2e^{2\mathfrak{C}}C_0L_n)^{1/3}}.
  \end{align*}
  Here the second inequality follows from Lemma 3.1 of~\cite{MR2814399}.
\end{enumerate}
Combining cases $(i)$ and $(ii)$, we get
\[
\mathbb{P}\left(\|Y_j\| > n^{1/2}\varepsilon_n/\log(ep)\right) ~\le~ \frac{2(\mu_j + \sigma_j)\log^{1/3}(ep)}{n^{1/3}(2e^{2\mathfrak{C}}C_0L_n)^{1/3}}\quad\mbox{for all}\quad 1\le j\le n,
\]
and hence
\begin{align*}
&\frac{1}{2e^{\mathfrak{C}}nL_n}\sum_{j=1}^n\max_{1\le j_1\le p}\int |x(j_1)|^3|\zeta_j|(dx)\mathbb{P}\left(\|Y_j\| > n^{1/2}\varepsilon_n/\log(ep)\right)\\ 
&\quad\le \frac{1}{2e^{\mathfrak{C}}nL_n}\sum_{j=1}^n \max_{1\le j_1\le p}\int |x(j_1)|^3|\zeta_j|(dx)\frac{2(\mu_j + \sigma_j)\log^{1/3}(ep)}{n^{1/3}(2e^{2\mathfrak{C}}C_0L_n)^{1/3}}\\
&\quad\le \frac{\log^{1/3}(ep)}{n^{1/3}L_n(2e^{5\mathfrak{C}}C_0L_n)^{1/3}}\left(\frac{1}{n}\sum_{j=1}^n (\mu_j + \sigma_j)\max_{1\le j_1\le p}\int |x(j_1)|^3|\zeta_j|(dx)\right).
\end{align*}
Substituting this bound in~\eqref{eq:DeltanmIneqChern}, we have
\begin{align*}
\delta_{n,m} &\le 4\Phi_{AC,0}\varepsilon_n + 2C_2\varepsilon_n^{-2}M_n(\varepsilon_n)\\
&\quad+ \frac{\log^{1/3}(ep)}{n^{1/3}L_n(2e^{5\mathfrak{C}}C_0L_n)^{1/3}}\left(\frac{1}{n}\sum_{j=1}^n (\mu_j + \sigma_j)\max_{1\le j_1\le p}\int |x(j_1)|^3|\zeta_j|(dx)\right)
\end{align*}
Thus the result follows.
\subsection{Proof of Theorem~\ref{thm:NonUniformPolyCLT}}
Define
\begin{equation}\label{eq:EpsilonnmChoice}
\varepsilon_{n,m} := \left(\frac{2^{2m+1}e^{2\mathfrak{C}}C_0\log^2(ep)L_n}{n^{1/2}}\right)^{1/3} = 2^{2m/3}\varepsilon_n.
\end{equation}
Fix $\varepsilon, r_{n,m} > 0$. Recall
\[
\delta_{n,m} = \sup_{r\ge 0}\,\Delta_{n,m}(r) = \max\left\{\sup_{r \le 4\varepsilon_{n,m}}\,\Delta_{n,m}(r),\, \sup_{4\varepsilon_{n,m} < r < r_{n,m}}\,\Delta_{n,m}(r),\,\sup_{r \ge r_{n,m}}\,\Delta_{n,m}(r)\right\}.
\]
Define $r^{\star}$ as the maximizing radius, that is, $\delta_{n,m} = \Delta_{n,m}(r^{\star}).$ If $r^{\star} < 4\varepsilon_{n,m}$, then 
\[
\delta_{n,m} = \sup_{0 \le r < 4\varepsilon_{n,m}}\,\Delta_{n,m}(r).
\]
We now prove that for $0 \le r \le 4\varepsilon_{n,m}$
\begin{equation}\label{eq:ToProveChern}
\Delta_{n,m}(r) \le \Delta_{n,m}(4\varepsilon_{n,m}) + 2^{2m+1}\varepsilon_{n,m}^{m+1}\Phi_{AC,0}.
\end{equation}
Firstly note that $\Delta_{n,m}(r) \le 4^m\varepsilon_{n,m}^m\Delta_{n,0}(r)$. Thus it is enough to prove that $\Delta_{n,0}(r) \le \Delta_{n,0}(4\varepsilon_{n,m}) + 2\Phi_{AC,0}\varepsilon_{n,m}$. Recall that
\[
\Delta_{n,0}(r) = \max_{1\le k\le n}|\mathbb{P}(\|U_{n,k}\| \le r) - \mathbb{P}(\|U_{n,0}\| \le r)|.
\]
For any $1\le k\le n$ if $\mathbb{P}(\|U_{n,k}\| \le r) \le \mathbb{P}(\|U_{n,0}\| \le r)$ then by monotonocity of $r\mapsto \mathbb{P}(\|U_{n,0}\| \le r)$ it follows that
\[
\mathbb{P}(\|U_{n,k}\| \le r) \le \mathbb{P}(\|U_{n,0}\| \le r) \le \mathbb{P}(\|U_{n,0}\| \le 4\varepsilon_{n,m}) \le 2\varepsilon_{n,m}\Phi_{AC,0}.
\]
Under this case $\Delta_{n,0}(r) \le 2\varepsilon_{n,m}\Phi_{AC,0}$. If, otherwise, $\mathbb{P}(\|U_{n,k}\| \le r) > \mathbb{P}(\|U_{n,0}\| \le r)$ then
\begin{align*}
\Delta_{n,0}(r) &\le \max_{1\le k\le n}\,\mathbb{P}(\|U_{n,k}\| \le r) - \mathbb{P}(\|U_{n,0}\| \le r)\\
&\le \max_{1\le k\le n}\,\mathbb{P}(\|U_{n,k}\| \le 4\varepsilon_{n,m}) - \mathbb{P}(\|U_{n,0}\| \le 4\varepsilon_{n,m}) + \mathbb{P}(r \le \|U_{n,0}\| \le 4\varepsilon_{n,m})\\
&\le \Delta_{n,0}(4\varepsilon_{n,m}) + 2\varepsilon_{n,m}\Phi_{AC,0}.
\end{align*}
This completes the proof of~\eqref{eq:ToProveChern}. Therefore,
\begin{equation}\label{eq:Case1Deltanm}
\begin{split}
\delta_{n,m} &= \sup_{0 \le r < 4\varepsilon_{n,m}}\,\Delta_{n,m}(r) \le (4\varepsilon_{n,m})^m\max_{0 \le r < 4\varepsilon_{n,m}}\Delta_{n,0}(r)\\
&\le \Delta_{n,m}(4\varepsilon_{n,m}) + 2^{2m+1}\varepsilon_{n,m}^{m+1}\Phi_{AC,0}.
\end{split}
\end{equation}
where $\varepsilon_n$ is the quantity defined in Theorem~\ref{thm:UniformCLT}. If $4\varepsilon_{n,m} \le r^{\star} < r_{n,m}$, then we have 
\[
\delta_{n,m} = \sup_{4\varepsilon_{n,m} \le r < r_{n,m}}\,\Delta_{n,m}(r),
\]
and hence Lemma~\ref{lem:GeneralmChern} implies
\begin{align*}
\delta_{n,m} &\le \frac{e^{2\mathfrak{C}}C_3\varepsilon_{n,m}^{-3}\left[{2\Phi_{AC, m}\varepsilon_{n,m}} + {4^{m}\delta_{n,m}}\right]L_n}{n^{1/2}}\\
&\quad+ \frac{e^{\mathfrak{C}}C_3\varepsilon_{n,m}^{-3}r_{n,m}^m}{2n^{3/2}}\sum_{j=1}^n\max_{1\le j_1\le p}\int |x(j_1)|^3|\zeta_j|(dx)\mathbb{P}(\|Y_j\| > n^{1/2}\varepsilon_{n,m}/\log(ep))\\
&\quad+ {C_2\varepsilon_{n,m}^{-2}r^m_{n,m}}M_n(\varepsilon_{n,m}) + \Phi_{AC,m}\varepsilon_{n,m}.
\end{align*}
The choice of $\varepsilon_{n,m}$ is made so that $e^{2\mathfrak{C}}C_3\varepsilon_{n,m}^{-3}/n^{1/2} = 2^{-2m-1}/L_n$ and hence
\begin{align*}
\delta_{n,m} &\le \frac{2\Phi_{AC,m}\varepsilon_{n,m} + 4^m\delta_{n,m}}{2^{2m+1}}+ C_2\varepsilon_{n,m}^{-2}r_{n,m}^mM_n(\varepsilon_{n,m}) + \Phi_{AC,m}\varepsilon_{n,m}\\
&\quad+ \frac{r_{n,m}^m}{L_n4^{m+1}e^{\mathfrak{C}}}\frac{1}{n}\sum_{j=1}^n \max_{1\le j_1\le p}\int |x(j_1)|^3|\zeta_j|(dx)\mathbb{P}(\|Y_j\| > n^{1/2}\varepsilon_{n,m}/\log(ep)).
\end{align*}
Simplifying this inequality, we get
\begin{align*}
\delta_{n,m} &\le 2(1+4^{-m})\Phi_{AC,m}\varepsilon_{n,m} + 2C_2\varepsilon_{n,m}^{-2}r_{n,m}^mM_n(\varepsilon_{n,m})\\
&\quad+ \frac{2r_{n,m}^m}{L_n 4^{m+1}e^{\mathfrak{C}}}\frac{1}{n}\sum_{j=1}^n \max_{1\le j_1\le p}\int |x(j_1)|^3|\zeta_j|(dx)\mathbb{P}(\|Y_j\| > n^{1/2}\varepsilon_{n,m}/\log(ep)).
\end{align*}
Following the proof of Theorem~\ref{thm:UniformCLT}, we obtain
\begin{align*}
\mathbb{P}(\|Y_j\| > n^{1/2}\varepsilon_{n,m}/\log(ep)) &\le \left(\frac{2\mu_j\log(ep)}{n^{1/2}\varepsilon_{n,m}}\right)^{m+1} + \left(\frac{m+1}{2e}\right)^{(m+1)/2}\left(\frac{8\sigma_j^2\log^2(ep)}{n\varepsilon_{n,m}^2}\right)^{(m+1)/2}\\
&\le \left(\frac{2\mu_j\log(ep)}{n^{1/2}\varepsilon_{n,m}}\right)^{m+1} + \left(\frac{4(m+1)}{e}\right)^{(m+1)/2}\left(\frac{\sigma_j\log(ep)}{n^{1/2}\varepsilon_{n,m}}\right)^{m+1}\\
&\le \left[(2\mu_j)^{m+1} + (\sigma_j\sqrt{4(m+1)/e})^{m+1}\right]\left(\frac{\log(ep)}{n^{1/2}\varepsilon_{n,m}}\right)^{m+1}\\
&\le 2^{m+1}\left[\mu_j^{m+1} + \sigma_j^{m+1}\left({\frac{m+1}{e}}\right)^{(m+1)/2}\right]\left(\frac{\log^{1/3}(ep)}{n^{1/3}(2^{2m+1}e^{2\mathfrak{C}}C_0L_n)^{1/3}}\right)^{m+1}. 
\end{align*}
Therefore, if $4\varepsilon_{n,m} \le r^{\star} \le r_{n,m}$ then
\begin{equation}\label{eq:Case2Deltanm}
\begin{split}
\delta_{n,m} &\le 2(1+4^{-m})\Phi_{AC,m}\varepsilon_{n,m} + 2C_2\varepsilon_{n,m}^{-2}r_{n,m}^mM_n(\varepsilon_{n,m})\\
&\quad+ \frac{r_{n,m}^m\bar{L}_{n,m}}{L_n2^{m}e^{\mathfrak{C}}}\left(\frac{\log(ep)}{n(2^{2m+1}e^{2\mathfrak{C}}C_0L_n)}\right)^{(m+1)/3}
\end{split}
\end{equation}
Finally, if $r^{\star} \ge r_{n,m}$, then
\begin{equation}\label{eq:Case3Deltanm}
\begin{split}
\delta_{n,m} &\le \sup_{r\ge r_{n,m}}\,\max_{1\le k\le n}\,r^m|\mathbb{P}(\|U_{n,k}\| > r) - \mathbb{P}\left(\|U_{n,0}\| > r\right)|\\
&\le \sup_{r\ge r_{n,m}}\,\max_{0\le k\le n} r^m\mathbb{P}(\|U_{n,k}\| > r) .
\end{split}
\end{equation}
Combining the cases~\eqref{eq:Case1Deltanm},~\eqref{eq:Case2Deltanm} and~\eqref{eq:Case3Deltanm}, we get
\begin{align*}
\delta_{n,m} &\le 2^{2m^2/3 + 8m/3 +1}\varepsilon_{n}^{m+1}\Phi_{AC,0} + (2^{2m/3+1} + 2)\Phi_{AC,m}\varepsilon_n\\ 
&\quad+ \frac{2C_0 \log(ep)r_{n,m}^m M_n(2^{2m/3}\varepsilon_n)}{\varepsilon_n^{2}} + \frac{r_{n,m}^m\bar{L}_{n,m}}{L_n 2^{m}e^{\mathfrak{C}}}\left(\frac{\log(ep)}{n(2^{2m+1}e^{2\mathfrak{C}}C_0L_n)}\right)^{(m+1)/3}\\ 
&\quad+ \sup_{r\ge r_{n,m}}\,\max_{0\le k\le n} r^m\mathbb{P}(\|U_{n,k}\| > r).
\end{align*}
\subsection{Proof of Proposition~\ref{prop:BentkusProof}}
Following the smoothing inequality, we need to control
\[
I_j := \left|\mathbb{E}[\varphi(W_{n,j} + n^{-1/2}X_j) - \varphi(W_{n,j} + n^{-1/2}Y_j)]\right|.
\]
Using the equality of mean and variance of $X, Y$, we get
\[
I_j \le I_j^{(1)} + I_j^{(2)} := \int_{\|x\| \le \varepsilon n^{1/2}}|\mathbb{E}[\mbox{Rem}_n(W_{n,j}, x)]||\zeta|(dx) + \int_{\|x\| > \varepsilon n^{1/2}} |\mathbb{E}\mbox{Rem}_n(W_{n,j}, x)||\zeta|(dx).
\]
Note that for all $x$,
\[
|\mbox{Rem}_n(W_{n,j}, x)| \le \frac{C_3\varepsilon^{-3}}{6n^{3/2}}\norm{x}^3\mathbbm{1}\left\{r - \varepsilon - n^{-1/2}\norm{x} \le \norm{W_{n,j}} \le r + \varepsilon + n^{-1/2}\norm{x}\right\},
\]
which implies that
\begin{equation*}
\begin{split}
I_j^{(1)} &\le \frac{C_3\varepsilon^{-3}}{6n^{3/2}}\int_{\norm{x} \le \varepsilon n^{1/2}}\mathbb{P}\left(r - \varepsilon - n^{-1/2}\norm{x} \le \norm{W_{n,j}} \le r + \varepsilon + n^{-1/2}\norm{x}\right)\norm{x}^3|\zeta|(dx).
\end{split}
\end{equation*}
Recall that
\[
W_{n,j} = (1 - 1/n)^{1/2}(n-1)^{-1/2}(X_1 + \ldots + X_{j-1} + Y_{j+1} + \ldots + Y_n),
\]
and this is a constant multiple of scaled average of $n-1$ independent random variables. This implies that with $t_n = (1 - 1/n)^{-1/2}$,
\begin{align*}
&\mathbb{P}(r - \varepsilon - n^{-1/2}\|x\| \le \|W_{n,j}\| \le r + \varepsilon + n^{-1/2}\|x\|)\\ &\qquad\le \mathbb{P}(t_n(r - \varepsilon - n^{-1/2}\|x\|) \le \|Y\| \le t_n(r + \varepsilon + n^{-1/2}\|x\|)) + 2\delta_{n-1,0}\\
&\qquad\le t_n(\varepsilon + n^{-1/2}\|x\|)\Phi_{AC,0} + 2\delta_{n-1,0}.
\end{align*}
Substituting this in the inequality for $I_{j}^{(1)}$ yields
\begin{equation*}
\begin{split}
I_j^{(1)} &\le \frac{C_3\varepsilon^{-3}}{6n^{3/2}}\int_{\norm{x} \le \varepsilon n^{1/2}}\left[t_n\Phi_{AC,0}(\varepsilon + n^{-1/2}\|x\|) + 2\delta_{n-1,0}\right]\norm{x}^3|\zeta|(dx)\\
&\le \frac{C_3\varepsilon^{-3}}{6n^{3/2}}\int_{\norm{x} \le \varepsilon n^{1/2}}\left[2t_n\Phi_{AC,0}\varepsilon + 2\delta_{n-1,0}\right]\norm{x}^3|\zeta|(dx)\\
&\le \frac{C_3\varepsilon^{-3}\left\{2t_n\Phi_{AC,0}\varepsilon + 2\delta_{n-1,0}\right\}}{6n^{3/2}}\int_{\norm{x} \le \varepsilon n^{1/2}}\norm{x}^3|\zeta|(dx). 
\end{split}
\end{equation*}
For $I_j^{(2)}$, note that for all $x$,
\[
|\mbox{Rem}_n(W_{n,j}, x)| \le \frac{C_2\varepsilon^{-2}}{n}\norm{x}^2\mathbbm{1}\left\{r - \varepsilon - n^{-1/2}\norm{x} \le \norm{W_{n,j}} \le r + \varepsilon + n^{-1/2}\norm{x}\right\}.
\]
Thus by the calculation above, we get
\begin{align*}
I_j^{(2)} &\le \frac{C_2\varepsilon^{-2}}{n}\int_{\norm{x} > \varepsilon n^{1/2}}\norm{x}^2\mathbb{P}\left(r - \varepsilon - n^{-1/2}\norm{x} \le \norm{W_{n,j}} \le r + \varepsilon + n^{-1/2}\norm{x}\right)|\zeta|(dx)\\
&\le \frac{C_2\varepsilon^{-2}}{n}\int_{\norm{x} > \varepsilon n^{1/2}}\left[t_n\Phi_{AC,0}(\varepsilon + n^{-1/2}\|x\|) + 2\delta_{n-1,0}\right]\norm{x}^2|\zeta|(dx)\\
&\le \frac{C_2\varepsilon^{-3}\{2t_n\Phi_{AC,0}\varepsilon + 2\delta_{n-1,0}\}}{n^{3/2}}\int_{\|x\| > \varepsilon n^{1/2}}\|x\|^3|\zeta|(dx). 
\end{align*}
Adding the bounds for $I_j^{(1)}$ and $I_j^{(2)}$, we get
\[
I_j \le \frac{(C_2 + C_3)\nu_3^3\varepsilon^{-3}\{2t_n\Phi_{AC,0}\varepsilon + 2\delta_{n-1,0}\}}{n^{3/2}}.
\]
This implies that for $1 \le k\le n$,
\[
I \le \sum_{j = 1}^k I_j \le \frac{(C_2 + C_3)\nu_3^3\varepsilon^{-3}}{n^{1/2}}\{2t_n\Phi_{AC,0}\varepsilon + 2\delta_{n-1,0}\}.
\]
Adding in the $\mathbb{P}(r - \varepsilon \le \|Y\| \le r + \varepsilon) \le \Phi_{AC,0}\varepsilon$ to $I$ yields
\[
\delta_{n,0} \le (C_2 + C_3)\nu_3^3\varepsilon^{-3}n^{-1/2}\{2\sqrt{2}\Phi_{AC,0}\varepsilon + 2\delta_{n-1,0}\} + \Phi_{AC,0}\varepsilon.
\]
Define
\begin{align*}
F(\varepsilon) &:= 2(C_2 + C_3)\nu_3^3\varepsilon^{-3}n^{-1/2},\quad\mbox{and}\quad\Upsilon(\varepsilon) := \Phi_{AC,0}\varepsilon\left[\sqrt{2} F(\varepsilon) + 1\right].
\end{align*}
The recursive inequality can be written as $\delta_{n,0} \le F(\varepsilon)\delta_{n-1,0} + \Upsilon(\varepsilon).$
Take $$\varepsilon = \nu_3\left[\frac{4(C_2+C_3)}{n^{1/2}}\right]^{1/3},$$ so that $F(\varepsilon) = 1/2$. Then, we have:
$$\delta_{n,0} \leq \frac{1}{2}\delta_{n-1,0} + A_n,$$
where $$A_n := \left(\frac{1}{\sqrt{2}}+1\right)\Phi_{AC,0}\nu_3\left[\frac{4(C_2+C_3)}{n^{1/2}}\right]^{1/3}~.$$
Now, we have for all $n \geq 2$:
$$\frac{A_{n-1}}{A_n} = \left(\frac{n}{n-1}\right)^{1/6} \leq 2^{1/6}.$$
Hence, Lemma~\ref{lem:RecursiveIneq} yields:
\begin{align*}
\delta_{n,0} &\leq \frac{\delta_{1,0}}{2^{n-1}} + 2.28A_n\\ &\le \frac{1}{2^{n-1}} + 3.9\Phi_{AC,0}\nu_3\left[\frac{4(C_2+C_3)}{n^{1/2}}\right]^{1/3}\\ &\le \frac{1}{2^{n-1}} + 8\Phi_{AC,0}\nu_3\frac{\left(C_0\log^2(ep)\right)^{1/3}}{n^{1/6}}~. 
\end{align*}

\subsection{Proof of Theorem~\ref{thm:OptimalCLT}}
Following the proof of Proposition~\ref{prop:BentkusProof}, we need to control
\[
I_j := \left|\mathbb{E}[\varphi(W_{n,j} + n^{-1/2}X_j) - \varphi(W_{n,j} + n^{-1/2}Y_j)]\right|.
\]
Define the event $\mathcal{E} := \{\|x\| \le n^{1/2}\varepsilon/\log(ep)\}$. Using the equality of mean and variance of $X, Y$, we get
\[
I_j \le I_j^{(1)} + I_j^{(2)} := \int_{x\in\mathcal{E}}|\mathbb{E}[\mbox{Rem}_n(W_{n,j}, x)]||\zeta|(dx) + \int_{x\in\mathcal{E}^c} |\mathbb{E}\mbox{Rem}_n(W_{n,j}, x)||\zeta|(dx).
\]
To bound $I_j^{(1)}$, observe that
\[
\mbox{Rem}_n(W_{n,j},x) = \frac{1}{2}\sum_{j_1,j_2,j_3 = 1}^p \frac{x(j_1)x(j_2)x(j_3)}{n^{3/2}}\int_0^1 (1 - t)^2\partial_{j_1,j_2,j_3}\varphi(W_{n,j} + tn^{-1/2}x)dt.
\]
Now following the steps from Lemma~\ref{lem:GeneralmChern}, we have
\begin{align*}
&|\mbox{Rem}_n(W_{n,j},x)|\mathbbm{1}\{x\in\mathcal{E}\}\\
&\quad\le \frac{e^{\mathfrak{C}}}{2n^{3/2}}\sum_{j_1,j_2,j_3 = 1}^p |x(j_1)x(j_2)x(j_3)|D_{j_1j_2j_3}(W_{n,j})\mathbbm{1}\{|\|W_{n,j}\| - r| \le \varepsilon(1 + 1/\log(ep)), x\in\mathcal{E}\},
\end{align*}
which implies that
\begin{align*}
\sum_{j=1}^n I_j^{(1)} &\le \frac{e^{\mathfrak{C}}}{2n^{3/2}}\sum_{j=1}^n \max_{1\le j_1\le p}\int |x(j_1)|^3|\zeta|(dx)\mathbb{E}\left[\sum_{j_1,j_2,j_3=1}^pD_{j_1j_2j_3}(W_{n,j})\mathbbm{1}\{|\|W_{n,j}\| - r| \le 2\varepsilon\}\right].
\end{align*}
Since $\sum_{j_1,j_2,j_3 =1}^p D_{j_1j_2j_3}(z) \le C_3\varepsilon^{-3}$ for all $z\in\mathbb{R}^p$, we get
\begin{equation}\label{eq:Ij1Bound}
\sum_{j=1}^n I_j^{(1)} \le \frac{C_3\varepsilon^{-3}e^{\mathfrak{C}}}{2n^{3/2}}\max_{1\le j_1\le p}\int |x(j_1)|^3|\zeta|(dx)\sum_{j=1}^n \mathbb{P}(r - 2\varepsilon \le \|W_{n,j}\| \le r + 2\varepsilon)
\end{equation}
Recall that
\[
W_{n,j} = (1 - 1/n)^{1/2}(n-1)^{-1/2}(X_1 + \ldots + X_{j-1} + Y_{j+1} + \ldots + Y_n),
\]
and this is a constant multiple of scaled average of $n-1$ independent random variables. This implies that with $t_n = (1 - 1/n)^{-1/2}$,
\begin{align*}
\mathbb{P}(r - 2\varepsilon \le \|W_{n,j}\| \le r + 2\varepsilon) &\le \mathbb{P}(t_n(r - 2\varepsilon) \le \|Y\| \le t_n(r + 2\varepsilon)) + \delta_{n-1,0}\\
&\le 2t_n\varepsilon\Phi_{AC,0} + 2\delta_{n-1,0}.
\end{align*}
Substituting this in~\eqref{eq:Ij1Bound} yields
\begin{equation}\label{eq:FinalIj1Bound}
\mathbf{I} := \sum_{j=1}^n I_j^{(1)} \le \frac{C_3\varepsilon^{-3}e^{\mathfrak{C}}L_n}{2n^{1/2}}\{2\sqrt{2}\varepsilon\Phi_{AC,0} + 2\delta_{n-1,0}\}.
\end{equation}
For $I_j^{(2)}$, note that for all $x$,
\[
|\mbox{Rem}_n(W_{n,j}, x)| \le \frac{C_2\varepsilon^{-2}}{n}\norm{x}^2\mathbbm{1}\left\{r - \varepsilon - n^{-1/2}\norm{x} \le \norm{W_{n,j}} \le r + \varepsilon + n^{-1/2}\norm{x}\right\}.
\]
Thus by the calculation above, we get
\begin{align*}
I_j^{(2)} &\le \frac{C_2\varepsilon^{-2}}{n}\int_{x\in\mathcal{E}^c}\norm{x}^2\mathbb{P}\left(r - \varepsilon - n^{-1/2}\norm{x} \le \norm{W_{n,j}} \le r + \varepsilon + n^{-1/2}\norm{x}\right)|\zeta|(dx)\\
&\le \frac{C_2\varepsilon^{-2}}{n}\int_{x\in\mathcal{E}^c}\left[t_n\Phi_{AC,0}(\varepsilon + n^{-1/2}\|x\|) + 2\delta_{n-1,0}\right]\norm{x}^2|\zeta|(dx)
\end{align*}
This implies
\begin{align*}
\mathbf{II} &:= \sum_{j=1}^n I_j^{(2)} \le C_2\varepsilon^{-2}\int_{x\in\mathcal{E}^c} \left[\sqrt{2}\Phi_{AC,0}(\varepsilon + n^{-1/2}\|x\|) + 2\delta_{n-1,0}\right]\|x\|^2|\zeta|(dx)\\
&\le \{2\delta_{n-1,0} + \sqrt{2}\Phi_{AC,0}\varepsilon\}\frac{C_2\log^{\tau}(ep)}{\varepsilon^{2+\tau}n^{\tau/2}}\int \|x\|^{2 + \tau}|\zeta|(dx) + \frac{\sqrt{2}\Phi_{AC,0}C_2(\log(ep))^{\tau - 1}}{\varepsilon^{\tau + 1}n^{\tau/2}}\int {\|x\|^{2 + \tau}}|\zeta|(dx)\\
&\le \{2\delta_{n-1,0} + \sqrt{2}\Phi_{AC,0}\varepsilon\}\frac{C_2\nu_{2 + \tau}^{2 + \tau}\log^{\tau}(ep)}{\varepsilon^{2+\tau}n^{\tau/2}} + \frac{\sqrt{2}\Phi_{AC,0}C_2\nu_{2+\tau}^{2 + \tau}(\log(ep))^{\tau - 1}}{\varepsilon^{\tau + 1}n^{\tau/2}}.
\end{align*}
Adding $\mathbf{I}$ and $\mathbf{II}$, we obtain
\[
\mathbf{I} + \mathbf{II} \le \{2\delta_{n-1,0} + \sqrt{2}\Phi_{AC,0}\varepsilon\}\left[\frac{C_3\varepsilon^{-3}e^{\mathfrak{C}}L_n}{n^{1/2}} + \frac{C_2\nu_{2 + \tau}^{2 + \tau}\log^{\tau}(ep)}{\varepsilon^{2+\tau}n^{\tau/2}}\right] + \frac{\sqrt{2}\Phi_{AC,0}C_2\nu_{2+\tau}^{2 + \tau}(\log(ep))^{\tau - 1}}{\varepsilon^{\tau + 1}n^{\tau/2}}.
\]
Adding in the anti-concentration term, we get
\[
\delta_{n,0} \le \{2\delta_{n-1,0} + \sqrt{2}\Phi_{AC,0}\varepsilon\}\left[\frac{C_3\varepsilon^{-3}e^{\mathfrak{C}}L_n}{n^{1/2}} + \frac{C_2\nu_{2 + \tau}^{2 + \tau}\log^{\tau}(ep)}{\varepsilon^{2+\tau}n^{\tau/2}}\right] + \frac{\sqrt{2}\Phi_{AC,0}C_2\nu_{2+\tau}^{2 + \tau}(\log(ep))^{\tau - 1}}{\varepsilon^{\tau + 1}n^{\tau/2}} + \Phi_{AC,0}\varepsilon.
\]
Define
\begin{align*}
F(\varepsilon) &:= \frac{2C_3\varepsilon^{-3}e^{\mathfrak{C}}L_n}{n^{1/2}} + \frac{2C_2\nu_{2 + \tau}^{2 + \tau}\log^{\tau}(ep)}{\varepsilon^{2+\tau}n^{\tau/2}},\\
\Upsilon(\varepsilon) &:= \frac{1}{\sqrt{2}}\Phi_{AC,0}\varepsilon F(\varepsilon) + \frac{\sqrt{2}\Phi_{AC,0}C_2\nu_{2+\tau}^{2 + \tau}(\log(ep))^{\tau - 1}}{\varepsilon^{\tau + 1}n^{\tau/2}} + \Phi_{AC,0}\varepsilon\\&= \Phi_{AC,0}\varepsilon\left[\frac{1}{\sqrt{2}}F(\varepsilon) + 1\right] + \frac{\sqrt{2}\Phi_{AC,0}C_2\nu_{2+\tau}^{2+\tau}\log^{\tau}(ep)}{\varepsilon^{\tau+2}n^{\tau/2}}\frac{\varepsilon}{\log(ep)}\\
&\le \Phi_{AC,0}\varepsilon\left[\frac{1}{\sqrt{2}}F(\varepsilon) + 1\right] + \frac{\varepsilon F_2(\varepsilon)\Phi_{AC,0}}{\sqrt{2}\log(ep)}\\
&= \varepsilon\left[\frac{1}{\sqrt{2}}\Phi_{AC,0}F(\varepsilon) + \Phi_{AC,0} + \frac{\Phi_{AC, 0}F_2(\varepsilon)}{\sqrt{2}\log(ep)}\right].
\end{align*}
The recursive inequality can be written as $\delta_{n,0} \le F(\varepsilon)\delta_{n-1,0} + \Upsilon(\varepsilon).$ 
Take
\[
\varepsilon = \max\left\{\left(\frac{8C_3e^{\mathfrak{C}}L_n}{n^{1/2}}\right)^{1/3}, \nu_{2+\tau}\left(\frac{8C_2\log^{\tau}(ep)}{n^{\tau/2}}\right)^{1/(2+\tau)}\right\},
\]
so that $\max\{F_1(\varepsilon), F_2(\varepsilon)\} \le 1/4$. This implies that 
\[
\delta_{n,0} \le \frac{1}{2}\delta_{n-1,0} + A_n,
\]
where
\[
A_n := \Phi_{AC, 0}\left(\frac{1}{2\sqrt{2}}+1+\frac{1}{4\sqrt{2}\log(ep)}\right)\left[\left(\frac{8C_3e^{\mathfrak{C}}L_n}{n^{1/2}}\right)^{1/3} + \nu_{2+\tau}\left(\frac{8C_0\log^{\tau+1}(ep)}{n^{\tau/2}}\right)^{1/(2+\tau)}\right]
\]
Note that $A_n$ is decreasing in $n$ and for $n\ge 2$, $n-1\ge n/2$ which implies that $2^{-1/2} \le A_n/A_{n-1} \le 1$. Hence Lemma~\ref{lem:RecursiveIneq} yields
\begin{align*}
\delta_{n,0} &\le \frac{\delta_{1,0}}{2^{n-1}} + \frac{\sqrt{2}}{\sqrt{2} - 1}A_n\\
&\le \frac{\delta_{1,0}}{2^{n-1}} + 6\Phi_{AC, 0}\left[\left(\frac{8C_0\log^2(ep)e^{\mathfrak{C}}L_n}{n^{1/2}}\right)^{1/3} + \nu_{2+\tau}\left(\frac{8C_0\log^{\tau+1}(ep)}{n^{\tau/2}}\right)^{1/(2+\tau)}\right].
\end{align*}
Therefore, we get
\[
\delta_{n,0} \lesssim \frac{1}{2^n} + \Phi_{AC,0}\left[\left(\frac{L_n^2\log^4(ep)}{n}\right)^{1/6} + \frac{\nu_{2+\tau}(\log(ep))^{(\tau+1)/(\tau+2)}}{n^{\tau/(2(2+\tau))}}\right].
\]
\subsection{Proof of Theorem~\ref{thm:OptimalNonUniformCLT}}
Recall $\delta_{n,m} = \sup_{r\ge 0}\Delta_{n,m}(r)$, where
\[
\Delta_{n,m}(r) = \max_{1\le k\le n}\,r^m|\mathbb{P}(\|U_{n,k}\| \le r) - \mathbb{P}(\|U_{n,0}\| \le r)|.
\]
If $r \le 4\varepsilon$ then $\Delta_{n,m}(r) \le (4\varepsilon)^m$ and hence enough to bound $\Delta_{n,m}(r)$ for $r > 4\varepsilon$. Smoothing inequality yields
\begin{equation}\label{eq:MainSmoothingIneq}
\begin{split}
\Delta_{n,m}(r) &\le r^m\max_{1\leq k \leq n}\max_{j=1,2}|\mathbb{E}[\varphi_j(U_{n,k}) - \varphi_j(U_{n,0})]| + r^m\mathbb{P}(r - \varepsilon \le \|U_{n,0}\| \le r + \varepsilon)\\
&\le r^m\max_{1\leq k \leq n}\max_{j=1,2}|\mathbb{E}[\varphi(U_{n,k}) - \varphi(U_{n,0})]| + \Phi_{AC,m}\varepsilon.
\end{split}
\end{equation}
Write $\varphi$ for $\varphi_1, \varphi_2$ both. Observe now that
\[
I = |\mathbb{E}[\varphi(U_{n,k}) - \varphi(U_{n,0})]| \le \sum_{j=1}^k I_j,
\]
where
\begin{align*}
I_j &= |\mathbb{E}[\varphi(W_{n,j} + n^{-1/2}X_j) - \varphi(W_{n,j} + n^{-1/2}Y_j)]|\\
&= \left|\mathbb{E}\int \varphi(W_{n,j} + n^{-1/2}x)\zeta(dx)\right|\\
&= \left|\mathbb{E}\int \mbox{Rem}_n(W_{n,j}, x)\zeta(dx)\right|\\
&= \left|\int_{x\in\mathcal{E}_1} + \int_{x\in\mathcal{E}_2} + \int_{x\in\mathcal{E}_3}\mbox{Rem}_n(W_{n,j}, x)\zeta(dx)\right|,
\end{align*}
where $\mathcal{E}_1 := \{x:\|x\| \le n^{1/2}\varepsilon/\log(ep)\}$, $\mathcal{E}_2 := \{x:n^{1/2}\varepsilon/\log(ep) < \|x\| \le rn^{1/2}/2\}$, and $\mathcal{E}_3 := \{x:\|x\| > rn^{1/2}/2\}.$
Set
\[
I_j^{(\ell)} = \mathbb{E}\int_{x\in\mathcal{E}_{\ell}} |\mbox{Rem}_n(W_{n,j}, x)||\zeta|(dx)\quad\mbox{for}\quad \ell = 1, 2, 3, 1\le j\le n.
\]
To bound $I_j^{(1)}$ note that
\[
\mbox{Rem}_n(W_{n,j}, x) = \frac{1}{2}\sum_{j_1,j_2,j_3 = 1}^p \frac{x(j_1)x(j_2)x(j_3)}{n^{3/2}}\int_0^1 (1 - t)^2\partial_{j_1j_2j_3}\varphi(W_{n,j} + tn^{-1/2}x)dt,
\]
and hence for $x\in\mathcal{E}_1$,
\begin{align*}
&|\mbox{Rem}_n(W_{n,j},x)|\mathbbm{1}\{x\in\mathcal{E}_1\}\\ 
&\quad\le \frac{e^{\mathfrak{C}}}{2n^{3/2}}\sum_{j_1j_2j_3} |x(j_1)x(j_2)x(j_3)|D_{j_1j_2j_3}(W_{n,j})\mathbbm{1}\{|\|W_{n,j}\| - r| \le \varepsilon(1 + 1/\log(ep)), x\in\mathcal{E}_1\}.
\end{align*}
Therefore,
\begin{align*}
I_j^{(1)} &\le \frac{e^{\mathfrak{C}}}{2n^{3/2}}\sum_{j_1j_2j_3}\max_{1\le j_1\le p}\int |x(j_1)|^3|\zeta|(dx)\mathbb{E}[D_{j_1j_2j_3}(W_{n,j})\mathbbm{1}\{r - 2\varepsilon \le \|W_{n,j}\| \le r + 2\varepsilon\}]\\
&\le \frac{C_3\varepsilon^{-3}e^{\mathfrak{C}}L_n}{2n^{3/2}}\mathbb{P}(r - 2\varepsilon \le \|W_{n,j}\| \le r + 2\varepsilon).
\end{align*}
To simplify the right hand side note that $W_{n,j} \overset{d}{=} ((n-1)/n)^{1/2}U_{n-1,j-1}$ and hence
\begin{align*}
&\mathbb{P}(r - 2\varepsilon \le \|W_{n,j}\| \le r + 2\varepsilon)\\ &\quad= \mathbb{P}(\sqrt{n/(n-1)}(r - 2\varepsilon) \le \|U_{n-1,j}\| \le \sqrt{n/(n-1)}(r + 2\varepsilon))\\
&\quad= \mathbb{P}(\sqrt{n/(n-1)}(r - 2\varepsilon) \le \|U_{n,0}\| \le \sqrt{n/(n-1)}(r + 2\varepsilon)) + \frac{2\delta_{n-1,m}}{((r - 2\varepsilon)\sqrt{n/(n-1)})^m}\\
&\quad\le \frac{2\Phi_{AC,m}\varepsilon\sqrt{n/(n-1)}}{(r\sqrt{n/(n-1)})^m} + \frac{2\delta_{n-1,m}}{((r - 2\varepsilon)\sqrt{n/(n-1)})^m}\\
&\quad\le \frac{2\sqrt{2}\Phi_{AC,m}\varepsilon}{r^m} + \frac{2^{m+1}\delta_{n-1,m}}{r^m},
\end{align*}
where the last inequality follows from the fact that $r - 2\varepsilon \ge r/2$ since $r > 4\varepsilon$ and $n\ge 2$. Therefore
\begin{equation}\label{eq:BoundIj1}
\sum_{j=1}^n I_{j}^{(1)} \le \frac{C_3\varepsilon^{-3}e^{\mathfrak{C}}L_n}{2n^{1/2}r^m}\left[2\sqrt{2}\Phi_{AC,m}\varepsilon + 2^{m+1}\delta_{n-1,m}\right].
\end{equation}

We now bound $I_j^{(2)}$. For $x\in\mathcal{E}_2$, we use
\[
|\mbox{Rem}_n(W_{n,j},x)| \le \frac{C_2\varepsilon^{-2}\|x\|^2}{n}\mathbbm{1}\{r - \varepsilon - n^{-1/2}\|x\| \le \|W_{n,j}\| \le r + \varepsilon + n^{-1/2}\|x\|\}.
\]
This implies
\[
I_j^{(2)} \le \int_{x\in\mathcal{E}_2}\frac{C_2\varepsilon^{-2}\|x\|^2}{n}\mathbb{P}(r - \varepsilon - n^{-1/2}\|x\| \le \|W_{n,j}\| \le r + \varepsilon + n^{-1/2}\|x\|)|\zeta|(dx).
\]
We can control the probability on right hand side in the same way done for $I_j^{(1)}$ to get
\begin{align*}
&\mathbb{P}(r - \varepsilon - n^{-1/2}\|x\| \le \|W_{n,j}\| \le r + \varepsilon + n^{-1/2}\|x\|)\\
&\quad\le \frac{\Phi_{AC,m}(\varepsilon + n^{-1/2}\|x\|)\sqrt{n/(n-1)}}{(r\sqrt{n/(n-1)})^m} + \frac{2\delta_{n-1,m}}{((r - \varepsilon - n^{-1/2}\|x\|)\sqrt{n/(n-1)})^m}.
\end{align*}
Since $r > 4\varepsilon$ and $\varepsilon/\log(ep) \le n^{-1/2}\|x\| \le r/2$ for $x\in\mathcal{E}_2$, we get $r - \varepsilon - n^{-1/2}\|x\| > r - r/4 - r/2 = r/4.$ Therefore
\[
\mathbb{P}(r - \varepsilon - n^{-1/2}\|x\| \le \|W_{n,j}\| \le r + \varepsilon + n^{-1/2}\|x\|) \le \frac{\sqrt{2}\Phi_{AC,m}(\varepsilon + n^{-1/2}\|x\|) + 2^{2m+1}\delta_{n-1,m}}{r^m},
\]
and hence
\begin{equation}\label{eq:BoundIj2}
\sum_{j=1}^n I_j^{(2)} \le \int_{x\in \mathcal{E}_2} \frac{C_2\varepsilon^{-2}\|x\|^2}{r^m}\left[\sqrt{2}\Phi_{AC,m}(\varepsilon + n^{-1/2}\|x\|) + 2^{2m+1}\delta_{n-1,m}\right]|\zeta|(dx).
\end{equation}
Finally for $n^{-1/2}\|x\| > r/2$, we use
\[
|\mbox{Rem}_n(W_{n,j},x)| \le \frac{C_2\varepsilon^{-2}\|x\|^2}{n},
\]
which implies
\begin{equation}\label{eq:BoundIj3}
\sum_{j=1}^n I_j^{(3)} \le \int_{x\in\mathcal{E}_3}{C_2\varepsilon^{-2}\|x\|^2}|\zeta|(dx).
\end{equation}
Combining~\eqref{eq:BoundIj1},~\eqref{eq:BoundIj2} and~\eqref{eq:BoundIj3} in~\eqref{eq:MainSmoothingIneq}, we get
\begin{align*}
\delta_{n,m} &\le 4^m\varepsilon^m + \Phi_{AC,m}\varepsilon + \frac{C_3\varepsilon^{-3}e^{\mathfrak{C}}L_n}{n^{1/2}}\left[\sqrt{2}\Phi_{AC,m}\varepsilon + 2^{m}\delta_{n-1,m}\right]\\
&\qquad+ \int_{x\in\mathcal{E}_2}{C_2\varepsilon^{-2}\|x\|^2}\left[\sqrt{2}\Phi_{AC,m}(\varepsilon + n^{-1/2}\|x\|) + 2^{2m+1}\delta_{n-1,m}\right]|\zeta|(dx)\\
&\qquad+ \sup_{r > 4\varepsilon} r^m\int_{x\in\mathcal{E}_3} C_2\varepsilon^{-2}\|x\|^2|\zeta|(dx)\\
&=: \mathbf{I} + \mathbf{II} + \mathbf{III} + \mathbf{IV} + \mathbf{V}.
\end{align*}
We now use the moment condition on $\|X\|$ and bound the terms above. Note that
\begin{align*}
\mathbf{III} + \mathbf{IV} &= \frac{C_3\varepsilon^{-3}e^{\mathfrak{C}}L_n}{n^{1/2}}\left[\sqrt{2}\Phi_{AC,m}\varepsilon + 2^m\delta_{n-1,m}\right]\\ &\qquad+ \int_{x\in\mathcal{E}_2}{C_2\varepsilon^{-2}\|x\|^2}\left[\sqrt{2}\Phi_{AC,m}\varepsilon + 2^{2m+1}\delta_{n-1,m}\right]|\zeta|(dx)\\
&\qquad+ \int_{x\in\mathcal{E}_2} C_2\varepsilon^{-2}\|x\|^2\left(\sqrt{2}\Phi_{AC,m}n^{-1/2}\|x\|\right)|\zeta|(dx)\\
&\le \frac{C_3\varepsilon^{-3}e^{\mathfrak{C}}L_n}{n^{1/2}}\left[\sqrt{2}\Phi_{AC,m}\varepsilon + 2^m\delta_{n-1,m}\right]\\ &\qquad+ C_2\varepsilon^{-2}\left[\sqrt{2}\Phi_{AC,m}\varepsilon + 2^{2m+1}\delta_{n-1,m}\right]\int_{x\in\mathcal{E}_2}\frac{\|x\|^{2+\tau}}{(n^{1/2}\varepsilon/\log(ep))^{\tau}}|\zeta|(dx)\\
&\qquad+ \frac{\sqrt{2}\Phi_{AC,m}C_2\varepsilon^{-2}}{n^{1/2}}\int_{x\in\mathcal{E}_2} \frac{\|x\|^{2+\tau}}{(n^{1/2}\varepsilon/\log(ep))^{\tau - 1}}|\zeta|(dx)\\
&\le \frac{C_3\varepsilon^{-3}e^{\mathfrak{C}}L_n}{n^{1/2}}\left[\sqrt{2}\Phi_{AC,m}\varepsilon + 2^m\delta_{n-1,m}\right]\\ &\qquad+ C_2\varepsilon^{-2}\left[\sqrt{2}\Phi_{AC,m}\varepsilon + 2^{2m+1}\delta_{n-1,m}\right]\frac{\nu_{2+\tau}^{2+\tau}}{(n^{1/2}\varepsilon/\log(ep))^{\tau}}\\
&\qquad+ \frac{\sqrt{2}\Phi_{AC,m}C_2\varepsilon^{-2}}{n^{1/2}}\frac{\nu_{2+\tau}^{2+\tau}}{(n^{1/2}\varepsilon/\log(ep))^{\tau - 1}}\\
&= \delta_{n-1,m}\left[F_1(\varepsilon) + F_2(\varepsilon)\right] + \frac{\sqrt{2}\Phi_{AC,m}\varepsilon}{2^m}\left(F_1(\varepsilon) + \frac{F_2(\varepsilon)}{2^{m+1}} + \frac{F_2(\varepsilon)}{2^{m+1}\log(ep)}\right),
\end{align*}
where
\[
F_1(\varepsilon) = \frac{2^m C_3\varepsilon^{-3}e^{\mathfrak{C}}L_n}{n^{1/2}}\quad\mbox{and}\quad F_2(\varepsilon) = \nu_{2+\tau}^{2+\tau}\frac{2^{2m+1}C_2\varepsilon^{-2}}{(n^{1/2}\varepsilon/\log(ep))^{\tau}}.
\]
To bound $\mathbf{V}$, note that
\[
\int_{x\in\mathcal{E}_3} C_2\varepsilon^{-2}\|x\|^2|\zeta|(dx) \le C_2\varepsilon^{-2}\int \frac{\|x\|^{2 + \tau}|\zeta|(dx)}{(rn^{1/2}/2)^{\tau}} = \frac{C_2\varepsilon^{-2}\nu_{2+\tau}^{2+\tau}2^{\tau}}{r^{\tau}n^{\tau/2}}.
\]
Multiplying by $r^m$ and taking supremum over $r > 4\varepsilon$, we get
\[
\mathbf{V} \le \frac{C_2\nu_{2+\tau}^{2+\tau}2^{\tau}}{n^{\tau/2}(4\varepsilon)^{\tau-m}\varepsilon^2} = \frac{F_2(\varepsilon)\varepsilon^m}{2(2\log(ep))^{\tau}},
\]
since $\tau > m$. Hence
\[
\delta_{n,m} \le [F_1(\varepsilon) + F_2(\varepsilon)]\delta_{n-1,m} + \Upsilon(\varepsilon),
\]
where
\[
\Upsilon(\varepsilon) := 2^{2m}\varepsilon^{m} + \Phi_{AC,m}\varepsilon + \frac{\sqrt{2}\Phi_{AC,m}\varepsilon}{2^m}\left(F_1(\varepsilon) + \frac{F_2(\varepsilon)}{2^m}\right) + \frac{F_2(\varepsilon)\varepsilon^m}{2(2\log(ep))^{\tau}}
\]
Choose $\varepsilon$ so that $\max\{F_1(\varepsilon), F_2(\varepsilon)\} \le 2^{-2-m/2}$, that is,
\[
\varepsilon = \varepsilon_n := \max\left\{\left(\frac{2^{2 + (3m/2)}C_3e^{\mathfrak{C}}L_n}{n^{1/2}}\right)^{1/3},\,\frac{\nu_{2+\tau}\left(2^{3+(5m/2)}C_2\log^{\tau}(ep)\right)^{1/(2+\tau)}}{n^{\tau/(4 + 2\tau)}}\right\}.
\]
This choice implies (using $m\ge1$)
\begin{align*}
\Upsilon(\varepsilon_n) &\le 2^{2m}\varepsilon_n^m + \Phi_{AC,m}\varepsilon_n + \frac{\sqrt{2}\Phi_{AC,m}\varepsilon_n}{2^m}\left(2^{-2-(m/2)} + 2^{-2-(3m/2)}\right) + \frac{\varepsilon_n^m}{2^{3+(m/2)}(2\log(ep))^{\tau}}\\
&\le (2^{2m} + 1/8)\varepsilon_n^m + \Phi_{AC,m}\varepsilon_n\left(1 + 2^{-(3/2)-(3m/2)}+ 2^{-(3/2)-(5m/2)}\right)\\
&\le (2^{2m}+1/8)\varepsilon_n^m + \Phi_{AC,m}\varepsilon_n\left(1+2^{-3}+2^{-4}\right)\\ &= (2^{2m}+1/8)\varepsilon_n^m + 1.1875\Phi_{AC,m}\varepsilon_n=: a_n,
\end{align*}
and $\delta_{n,m} \le 2^{-1-m/2}\delta_{n-1,m} + a_n$. It is also easy to verify that
\[
\frac{a_{n-1}}{a_n} \le \max\left\{\frac{\varepsilon_{n-1}^m}{\varepsilon_n^m}, \frac{\varepsilon_{n-1}}{\varepsilon_n}\right\} \le \max\left\{\left(\frac{n}{n-1}\right)^{m/6}, \left(\frac{n}{n-1}\right)^{m\tau/(4 + 2\tau)}\right\} \le 2^{m/2}.
\]
Hence the recursion on $\delta_{n,m}$ satisfies the assumption of Lemma~\ref{lem:RecursiveIneq} with $\bar{C} = 2^{m/2}$ and $\kappa = 2^{-1-m/2}$ which now implies that
\[
\delta_{n,m} \le \left(\frac{1}{2^{1 + m/2}}\right)^{n-1}\delta_{1,m} + 2a_n.
\]
To control $\delta_{1,m}$, observe that
\[
\delta_{1,m} = \sup_{r\ge 0}\,r^m\left|\int_{\|x\| > r}\zeta(dx)\right| \le \sup_{r \ge 0}\, r^m\int_{\|x\| > r} \frac{\|x\|^m}{r^m}|\zeta|(dx) = \sup_{r\ge 0}\int_{\|x\| > r} \|x\|^m|\zeta|(dx) = \nu_m^m.
\]
Therefore,
\begin{align*}
\delta_{n,m} &\le 2^{m/2}\left(\frac{\nu_m}{2^{n/2}}\right)^m + 2\left[(2^{2m}+1/8)\varepsilon_n^m + 1.1875\Phi_{AC,m}\varepsilon_n\right]\\ 
&= 2^{m/2}\left(\frac{\nu_m}{2^{n/2}}\right)^m + 2^{2 + 2m}\varepsilon_n^{m} + 2.375\Phi_{AC,m}\varepsilon_n.
\end{align*}
\section{Proofs of Corollaries~\ref{cor:UniformCLT} and~\ref{cor:NonUniformPolyCLT}}
The following result provides a dimension-free bound on $\delta_{n,0}$ and shows bounds on $M_n(\varepsilon_n)$.
\begin{prop}\label{propzero}
If the random vectors $X_1, \ldots, X_n\in\mathbb{R}^p$ satisfy 
\[
\frac{1}{n}\sum_{i=1}^n \mathbb{E}\left[\|X_i\|^{q}\right] ~\le~ q^{q/\alpha}B_{p,n}^{q},\quad\mbox{for all}\quad q \geq 2,
\]
for some $0 < \alpha \le 2$ and $0 < B_{p,n} < \infty$, then for
\begin{equation}\label{eq:SampleChernAssump}
n ~\ge~ \left(\frac{2^{3/\alpha-1}e^{3/\alpha - 2\mathfrak{C}}}{C_0L_n}\right)\times \max\{B_{p,n},C_{p,n}\}^{3}\log(ep),
\end{equation}
where $C_{p,n}:= 8\pi e^{1/e}(\mu_{(n)} + \sigma_{(n)})$, $\mu_{(n)}:= \max_{1\leq i\leq n} \mu_i$ and $\sigma_{(n)}:= \max_{1\leq i \leq n} \sigma_i$, we obtain
\begin{eqnarray*}
\delta_{n,0} &\le& 3e^{2\mathfrak{C}/3}C_0^{1/3}(L_n^{1/3} + 6L_n^{-2/3})\frac{\Phi_{AC,0}\log^{2/3}(ep)}{n^{1/6}} + \frac{\log^{1/3}(ep)\bar{L}_n}{n^{1/3}L_n(2e^{5\mathfrak{C}}C_0L_n)^{1/3}}\\&+& \left(\frac{2C_0n}{L_n^2\log(ep)}\right)^{1/3} \Bigg[C_{\alpha}B_{p,n}^2\exp\left(-\frac{(2ne^{2\mathfrak{C}}C_0L_n/\log(ep))^{\alpha/3}}{2B_{p,n}^{\alpha}\alpha e}\right)\\ &+& \mathcal{C}_2 C_{p,n}^2\exp\left(-\frac{(2ne^{2\mathfrak{C}}C_0L_n/\log(ep))^{2/3}}{4C_{p,n}^{2}e}\right)\Bigg] .
\end{eqnarray*}
for some constant $C_\alpha$ depending only on $\alpha$.
\end{prop}
\begin{proof}
To start with, note that
\[
M_n(\varepsilon_n) \le \frac{1}{n}\sum_{i=1}^n \mathbb{E}\left[\|X_i\|^2\mathbbm{1}\{\|X_i\| > n^{1/2}\varepsilon_n/\log(ep)\} + \|Y_i\|^2\mathbbm{1}\{\|Y_i\| > n^{1/2}\varepsilon_n/\log(ep)\}\right].
\]
For any $t > 0$ note that for any $q > 0$,
\begin{align*}
\frac{1}{n}\sum_{i=1}^n \mathbb{E}\left[\|X_i\|^2\mathbbm{1}\{\|X_i\| > t\}\right] &\le \frac{1}{n}\sum_{i=1}^n \frac{\mathbb{E}[\|X_i\|^{q+2}]}{t^q}\\
&\le \frac{(q+2)^{(q+2)/\alpha}B_{p,n}^{q+2}}{t^q}\\
&\le \frac{B_{p,n}^2(q + 2)^{(q+2)/\alpha}}{(t/B_{p,n})^{q}}.
\end{align*}
If $t/B_{p,n} \ge (2e)^{1/\alpha}$, then taking $q = (t/B_{p,n})^{\alpha}/e - 2$ we get
\[
\frac{1}{n}\sum_{i=1}^n \mathbb{E}\left[\|X_i\|^2\mathbbm{1}\{\|X_i\| > t\}\right] \le t^2\exp\left(-\frac{(t/B_{p,n})^{\alpha}}{\alpha e}\right) \le C_{\alpha}B_{p,n}^2\exp\left(-\frac{(t/B_{p,n})^{\alpha}}{2\alpha e}\right),
\]
for some constant $C_\alpha$ depending only on $\alpha$.

 Assumption~\eqref{eq:SampleChernAssump} implies that $t = n^{1/2}\varepsilon_n/\log(ep) = (2ne^{2\mathfrak{C}}C_0L_n/\log(ep))^{1/3}$ satisfies $t/B_{p,n} \ge (2e)^{1/\alpha}$ and hence
\begin{equation}\label{eq:xpart}
\frac{1}{n}\sum_{i=1}^n \mathbb{E}\left[\|X_i\|^2\mathbbm{1}\{\|X_i\| > n^{1/2}\varepsilon_n/\log(ep)\}\right] \le C_{\alpha}B_{p,n}^2\exp\left(-\frac{(2ne^{2\mathfrak{C}}C_0L_n/\log(ep))^{\alpha/3}}{2B_{p,n}^{\alpha}\alpha e}\right).
\end{equation}

Next, note that for every $t > 0$, we have from Lemma 3.1 of~\cite{MR2814399}
\[
\mathbb{P}\left(\big|\|Y_i\| - \mu_i\big| \ge t\right) \leq 2\exp\left(-\frac{t^2}{2\sigma_i^2}\right).
\]
Hence, for every $q>0$, we have:
\begin{eqnarray*}
\mathbb{E}\left[\big|\|Y_i\| - \mu_i\big| ^q\right] &\leq& 2\int_0^\infty qt^{q-1} \exp\left(-\frac{t^2}{2\sigma_i^2}\right)~dt\\ &=& q\sigma_i\sqrt{2\pi} ~\mathbb{E} |N(0,\sigma_i^2)|^{q-1}\\&\leq& q\sigma_i\sqrt{2\pi} \left(\sigma_i e^{1/e} \sqrt{q-1}\right)^{q-1}\\&\leq& \sqrt{2\pi}e^{(q-1)/e}\sigma_i^q q^{(q+1)/2}\\&\leq& \sqrt{2\pi}e^{(q-1)/e}(2\sigma_i)^q q^{q/2}.
\end{eqnarray*}
Hence,
\begin{eqnarray}
\frac{1}{n} \sum_{i=1}^n \mathbb{E}\left[\|Y_i\|^q\right] \leq \sqrt{2\pi}2^{2q}e^{q/e}q^{q/2} \frac{1}{n}\sum_{i=1}^n (\mu_i^q + \sigma_i^q)\leq C_{p,n}^q q^{q/2},
\end{eqnarray}
Hence, by assumption~\eqref{eq:SampleChernAssump}, we have
\begin{equation}\label{eq:ycpn}
\frac{1}{n}\sum_{i=1}^n \mathbb{E}\left[\|Y_i\|^2\mathbbm{1}\{\|Y_i\| > n^{1/2}\varepsilon_n/\log(ep)\}\right] \le \mathcal{C}_2 C_{p,n}^2\exp\left(-\frac{(2ne^{2\mathfrak{C}}C_0L_n/\log(ep))^{2/3}}{4C_{p,n}^{2}e}\right)
\end{equation}
for some universal constant $\mathcal{C}_2$ .
Combining \eqref{eq:xpart} and \eqref{eq:ycpn}, we have:
\begin{eqnarray}\label{mnbound}
&&M_n(\varepsilon_n)\\ &\leq& C_{\alpha}B_{p,n}^2\exp\left(-\frac{(2ne^{2\mathfrak{C}}C_0L_n/\log(ep))^{\alpha/3}}{2B_{p,n}^{\alpha}\alpha e}\right) + \mathcal{C}_2 C_{p,n}^2\exp\left(-\frac{(2ne^{2\mathfrak{C}}C_0L_n/\log(ep))^{2/3}}{4C_{p,n}^{2}e}\right).
\end{eqnarray}
Proposition \ref{propzero} now follows from Theorem~\ref{thm:UniformCLT}.
\end{proof}
The following preliminary lemma is required for the proofs.
\begin{lem}\label{orlicz}
Let $X = \left(X(1),\ldots,X(p)\right)$ be an $\mathbb{R}^p$ valued random variable. Suppose that there exists a constant $0 \leq K_p < \infty$, such that 
\begin{equation}\label{condorlicz}
\max\limits_{1\leq j\leq p} \norm{X(j)}_{\psi_\alpha} \leq K_p
\end{equation}
for some $\alpha \geq 0$. 
Then, for all $q \geq 1$, we have:
\begin{equation}\label{part1orlicz}
\mathbb{E}\norm{X}_\infty^q \leq K_p^q \left(2^{1/q} \left(\frac{6q}{e\alpha}\right)^{1/\alpha} + 2^{1/\alpha} (\log p)^{1/\alpha}\right)^q.
\end{equation}
Moreover, for $\alpha \geq 1$, we have
\begin{equation}\label{part2orlicz}
\mathbb{E} \exp\left[\frac{\norm{X}_\infty}{3^{1/\alpha} (\log 2)^{(1-\alpha)/\alpha}K_p \left(1+(\log p)^{1/\alpha}\right)}\right] \leq 2 .
\end{equation}
\end{lem}
\begin{proof}
We prove~\eqref{part2orlicz} first. Fix $1 \leq j \leq p$. Since $\norm{X(j)}_{\psi_\alpha} \leq K_p$ and $\psi_\alpha$ is increasing, we have
\begin{equation}\label{orlicz1}
\mathbb{E} \psi_\alpha \left(\frac{|X(j)|}{K_p}\right) \leq 1.
\end{equation}
Hence, by an application of Markov's inequality, we have for all $t \geq 0$ :
\begin{equation}\label{Markovorlicz}
\mathbb{P}\left(|X(j)| \geq K_p t^{1/\alpha}\right) \leq \frac{\mathbb{E} \psi_\alpha \left(|X(j)|/K_p\right)+1}{\psi_\alpha\left(t^{1/\alpha}\right)+1} \leq 2 e^{-t}.
\end{equation}
It follows from~\eqref{Markovorlicz} and an union bound, that for all $t \geq 0$, 
\begin{equation}
\mathbb{P}\left(\norm{X}_\infty \geq K_p (t+\log p)^{1/\alpha}\right) \leq 2 e^{-t}.
\end{equation}
Since $\alpha \geq 1$, $(t+\log p)^{1/\alpha} \leq t^{1/\alpha} + (\log p)^{1/\alpha}$. Hence, for all $t \geq 0$,
\begin{equation}\label{concaveorlicz}
\mathbb{P}\left(\norm{X}_\infty - K_p (\log p)^{1/\alpha} \geq K_p t^{1/\alpha}\right) \leq 2 e^{-t}.
\end{equation}
Define $$W := \left(\norm{X}_\infty - K_p (\log p)^{1/\alpha}\right)_+ ~.$$ Then it follows from~\eqref{concaveorlicz}, that for all $t \geq 0$,
\begin{equation}\label{Worlicz1}
\mathbb{P}\left(W \geq K_p t^{1/\alpha}\right) \leq 2 e^{-t}.
\end{equation}
Hence, we have from~\eqref{Worlicz1},
\begin{eqnarray}
\mathbb{E}\left[\exp\left(\frac{W^\alpha}{3K_p^\alpha}\right) - 1\right] &=& \int_0^\infty \mathbb{P}\left(\exp\left(\frac{W^\alpha}{3K_p^\alpha}\right) - 1 \geq t\right) dt\\&=& \int_0^\infty \mathbb{P}\left(W \geq K_p \left(3\log(1+t)\right)^{1/\alpha}\right) dt\\&\leq& 2\int_0^\infty (1+t)^{-3} dt = 1.
\end{eqnarray}
Hence, $\norm{W}_{\psi_\alpha} \leq 3^{1/\alpha} K_p .$
Consequently, 
\begin{eqnarray}\label{normpsiorl}
\norm{\norm{X}_\infty}_{\psi_\alpha} &\leq& \norm{W}_{\psi_\alpha} + \norm{K_p(\log p)^{1/\alpha}}_{\psi_{\alpha}}\\&\leq& 3^{1/\alpha} K_p + K_p(\log p)^{1/\alpha} (\log 2)^{-1/\alpha}\\&\leq& 3^{1/\alpha} K_p \left(1+(\log p)^{1/\alpha}\right) .
\end{eqnarray}
Hence, from Problem 5 of Chapter 2.2 of~\cite{VdvW96}, $$\norm{\norm{X}_\infty}_{\psi_1} \leq \norm{\norm{X}_\infty}_{\psi_\alpha} (\log 2)^{(1-\alpha)/\alpha} \leq 3^{1/\alpha} (\log 2)^{(1-\alpha)/\alpha}K_p \left(1+(\log p)^{1/\alpha}\right).$$ This proves~\eqref{part2orlicz}. For proving~\eqref{part1orlicz}, note that we have by a similar argument as before, but using the additional fact that $(t+\log p)^{1/\alpha} \leq 2^{1/\alpha}t^{1/\alpha} + 2^{1/\alpha} (\log p)^{1/\alpha}$ for $\alpha \geq 0$, 
\begin{equation}\label{wtildeorlicz}
\mathbb{P}\left(\tilde{W} \geq 2^{1/\alpha} K_p t^{1/\alpha}\right) \leq 2e^{-t},
\end{equation}
where $\tilde{W} := \left(\norm{X}_\infty - 2^{1/\alpha}K_p (\log p)^{1/\alpha}\right)_+$. Hence, we have:
\begin{equation}\label{wtilde2}
\mathbb{E}\left[\exp\left(\frac{\tilde{W}^\alpha}{6 K_p^\alpha}\right)-1\right] \leq 1,
\end{equation}\label{wtildeprel}
It is easy to see that for all $x \geq 0$ and $v_1, v_2 \geq 0$, we have:
\begin{equation}
x^{v_1}\exp\left(-\frac{x}{v_2}\right) \leq v_1^{v_1} v_2^{v_2} \exp(-v_1).
\end{equation}
Using~\eqref{wtildeprel} with $x = \left(\tilde{W}/(6^{1/\alpha} K_p)\right)^\alpha$, $v_1 = q/\alpha$ and $v_2 = 1$, we have:
\begin{equation}\label{wtilde33}
\frac{\tilde{W}^q}{6^{q/\alpha}K_p^q}\exp\left(-\frac{\tilde{W}^\alpha}{6 K_p^\alpha}\right) \leq \left(\frac{q}{e\alpha}\right)^{q/\alpha}.
\end{equation} 
It follows from~\eqref{wtilde33} and~\eqref{wtilde2}, that 
\begin{equation}\label{wtilde44}
\mathbb{E} \tilde{W}^q \leq 2K_p^q \left(\frac{6q}{e\alpha}\right)^{q/\alpha}.
\end{equation}
Consequently, we have
\begin{eqnarray}
\left(\mathbb{E} \norm{X}_\infty^q\right)^{1/q} &\leq& \left(\mathbb{E} \tilde{W}^q  \right)^{1/q} + 2^{1/\alpha} K_p (\log p)^{1/\alpha}\\&\leq& K_p \left(2^{1/q} \left(\frac{6q}{e\alpha}\right)^{1/\alpha} + 2^{1/\alpha} (\log p)^{1/\alpha}\right).
\end{eqnarray}
This proves~\eqref{part1orlicz} and completes the proof of Lemma \ref{orlicz}.
\end{proof}
\subsection{Proof of Corollary~\ref{cor:NonUniformPolyCLT}}
To begin with, note that:
\begin{equation}\label{orliczbd}
\max_{1\leq i\leq n}\max_{1\leq j\leq p} \max\{\|X_i(j)\|_{\psi_\alpha} , \|Y_i(j)\|_{\psi_\alpha}\} \leq \max\{K_p, \mathcal{W}_\alpha \sigma_{\max}\} =: \mathfrak{J}_{p,\alpha}. 
\end{equation}
Also, let $\Gamma_{n,p} := \max_{1\leq j\leq p} n^{-1} \sum_{i=1}^n \mathbb{E} \left[Y_i^2(j)\right]$. By Theorem 3.4 of~\cite{kuchibhotla2018moving} we have for any $t \geq 0$, with probability at least $1- 3e^{-t}$,
\begin{equation}\label{sg1}
\|U_{n,k}\| \leq 7\sqrt{\Gamma_{n,p}(t+\log p)} + n^{-\frac{1}{2}} \mathfrak{V}_\alpha \mathfrak{J}_{p,\alpha} \log^{1/\alpha}(2n)(t+\log p)^{1/\alpha^*}~,
\end{equation}
where $\alpha^* := \min\{\alpha,1\}$ and $\mathfrak{V}_\alpha > 0$ is some constant depending only on $\alpha$. Clearly, $\sqrt{t + \log p} \leq \sqrt{t} + \sqrt{\log p}$, and since the function $x \mapsto x^{1/\alpha^*}$ is convex, we have:
$$(t+\log p)^{1/\alpha^*} \leq 2^{(1/\alpha^*)-1}\left(t^{1/\alpha^*} + (\log p)^{1/\alpha^*}\right).$$
It thus follows from \eqref{sg1} that with probability at least $1-3e^{-t}$ we have:
\begin{equation}\label{sg12}
\|U_{n,k}\| \leq \mathfrak{S}_{n,p}\sqrt{\log p} + \mathfrak{T}_{p,\alpha} n^{-\frac{1}{2}}\log^{1/\alpha}(2n)(\log p)^{1/\alpha^*} + \mathfrak{S}_{n,p} \sqrt{t} +  \mathfrak{T}_{p,\alpha} n^{-\frac{1}{2}}\log^{1/\alpha}(2n) t^{1/\alpha^*}~,
\end{equation}
where $\mathfrak{S}_{n,p} = 7\sqrt{\Gamma_{n,p}}$ and $\mathfrak{T}_{p,\alpha} = 2^{(1/\alpha^*)-1}\mathfrak{V}_\alpha \mathfrak{J}_{p,\alpha}$. Define:
$$r_{n,m} = 4\sqrt{2} \left(\mathfrak{S}_{n,p}\sqrt{\log (epn)} + \mathfrak{T}_{p,\alpha} n^{-\frac{1}{2}}\log^{1/\alpha}(2n)(\log (epn))^{1/\alpha^*}\right).$$
Then, $r_{n,m} \geq  2\mathfrak{S}_{n,p}\sqrt{\log p} + 2\mathfrak{T}_{p,\alpha} n^{-\frac{1}{2}}\log^{1/\alpha}(2n)(\log p)^{1/\alpha^*}$, and $t = t_{r,n,p,\alpha}$ is such that 
\begin{equation}\label{tgeq}
\mathfrak{S}_{n,p} \sqrt{t} +  \mathfrak{T}_{p,\alpha} n^{-\frac{1}{2}}\log^{1/\alpha}(2n) t^{1/\alpha^*} \leq \frac{r}{2}.
\end{equation}
From now on we write
\[
r_{n,m} = \Theta\sigma_{\max}\sqrt{\log(epn)} + \Theta_{\alpha}K_pn^{-1/2}(\log(2n))^{1/\alpha}(\log(epn))^{1/\alpha^*}.
\]
Under the assumptions
\begin{align*}
\sup_{r \ge r_{n,m}}\max_{0 \le k\le n}\,r^m\mathbb{P}(\|U_{n,k}\| > r) &\le \Theta_m\frac{\sigma_{\max}^m}{epn} + \Theta_m\frac{K_p^{m+2}(\log(2n))^{(m+2)/\alpha}}{n^{(m+2)/2}\log(epn)\sigma_{\max}^2}\\
&\le \Theta_m\frac{\sigma_{\max}^{m}}{epn} + \Theta_m\frac{K_p^{m+2}(\log(2n))^{(m+2)/\alpha - 1}}{n^{(m+2)/2}\sigma_{\max}^2}\\
&\le \Theta_m\frac{K_p^{m+2}}{\sigma_{\max}^2n^{2/3}}.
\end{align*}
Since $\bar{L}_{n,m} \le \Theta_m\sigma_{\max}^{m+1}L_n(\log(ep))^{(m+1)/2}$, we have
\begin{align*}
&\frac{r_{n,m}^m\bar{L}_{n,m}}{L_n}\left(\frac{\log(ep)}{nL_n}\right)^{(m+1)/3}\\ 
&\qquad\le \sigma_{\max}^{2m + 1}(\log(epn))^{m/2}\left(\frac{\log(ep)}{nL_n}\right)^{(m+1)/3}\\
&\qquad\quad+ \sigma_{\max}^{m+1}(\log(ep))^{(m+1)/2}K_p^mn^{-m/2}(\log(2n))^{m/\alpha}(\log(ep))^{m/\alpha^*}\left(\frac{\log(ep)}{nL_n}\right)^{(m+1)/3}\\
&\qquad\le \sigma_{\max}^{2m+1}\left(\frac{(\log(epn))^{(8m+5)/(2m+2)}}{nL_n}\right)^{(m+1)/3}\\
&\qquad\quad+ \sigma_{\max}^{m+1}K_p^m(\log(2n))^{m/\alpha}\left(\frac{(\log(epn))^{(6m + 5m\alpha^* + 5\alpha^*)/(5m\alpha^* + 2\alpha^*)}}{nL_n}\right)^{(5m + 2)/6}\\
&\qquad\le K_p^{m}\sigma_{\max}^{m+1}\left(\frac{(\log(epn))^{1 + 3/\alpha^*}}{nL_n}\right)^{(m+1)/3}.
\end{align*}
Next observe that
\[
\varepsilon_n^{m+1}\Phi_{AC,0} \le \Phi_{AC,0}\left(\frac{L_n^2\log^4(ep)}{n}\right)^{(m+1)/6} \le \Phi_{AC,m}\left(\frac{L_n^2\log^4(ep)}{n}\right)^{1/6} = \Theta\Phi_{AC,m}\varepsilon_{n,m}.
\]
It is now enough to bound $\log(ep)r_{n,m}^m\varepsilon_n^{-2}M_n(2^{2m/3}\varepsilon_n)$. From the definition of $r_{n,m}$, we have
\begin{align*}
\log(ep)r_{n,m}^m\varepsilon_n^{-2}M_n(2^{2m/3}\varepsilon_n) &\le \Theta^m\log(ep)\sigma_{\max}^m(\log(epn))^{m/2}\left[\varepsilon_n^{-2}M_n(\varepsilon_n)\right]\\
&\quad+ \Theta_{\alpha}^m\log(ep)K_p^mn^{-m/2}(\log(2n))^{m/\alpha}(\log(epn))^{m/\alpha^*}\left[\varepsilon_n^{-2}M_n(\varepsilon_n)\right].
\end{align*}
From the assumptions
\[
M_n(\varepsilon_n) \le \Theta_{\alpha,q}K_p^{2 + 3q}(\log(ep))^{(2 + 3q)/\alpha + q}(nL_n)^{-q},
\]
and so for any $q \ge 0$
\begin{align*}
\varepsilon_n^{-2}M_n(\varepsilon_n) &\le \Theta_{\alpha,q}K_p^{2 + 3q}n^{1/3}L_n^{-2/3}(\log(ep))^{-4/3}(\log(ep))^{(2 + 3q)/\alpha + q}(nL_n)^{-q}\\
&= \Theta_{\alpha, q}K_p^{2 + 3q}L_n^{-2/3 - q}(\log(ep))^{(2 + 3q)/\alpha + q - 4/3}n^{-q + 1/3}.
\end{align*}
This implies that
\begin{align*}
&\log(ep)r_{n,m}^m\varepsilon_n^{-2}M_n(2^{2m/3}\varepsilon_n)\\ 
&\le \Theta^m\log(ep)\sigma_{\max}^m(\log(epn))^{m/2}\Theta_{\alpha, q}K_p^{2 + 3q}L_n^{-2/3 - q}(\log(ep))^{(2 + 3q)/\alpha + q - 4/3}n^{-q + 1/3}\\
&\quad+ \Theta_{\alpha}^m\log(ep)K_p^mn^{-m/2}(\log(2n))^{m/\alpha}(\log(epn))^{m/\alpha^*}\Theta_{\alpha, q}K_p^{2 + 3q}L_n^{-2/3 - q}(\log(ep))^{(2 + 3q)/\alpha + q - 4/3}n^{-q + 1/3}\\
&\le \Theta^m\Theta_{\alpha,q}\sigma_{\max}^m(K_p/L_n^{1/3})^{2 + 3q}(\log(epn))^{(2 + 3q)/\alpha + q - 1/3 + m/2}n^{-q + 1/3}\\
&\quad+ \Theta_{\alpha}^m\Theta_{\alpha,q}\Theta_{\alpha,m}K_p^m(K_p/L_n^{1/3})^{2 + 3q}(\log(epn))^{(2 + 3q)/\alpha + q - 1/3 + m/\alpha^*}n^{-q + 1/3}\\
&\le \Theta_{\alpha,q,m}K_p^m(K_p/L_n^{1/3})^{2 + 3q}(\log(epn))^{(2 + 3q)/\alpha + q - 1/3 + m/\alpha^*}n^{-q + 1/3}.
\end{align*}
If $1 < \alpha \le 2$, then take
\[
q = \frac{2 + (m+1)\alpha}{3(\alpha - 1)}\quad\Rightarrow\quad q - \frac{1}{3} = \frac{3 + m\alpha}{3(\alpha - 1)}.
\]
This implies for $1 < \alpha \le 2$
\begin{align*}
\log(ep)r_{n,m}^m\varepsilon_n^{-2}M_n(2^{2m/3}\varepsilon_n) &\le \Theta_{\alpha,q,m}K_p^m(K_p/L_n^{1/3})^{(m+1)\alpha/(\alpha - 1)}\left(\frac{\log^4(epn)}{n}\right)^{(3 + m\alpha)/(3\alpha - 3)}\\
&\le \Theta_{\alpha,q,m}K_p^m(K_p/L_n^{1/3})^{(m+1)\alpha/(\alpha - 1)}\left(\frac{\log^4(epn)}{n}\right)^{1/(\alpha - 1)}.
\end{align*}
If $0 < \alpha \le 1$, then take $q$ such that
\[
q - \frac{1}{3} = \frac{12}{\alpha} + 2m.
\]
This implies for $0 < \alpha \le 1$
\begin{align*}
\log(ep)r_{n,m}^m\varepsilon_n^{-2}M_n(2^{2m/3}\varepsilon_n) \le \Theta_{\alpha,q,m}K_p^m(K_p/L_n^{1/3})^{3 + 6m + 36/\alpha}\left(\frac{(\log(epn))^{5/4 + 3/\alpha}}{n}\right)^{12/\alpha + 2m}.
\end{align*}
Combining these bounds with Theorem~\ref{thm:NonUniformPolyCLT}, we get Corollary~\ref{cor:NonUniformPolyCLT}.
\section{Proof of Proposition~\ref{prop:MAMNonUniform}}
Consider the event
\[
\mathcal{E} := \left\{\|R_n\| \le \delta\right\}.
\]
Observe that for any $r > 0$
\begin{align*}
\mathbb{P}(\sqrt{n}\|\hat{\theta} - \theta_0\| \le r) &= \mathbb{P}\left(\|S_n^{\psi} + R_n\| \le r\right)\\
&= \mathbb{P}\left(\{\|S_n^{\psi} + R_n\| \le r\}\cap\mathcal{E}\right) + \mathbb{P}\left(\{\|S_n^{\psi} + R_n\| \le r\}\cap\mathcal{E}^c\right).
\end{align*}
It now follows that
\[
\mathbb{P}(\|S_n^{\psi}\| \le r - \delta) \le \mathbb{P}\left(\{\|S_n^{\psi} + R_n\| \le r\}\cap\mathcal{E}\right) \le \mathbb{P}(\|S_n^{\psi}\| \le r + \delta),
\]
and
\[
\mathbb{P}\left(\{\|S_n^{\psi} + R_n\| \le r\}\cap\mathcal{E}^c\right) \le \mathbb{P}(\|R_n\| > \delta).
\]
This implies that
\begin{align*}
\left|\mathbb{P}(\sqrt{n}\|\hat{\theta} - \theta_0\| \le r) - \mathbb{P}(\|Y^{\psi}\| \le r)\right| &\le |\mathbb{P}(\|S_n^{\psi}\| \le r - \delta) - \mathbb{P}(\|Y^{\psi}\| \le r - \delta)|\\
&\qquad+ |\mathbb{P}(\|S_n^{\psi}\| \le r + \delta) - \mathbb{P}(\|Y^{\psi}\| \le r + \delta)|\\
&\qquad+ \mathbb{P}(r - \delta \le \|Y^{\psi}\| \le r + \delta)\\
&\qquad+ \mathbb{P}(\|R_n\| > \delta).
\end{align*}
If $r - \delta \le 4\delta$ then following the proof of Theorem~\ref{thm:NonUniformPolyCLT} we obtain
\begin{align*}
|\mathbb{P}(\|S_n^{\psi}\| \le r - \delta) - \mathbb{P}(\|Y^{\psi}\| \le r - \delta)| &\le |\mathbb{P}(\|S_n^{\psi}\| \le 4\delta) - \mathbb{P}(\|Y^{\psi}\| \le 4\delta)| + 2\Phi_{AC,0}^{\psi}\delta\\
&\le \frac{\Delta_{n,m}^{\psi}}{(4\delta)^m} + \frac{\Phi_{AC,0}^{\psi}(4\delta)^{m+1}}{2(4\delta)^{m}}.
\end{align*}
Since $r \le 5\delta$, we have $(4\delta)^{-m} \le (4r/5)^{-m}$ and hence
\begin{align*}
|\mathbb{P}(\|S_n^{\psi}\| \le r - \delta) - \mathbb{P}(\|Y^{\psi}\| \le r - \delta)| &\le \frac{(5/4)^m\Delta_{n,m}^{\psi} + 5^m2\Phi_{AC,0}^{\psi}\delta^{m+1}}{r^m}.
\end{align*}
If $r - \delta > 4\delta$ then $r - \delta > 4r/5$ and hence
\begin{align*}
|\mathbb{P}(\|S_n^{\psi}\| \le r - \delta) - \mathbb{P}(\|Y^{\psi}\| \le r - \delta)| &\le \frac{\Delta_{n,m}^{\psi}}{(r - \delta)^m} \le \frac{(5/4)^m\Delta_{n,m}^{\psi}}{r^m}.
\end{align*}
Thus for all $r \ge 0$,
\[
|\mathbb{P}(\|S_n^{\psi}\| \le r - \delta) - \mathbb{P}(\|Y^{\psi}\| \le r - \delta)| \le \frac{(5/4)^m\Delta_{n,m}^{\psi} + 5^m2\Phi_{AC,0}^{\psi}\delta^{m+1}}{r^m}.
\]
Since $r + \delta > r$, it follows that
\[
|\mathbb{P}(\|S_n^{\psi}\| \le r + \delta) - \mathbb{P}(\|Y^{\psi}\| \le r + \delta)| \le \frac{\Delta_{n,m}^{\psi}}{r^m}.
\]
Combining these inequalities, we get
\begin{align*}
r^m\left|\mathbb{P}(\sqrt{n}\|\hat{\theta} - \theta_0\| \le r) - \mathbb{P}(\|Y^{\psi}\| \le r)\right| &\le 2\left[(5/4)^m\Delta_{n,m}^{\psi} + 5^m\Phi_{AC,0}^{\psi}\delta^{m+1}\right] + \Phi_{AC,m}^{\psi}\delta\\ &\qquad+ r^m\mathbb{P}(\|R_n\| > \delta).
\end{align*}
This completes the proof.
\section{Proof of Theorem \ref{thm:largedeviation}}\label{appsec:Theorem31}
We will require the following notations in the proof of Theorem \ref{thm:largedeviation}.
\begin{equation}\label{eq:DefDeltamIII}
\begin{split}
U_{n,k} &:= n^{-1/2}\left(X_1 + X_2 + \ldots + X_k + Y_{k+1} + Y_{k+2} + \ldots + Y_n\right),\\
W_{n,k} &:= n^{-1/2}\left(X_1 + X_2 + \ldots + X_{k-1} + Y_{k+1} + \ldots + Y_n\right),\\
\Delta_{n,k}(r) &:= \left|\mathbb{P}\left(\norm{U_{n,k}} \le r\right) - \mathbb{P}\left(\norm{Y} \le r\right)\right|.
\end{split}
\end{equation}
Below, we state an elementary fact about Taylor series expansion.
\begin{lem}\label{lem:Taylor}
For any thrice differentiable function $f:\mathbb{R}^p\to\mathbb{R}$, we have
\begin{equation}\label{eq:Taylor}
f(y + xn^{-1/2}) -  f(y) - n^{-1/2}x^{\top}\nabla f(y) - \frac{1}{2n}x^{\top}\nabla_2f(y)x = \mbox{Rem}_n(y, x),
\end{equation}
where
\[
\left|\mbox{Rem}_n(y, x)\right| \le \min\left\{\frac{\norm{x}^3}{6n^{3/2}}\sup_{0 \le \theta \le 1}\norm{D^3f(y + x\theta n^{-1/2})}_1, \frac{\norm{x}^2}{n}\sup_{0\le \theta\le 1}\norm{D^2f(y + x\theta n^{-1/2})}_1\right\}.
\]
\end{lem}
\begin{proof}
The proof follows directly from the mean value theorem and the definition of $D^2, D^3.$
\end{proof} 
As sketched in Section \ref{sec:Outline}, the proof of Theorem \ref{thm:largedeviation} proceeds through several steps, the details of which are described in this section.
\begin{lem}\label{lem:DeltankIntoUnminus1k}
Set $B = 2(\Phi_1 + 1)/H$ and $f_n(r) = B(r^2 + \log(en))n^{-1/2}$. For any $r,\varepsilon \ge 0$,
\begin{align*}
\Delta_n(r) &\le \mathbb{P}(r - \varepsilon \le \|Y\| \le r + \varepsilon) + \frac{8C_2\varepsilon^{-2}\beta}{\Phi_0H^2n^{1+\Phi_1}}\mathbb{P}(\|Y\| > r)\\
&\qquad+ \frac{C_3\varepsilon^{-3}L_ne^{\mathfrak{C}f_n(r)\log(ep)/\varepsilon}}{6n^{1/2}}\max_{0\le k\le n}\mathbb{P}(a_n(r) \le \|U_{n-1,k}\| \le b_n(r)),
\end{align*}
where $a_n(r) = r - \varepsilon - f_n(r)$ and $b_n(r) = (1 - 1/n)^{-1}(r + \varepsilon + f_n(r))$.
\end{lem}
\begin{proof}
By smoothing lemma,
\[
\Delta_{n,k}(r) \le \sum_{j=1}^k I_j + \mathbb{P}(r - \varepsilon \le \|Y\| \le r + \varepsilon),
\]
where
\begin{align*}
I_j &= |\mathbb{E}[\varphi(W_{n,j} + n^{-1/2}X_j) - \varphi(W_{n,j} + n^{-1/2}Y_j)]|\\
&= \left|\int \mathbb{E}\mbox{Rem}_n(W_{n,j}, x)\zeta(dx)\right|\\
&\le \int_{\|x\| \le n^{1/2}f_n(r)}\mathbb{E}|\mbox{Rem}_n(W_{n,j},x)||\zeta|(dx) + \int_{\|x\| > n^{1/2}f_n(r)} \mathbb{E}|\mbox{Rem}_n(W_{n,j},x)||\zeta|(dx)\\
&=: I_j^{(1)} + I_j^{(2)}.
\end{align*}
To bound $I_j^{(1)}$, we use the stability property of third derivative bound.
\begin{align*}
|\mbox{Rem}_n(W_{n,j},x)| &\le \frac{1}{2}\sum_{j_1,j_2,j_3 = 1}^p \frac{|x(j_1)x(j_2)x(j_3)|}{n^{3/2}}\int_0^1 (1 - t)^2D_{j_1j_2j_3}(W_{n,j} + tn^{-1/2}x)dt\\
&\le \frac{1}{6}\sum_{j_1,j_2,j_3=1}^p \frac{|x(j_1)x(j_2)x(j_3)|}{n^{3/2}}e^{\mathfrak{C}n^{-1/2}\|x\|\log(ep)/\varepsilon}D_{j_1j_2j_3}(W_{n,j}).
\end{align*}
This yields
\begin{align*}
&\int_{\|x\| \le n^{1/2}f_n(r)} \mathbb{E}|\mbox{Rem}_n(W_{n,j},x)||\zeta|(dx)\\ 
&\le \frac{e^{\frac{\mathfrak{C}f_n(r)\log(ep)}{\varepsilon}}}{n^{3/2}}\sum_{j_1,j_2,j_3=1}^p \int |x(j_1)x(j_2)x(j_3)||\zeta|(dx)\mathbb{E}[D_{j_1j_2j_3}(W_{n,j})\mathbbm{1}\{r - \varepsilon - f_n(r) \le \|W_{n,j}\| \le r + \varepsilon + f_n(r)\}]\\
&\le \frac{C_3\varepsilon^{-3}e^{\mathfrak{C}f_n(r)\log(ep)/\varepsilon}L_n}{n^{3/2}} \mathbb{P}\left(r - \varepsilon - f_n(r) \le \|W_{n,j}\| \le r + \varepsilon + f_n(r)\right).
\end{align*}
To bound $I_j^{(2)},$ we use $|\mbox{Rem}_n(W_{n,j},x)| \le C_2\varepsilon^{-2}n^{-1}\|x\|^2$ and get
\begin{align*}
I_j^{(2)} &\le C_2\varepsilon^{-2}n^{-1}\int_{\|x\| > n^{1/2}f_n(r)} \|x\|^2|\zeta|(dx)\\
&\le C_2\varepsilon^{-2}n^{-1}\int_{\|x\| > n^{1/2}f_n(r)}\|x\|^2\frac{\exp(H\|x\|/2)}{\exp(Hn^{1/2}f_n(r)/2)}|\zeta|(dx)\\
&\le \frac{8C_2\varepsilon^{-2}}{nH^2}\int_{\|x\| > n^{1/2}f_n(r)} \frac{\exp(H\|x\|)}{\exp(Hn^{1/2}f_n(r)/2)}|\zeta|(dx)\\
&= \frac{8\beta C_2\varepsilon^{-2}}{nH^2}\exp(-Hn^{1/2}f_n(r)/2) = \frac{8C_2\varepsilon^{-2}\beta}{nH^2}\exp(-(1+\Phi_1)(r^2+\log n))\\ &\le \frac{8C_2\varepsilon^{-2}\beta}{n^{2+\Phi_1}H^2\Phi_0}\mathbb{P}(\|Y\| > r). 
\end{align*}
Combining the bounds on $I_j^{(1)}, I_j^{(2)}$ and summing over $1\le j\le n$, we get
\begin{align*}
\Delta_{n,k}(r) &\le \mathbb{P}(r - \varepsilon \le \|Y\| \le r + \varepsilon)\\
&\quad+ \frac{C_3\varepsilon^{-3}e^{\mathfrak{C}f_n(r)\log(ep)/\varepsilon}L_n}{n^{1/2}}\max_{1\le j\le n}\mathbb{P}(r - \varepsilon - f_n(r) \le \|W_{n,j}\| \le r + \varepsilon + f_n(r))\\
&\quad+ \frac{8C_2\varepsilon^{-2}\beta}{n^{1+\Phi_1}H^2\Phi_0}\mathbb{P}(\|Y\| > r).
\end{align*}
Since $W_{n,j}$ is identically distributed as $U_{n-1,j}\sqrt{1 - 1/n}$, we get
\[
\mathbb{P}(r - \varepsilon - f_n(r) \le \|W_{n,j}\| \le r + \varepsilon + f_n(r)) \le \mathbb{P}(a_n(r) \le \|U_{n-1,j}\| \le (r + \varepsilon + f_n(r))\sqrt{n/(n-1)}).
\]
Further since $\sqrt{n/(n-1)} \le n/(n-1)$ the result follows.
\end{proof}
\noindent For the following lemmas, we use the following notation:
\begin{equation}\label{eq:3p7}
\begin{split}
\mathbb{P}\left(\norm{Y} >q r\right) &\ge \Phi_0\exp\left(-\Phi_1r^2\right)\\
\mathbb{P}\left(r - \varepsilon \le \norm{Y} \le r + \varepsilon\right) &\le \Phi_2\varepsilon(1 + r)\mathbb{P}\left(\norm{Y} >q r - \varepsilon\right)\\
\mathbb{P}\left(\norm{Y} >q r - \varepsilon\right) &\le \Phi_3\exp\left(\Phi_4(r + 1)\varepsilon\right)\mathbb{P}\left(\norm{Y} >q r\right).
\end{split}
\end{equation}
Note that we have proved in bounds on $\Phi_0,\ldots,\Phi_4$ in Theorem~\ref{thm:density}. Now Define
\[
\Pi := \max\left\{1, (4e^{1/2}\Phi_2\Phi_3)^{4/3}, \frac{eC_3\Phi_3L_ne^{\mathfrak{C}}}{3\Phi_4}, \left(\frac{2eC_3\Phi_3L_ne^{\mathfrak{C}}}{3}\right)^{4/7}, \left(\frac{C_2\beta}{\Phi_0\Phi_4^5H^2}\right)^{4/11}\right\}.
\]
\begin{lem}\label{lem:sixtpt}
Let $T_{n,r} = (r+1)^3n^{-1/2}$. Then for all $n\ge1$ and $0\le k\le n$,
\begin{equation}\label{eq:FirstVersionLargeDeviation}
\left|\frac{\mathbb{P}(\|U_{n,k}\| > r)}{\mathbb{P}(\|Y\| > r)} - 1\right| \le \Pi T_{n,r}^{1/4},
\end{equation}
for all $r\in\mathbb{R}$ satisfying
\begin{equation}\label{eq:TnrSatisfying}
T_{n,r} \le \min\left\{\frac{1}{16\Phi_4^4\Pi}, \frac{1}{2\Phi_4B\log(en)}, \frac{\Pi^{1/3}}{(B\log(en)\log(ep))^{4/3}}\right\}.
\end{equation}
\end{lem}
\begin{proof}
Inequality~\eqref{eq:FirstVersionLargeDeviation} is trivally true for $r \le 0$ for all $n\ge1$. Hence we only consider the case of $r > 0$. We prove the result by induction.

\noindent {\bf Case $n = 1$:} Since $\Phi_4, \Pi \ge 1$, we get $r\ge 0$ satisfying~\eqref{eq:TnrSatisfying} also satisfies $T_{n,r} \le 1$ which is equivalent (for $n=1$) to $(r+1)^3\le1$ which holds only for $r\le0$ the case in which the result is already proved.

Suppose now that the result~\eqref{eq:FirstVersionLargeDeviation} holds for sample sizes $1, 2, \ldots, n-1$. This implies that for any $u$ satisfying~\eqref{eq:TnrSatisfying} (with $n$ replaced by $n-1$ and $r$ replaced by $u$), we get
\begin{equation}\label{eq:InductionHypo}
\frac{\mathbb{P}(\|U_{n-1,k}\| \le u)}{\mathbb{P}(\|Y\| \ge u)} \le 1 + \Pi T_{n-1,u}^{1/4} \le 1 + \Pi/(2\Phi_4\Pi^{1/4}) \le 1 + \Pi^{3/4}/(2\Phi_4).
\end{equation}
The second inequality above follows from~\eqref{eq:TnrSatisfying}. Proceeding to prove for the case of $n$, fix $r > 0$ satisfying~\eqref{eq:TnrSatisfying} and let $\varepsilon = (T_{n,r}\Pi)^{1/4}(r + 1)^{-1}$. Firstly note that
\begin{align*}
\frac{f_n(r)\log(ep)}{\varepsilon} &= \frac{B(r^2 + \log(en))\log(ep)}{n^{1/2}\varepsilon}\\
&\le \frac{B(r+1)^2\log(en)\log(ep)}{n^{1/2}\varepsilon}\\
&\le \frac{B(r+1)^3\log(en)\log(ep)}{n^{1/2}(T_{n,r}\Pi)^{1/4}} = \frac{BT_{n,r}^{3/4}\log(en)\log(ep)}{\Pi^{1/4}}\\
&\le \frac{B\Pi^{1/4}\log(en)\log(ep)}{\Pi^{1/4}(B\log(en)\log(ep))} = 1.
\end{align*}
The last inequality above follows from the fact $r$ satisfies~\eqref{eq:TnrSatisfying}. Substituting this in Lemma~\ref{lem:DeltankIntoUnminus1k}, we get
\begin{equation}\label{eq:LemmaImplication}
\begin{split}
\Delta_n(r) &\le \mathbb{P}(r - \varepsilon \le \|Y\| \le r + \varepsilon) + \frac{8C_2\varepsilon^{-2}\beta}{\Phi_0H^2n^{1+\Phi_1}}\mathbb{P}(\|Y\| > r)\\
&\quad+ \frac{C_3\varepsilon^{-3}L_ne^{\mathfrak{C}}}{6n^{1/2}}\max_{0\le k\le n}\mathbb{P}(\|U_{n-1,k}\| \ge a_n(r)).
\end{split}
\end{equation}
We will now simplify the right hand side terms as multiples of $\mathbb{P}(\|Y\| > r)$. Firstly note that
\[
\mathbb{P}(r - \varepsilon \le \|Y\| \le r +\varepsilon) \le \Phi_2\Phi_3\varepsilon(r+1)e^{\Phi_4\varepsilon(r+1)}\mathbb{P}(\|Y\| > r).
\]
Secondly, since $f_n(r)/\varepsilon \le B\log(en)T_{n,r}^{3/4}/\Pi^{1/4}$, we get $a_n(r) \ge r - c'\varepsilon$ where $c' = 1 + B\log(en)T_{n,r}^{3/4}/\Pi^{1/4}$. Hence
\[
\mathbb{P}(\|U_{n-1,k}\| \ge a_n(r)) \le \mathbb{P}(\|U_{n-1,k}\| > r - c'\varepsilon).
\]
In order to control the right hand side we use the induction hypothesis. For this we need to verify that $r - c'\varepsilon$ satisfies~\eqref{eq:TnrSatisfying} (with $n$ replaced by $n-1$ and $r$ replaced by $r-c'\varepsilon$) which is verified as follows: 
\begin{equation}\label{eq:MainVerify}
\frac{T_{n-1,r-c'\varepsilon}}{T_{n,r}} = \left(\frac{r - c'\varepsilon + 1}{r+1}\right)^{3}\left(\frac{n}{n-1}\right)^{1/2} = \left(1 - \frac{c'\varepsilon}{r+1}\right)^3\left(1 - \frac{1}{n}\right)^{-1/2}.
\end{equation}
Since $T_{n,r} \le 1/(16\Phi_4^4\Pi)$, proving the right hand side of~\eqref{eq:MainVerify} is bounded by 1 proves $T_{n-1,r-c'\varepsilon} \le 1/(16\Phi_4^4\Pi)$. Since $T_{n,r} \le 1/(2\Phi_4^4B\log(en))$, proving the right hand side of~\eqref{eq:MainVerify} is bounded by $\log(en)/\log(e(n-1))$ proves $T_{n-1,r-c'\varepsilon} \le 1/(2\Phi_4^4B\log(e(n-1)))$. Similarly, proving the right hand side of~\eqref{eq:MainVerify} is bounded by $(\log(en)/\log(e(n-1)))^{4/3}$ proves the last inequality for $T_{n-1,r-c'\varepsilon}$. Hence it is enough to show that
\[
\left(1 - \frac{c'\varepsilon}{r+1}\right)^3\left(1 - \frac{1}{n}\right)^{-1/2} \le 1,
\] 
which is the tightest among the three. Equivalently, we need to verify
\[
\left(1 - \frac{c'\varepsilon}{r+1}\right) \le \left(1 - \frac{1}{n}\right)^{1/6}.
\]
Since $1 - (1-1/n)^{1/6} \le 1/n$ for all $n\ge1$, it suffices to verify $1/n \le c'\varepsilon/(r+1) = c'(r+1)^{3/4}n^{-1/8}\Pi^{1/4}/(r+1)^2$ which is equivalent to $(r+1)^{5/4} \le n^{7/8}c'\Pi^{1/4}$. But we know that $r+1\le n^{1/6}$ which implies $(r+1)^{5/4} \le n^{5/24} \le n^{21/24} \le n^{7/8}c'\Pi^{1/4}$ for all $n\ge1$. This shows that $r - c'\varepsilon$ satisfies~\eqref{eq:TnrSatisfying} and hence from~\eqref{eq:InductionHypo}, we get
\[
\mathbb{P}(\|U_{n-1,k}\| > r - c'\varepsilon) \le 
\mathbb{P}(\|Y\| > r - c'\varepsilon)(1 + \Pi^{3/4}/(2\Phi_4)) \le \Phi_3\exp(\Phi_4\varepsilon(r+1))\mathbb{P}(\|Y\| > r).
\]
Substituting these in~\eqref{eq:LemmaImplication}, we get
\[
\frac{\Delta_{n}(r)}{\mathbb{P}(\|Y\| > r)} \le \Phi_2\Phi_3\varepsilon(r+1)e^{\Phi_4\varepsilon(r+1)} + \Phi_3e^{\Phi_4c'\varepsilon(r+1)}\left(1 + \frac{\Pi^{3/4}}{2\Phi_4}\right)\frac{C_3\varepsilon^{-3}L_ne^{\mathfrak{C}}}{6n^{1/2}} + \frac{8C_2\varepsilon^{-2}\beta}{\Phi_0H^2n^{1+\Phi_1}}.
\]
Observe now that $\varepsilon(r+1) = (T_{n,r}\Pi)^{1/4},$
\[
\frac{\varepsilon^{-3}}{n^{1/2}} = \frac{T_{n,r}^{1/4}}{\Pi^{3/4}},\quad\mbox{and}\quad \frac{\varepsilon^{-2}}{n} \le \frac{(r+1)^2}{(T_{n,r}\Pi)^{1/2}n} = \frac{T_{n,r}^{2/3}}{(T_{n,r}\Pi)^{1/2}n^{2/3}} \le \frac{T_{n,r}^{3/2}}{\Pi^{1/2}} \le \frac{T_{n,r}^{1/4}}{32\Phi_4^5\Pi^{7/4}},
\]
where the last two inequalities follows from $n^{-2/3} \le T_{n,r}^{4/3}$ and $T_{n,r}^{5/4} \le 1/(32\Phi_4^5\Pi^{5/4})$. Using these, we obtain
\begin{align*}
\frac{\Delta_{n}(r)}{\mathbb{P}(\|Y\| > r)} &\le \Phi_2\Phi_3(T_{n,r}\Pi)^{1/4}e^{\Phi_4(T_{n,r}\Pi)^{1/4}} + \Phi_3e^{\Phi_4c'(T_{n,r}\Pi)^{1/4}}\left(1 + \frac{\Pi^{3/4}}{2\Phi_4}\right)\frac{C_3L_ne^{\mathfrak{C}}T_{n,r}^{1/4}}{6\Pi^{3/4}}\\ 
&\qquad+ \frac{8C_2\beta T_{n,r}^{1/4}}{32\Phi_0\Phi_4^5H^2\Pi^{7/4}}.
\end{align*}
Note that since $c' \ge 1$,
\begin{align*}
\Phi_4(T_{n,r}\Pi)^{1/4} \le \Phi_4c'(T_{n,r}\Pi)^{1/4} &\le \Phi_4(T_{n,r}\Pi)^{1/4}\left(1 + \frac{B\log(en)T_{n,r}^{3/4}}{\Pi^{1/4}}\right)\\ 
&= \Phi_4(T_{n,r}\Pi)^{1/4} + B\Phi_4\log(en)T_{n,r} \le \frac{1}{2} + \frac{1}{2} = 1.
\end{align*}
Therefore,
\begin{align*}
\frac{\Delta_{n}(r)}{\mathbb{P}(\|Y\| > r)} &\le T_{n,r}^{1/4}\left[e^{1/2}\Phi_2\Phi_3\Pi^{1/4} + \frac{e\Phi_3(1 + \Pi^{3/4}/(2\Phi_4))C_3L_ne^{\mathfrak{C}}}{6\Pi^{3/4}} + \frac{C_2\beta}{4\Phi_0\Phi_4^5H^2\Pi^{7/4}}\right]\\
&\le \Pi T_{n,r}^{1/4},
\end{align*}
where the last inequality follows from the definition of $\Pi$. Hence the result is proved.
\end{proof}
Before improving $T_{n,r}^{1/4}$ to $T_{n,r}^{1/3}$, we prove the following sharp large deviation for $r \le \mu$.
\begin{lem}\label{lem:rSmallMu}
For all $r \le \mu := \mbox{median}(\|Y\|)$ and $n\ge1$, we have
\[
\max_{0\le k\le n}\left|\frac{\mathbb{P}(\|U_{n,k}\| > r)}{\mathbb{P}(\|Y\| > r)} - 1\right| \le \tilde{\Pi}_nT_{n,r}^{1/3},
\]
where
\[
\tilde{\Pi}_n := 4 + 12\Phi_{AC,0}{(8C_0e^{\mathfrak{C}}L_n\log^2(ep))^{\frac{1}{3}}} + \frac{12\Phi_{AC,0}\log(ep)\beta^{\frac{1}{3}}\log(8C_0n)}{Hn^{1/3}} + \frac{5.1\log(ep)}{C_0n^{5/6}}.
\]
\end{lem}
\begin{proof}
From the proof of Theorem~\ref{thm:OptimalCLT}, it follows that 
\[
\delta_{n,0} \le \frac{2}{2^n} + 6\Phi_{AC,0}\left[\left(\frac{8C_0e^{\mathfrak{C}}L_n\log^2(ep)}{n^{1/2}}\right)^{1/3} + \nu_{2+\tau}\left(\frac{8C_0(\log(ep))^{\tau+1}}{n^{\tau/2}}\right)^{1/(2+\tau)}\right].
\]
Since $\beta = \int \exp(H\|x\|)|\zeta|(dx)$, we get
\[
\nu_{2+\tau}^{2+\tau} \le \int \left(\frac{2+\tau}{He}\right)^{2+\tau}\exp(H\|x\|)|\zeta|(dx) \le \beta\left(\frac{2+\tau}{He}\right)^{2+\tau}.
\]
Hence $\nu_{2+\tau} \le e^{-1}(2+\tau)H^{-1}\beta^{1/3}$ for all $\tau\ge1$. We now choose $2+\tau = \log(8C_0n/\log(ep))$ in the bound on $\delta_{n,0}$ and simplifying the bound, we get
\[
\delta_{n,0} \le \frac{2}{2^n} + 6\Phi_{AC,0}\frac{(8C_0e^{\mathfrak{C}}L_n\log^2(ep))^{1/3}}{n^{1/6}} + 6\Phi_{AC,0}\frac{\log(ep)\beta^{1/3}\log(8C_0n/\log(ep))}{Hn^{1/2}}.
\]
The calculation above requires $\log(8C_0n/\log(ep)) \ge 3$. In case $\log(8C_0n/\log(ep)) \le 3$ then $8C_0n/\log(ep) \le e^3$ or equivalently $1 \le e^3\log(ep)/(8C_0n)$. Since $\delta_{n,0} \le 1$, this implies that $\delta_{n,0} \le e^3\log(ep)/(8C_0n)$. Combining both cases, we get
\[
\delta_{n,0}n^{1/6} \le 2 + 6\Phi_{AC,0}{(8C_0e^{\mathfrak{C}}L_n\log^2(ep))^{1/3}} + \frac{6\Phi_{AC,0}\log(ep)\beta^{1/3}\log(8C_0n)}{Hn^{1/3}} + \frac{e^3\log(ep)}{8C_0n^{5/6}}.
\]

Now note that for all $r \le \mu$, $\mathbb{P}(\|Y\| > r) \ge 1/2$ and hence for $r \le \mu$,
\begin{align*}
\frac{\delta_{n,0}T_{n,r}^{-1/3}}{\mathbb{P}(\|Y\| > r)} &\le \frac{2\delta_{n,0}n^{1/6}}{r+1} \le 2\delta_{n,0}n^{1/6}\\ 
&\le 4 + 12\Phi_{AC,0}{(8C_0e^{\mathfrak{C}}L_n\log^2(ep))^{\frac{1}{3}}} + \frac{12\Phi_{AC,0}\log(ep)\beta^{\frac{1}{3}}\log(8C_0n)}{Hn^{1/3}} + \frac{5.1\log(ep)}{C_0n^{5/6}}\\ 
&=: \tilde{\Pi}_n.
\end{align*}
\end{proof}
\noindent Let us now introduce (or recall) a few notations:
\begin{align*}
&\Pi = \max\left\{1, (4e^{1/2}\Phi_2\Phi_3)^{4/3}, \frac{eC_3\Phi_3L_ne^{\mathfrak{C}}}{3\Phi_4}, \left(\frac{2eC_3\Phi_3L_ne^{\mathfrak{C}}}{3}\right)^{4/7}, \left(\frac{C_2\beta}{\Phi_0\Phi_4^5H^2}\right)^{4/11}\right\}\\
&M := \max\left\{\sqrt{2}(\Pi + \tilde{\Pi}_n),\left(112\Phi_2 + 83C_3L_n\right)^{\frac{4}{3}}, \left(\frac{48C_2}{\Phi_4^{16/5}H^2}\right)^{\frac{10}{23}}, 36\left(C_3L_n\Phi_2\right)^{\frac{2}{3}}, \left(\frac{124C_3L_n}{(\mu+1)^{17/16}n^{5/32}}\right)^{2}\right\},\\
&F_n := \min\left\{\frac{1}{16\Phi_4^4\Pi}~,~ \frac{1}{2\Phi_4B\log(en)}, \frac{\Pi^{1/3}}{(B\log(en)\log(ep))^{4/3}}\right\}~,\\
&G_n := \min \left\{\left(\frac{1}{81\Phi_4^4 M}\right)^{4/5},~\left(\frac{M^{1/4}}{6\mathfrak{C}B\log(en)\log(ep)}\right)^{3/2}\right\}~,\\
&\mathfrak{B}_0 := \min\left\{\left(\frac{F_n}{\sqrt{2}}\right)^{1/3}, G_n^{1/3}\right\},\quad \mathfrak{B}_s := \mathfrak{B}_0\left(1 + \frac{17M^{1/4}}{6(\mu+1)^{17/16}n^{5/32}}\right)^{-s},\\
& I_{s,M} := \left(-\infty, -1 + \mathfrak{B}_sn^{1/6}\right),\quad\alpha(s) := \frac{1-4^{-s-1}}{3},\quad\varepsilon_{s,M} := \frac{M^{1/4}T_{n,r}^{\alpha(s)}}{r+1}.
\end{align*}
\begin{lem}\label{sixtph}
  For every $n \geq {4}$, $0\leq s\leq l := \lfloor \log(en)\rfloor$, we have:
  \begin{equation}\label{eq:ResultFinalLargeDev}
  \max_{0\le k\le n-l+s}\left|\frac{\mathbb{P}(\|U_{n-l+s,k}\| > r)}{\mathbb{P}(\|Y\| > r)} - 1\right| ~\leq~ M T_{n,r}^{\alpha(s)}\quad\mbox{for all}\quad r\in I_{s,M}.
  \end{equation}
\end{lem}
\begin{proof}
  We prove~\eqref{eq:ResultFinalLargeDev} by induction on $s$. But first we note the following. Fix $0 \le s \le l$. If $r \le \mu$, then by Lemma~\ref{lem:rSmallMu} we get the result since $M \ge \sqrt{2}\tilde{\Pi}_n$, $\tilde{\Pi}_{n-l+s} \leq 2^{1/3}\tilde{\Pi}_n$ and $T_{n-l+s,r} \le \sqrt{2}\,T_{n,r}$. From now on we freely assume during the induction $r \ge \mu$.
  
  For $s=0$, the claim becomes:
  \[
  \Delta_{n-l} \leq M T_{n,r}^{1/4} \mathbb{P}(\|Y\| > r)\quad\mbox{for all}\quad n\ge3, \,r \leq -1 + \mathfrak{B}_0 n^{1/6}.
  \] 
  Towards showing this, note that $r \leq -1 + \mathfrak{B}_0n^{-1/6}$ implies $T_{n,r} \leq \mathfrak{B}_0^3$, so 
  \[
  T_{n-l,r} = \frac{(r+1)^3}{(n-l)^{1/2}} \leq \frac{(r+1)^3}{\left(n/2\right)^{1/2}} = \sqrt{2}\,T_{n,r} \leq \sqrt{2}\,\mathfrak{B}_0^3 \leq F_n \leq F_{n-l}.
  \]
  In the second inequality above, we used the fact that $l \leq n/2$ for $n \geq 4$. Hence, by Lemma \ref{lem:sixtpt} we have
  $$\frac{\Delta_{n-l}(r)}{\mathbb{P}(\|Y\|> r)} \leq \Pi T_{n-l,r}^{1/4} \leq 2^{1/8}\Pi T_{n,r}^{1/4}\leq M T_{n,r}^{\alpha(0)}~.$$
  The case $s=0$ is now verified.
  
  We now assume that~\eqref{eq:ResultFinalLargeDev} holds for $0,\ldots,s-1$, and aim at proving it for $s$. Take $r \in I_{s,M}$. Let us set $\varepsilon = \varepsilon_{s,M}$ for the rest of the proof. First, note that by Lemma \ref{lem:DeltankIntoUnminus1k}, we have
  \begin{align*}
  \Delta_{n-l+s, k}(r) &\le \mathbb{P}(r - \varepsilon \le \|Y\| \le r + \varepsilon) + \frac{8C_2\varepsilon^{-2}\beta}{\Phi_0H^2(n-l+s)}\mathbb{P}(\|Y\| > r)\\
  &\quad+ \frac{C_3\varepsilon^{-3}L_{n}e^{\mathfrak{C}f_{n-l+s}(r)\log(ep)/\varepsilon}}{6(n-l+s)^{1/2}}\max_{0\le k\le n-l+s-1}\mathbb{P}(a_{n-l+s}(r) \le \|U_{n-l+s-1,k}\| \le b_{n-l+s}(r)),
  \end{align*}
  where $a_{n-l+s}(r) = r-\varepsilon-f_{n-l+s}(r),$ and $b_{n-l+s}(r) = (1-{1}/{(n-l+s)})^{-1}(r+\varepsilon+f_{n-l+s}(r))$. Since $f_{n-l+s}(r) \le 2f_n(r)$, $a_{n-l+s}(r) \ge A_n(r) := r - \varepsilon - 2f_n(r)$ and $b_{n-l+s}(r) \le B_n(r) := (1-2/n)^{-1}(r + \varepsilon + 2f_n(r))$.
  Consequently, we have:
  \begin{align}
  \Delta_{n-l+s,k}(r) &\le \mathbb{P}(r - \varepsilon \le \|Y\| \le r + \varepsilon) + \frac{16C_2\varepsilon^{-2}\beta}{\Phi_0H^2n}\mathbb{P}(\|Y\| > r)\label{deltanls}\\
  &\quad+ \frac{\sqrt{2}C_3\varepsilon^{-3}L_{n}e^{2\mathfrak{C}f_{n}(r)\log(ep)/\varepsilon}}{6n^{1/2}}\max_{0\le k\le n-l+s-1}\mathbb{P}(A_n(r) \le \|U_{n-l+s-1,k}\| \le B_n(r)),
  \end{align}
  To simplify the probability on the right hand side, note that since $r \in I_{s,M}$, we have:
  $$T_{n,r} \leq \mathfrak{B}_0^3 \leq G_n \leq \left(\frac{M^{1/4}}{6\mathfrak{C}B\log(en)\log(ep)}\right)^{3/2}$$
  \begin{equation}\label{eq:FnrbyVarepsilonBound}
  \frac{f_n(r)}{\varepsilon} \le \frac{B\log(en)(r+1)^2}{n^{1/2}\varepsilon} = \frac{B\log(en)(r+1)^3}{n^{1/2}M^{1/4}T_{n,r}^{\alpha(s)}} \le \frac{B\log(en)T_{n,r}^{2/3}}{M^{1/4}} \le \frac{1}{6\mathfrak{C}\log(ep)} \le \frac{1}{6}.
  \end{equation}
  This implies $\varepsilon + 2f_n(r) \le 4\varepsilon/3$ and so, $A_n(r) = r - \varepsilon - 2f_n(r) \ge r - 4\varepsilon/3.$ Further for $n\ge4$,
  \begin{align*}
  B_n(r) &= r + \varepsilon + 2f_n(r) + [(1 - 2/n)^{-1}-1](r+\varepsilon+2f_n(r))\\
  &\le r + \varepsilon + 2f_n(r) + 4(r+\varepsilon+2f_n(r))/n\\
  &\le r+2(\varepsilon+2f_n(r))+\frac{4r}{n} \le r+\frac{8\varepsilon}{3} + \frac{4r}{n\varepsilon}\varepsilon.
  \end{align*}
  We now bound $r/(n\varepsilon)$. Since $r \le r + 1$, we have
  \begin{equation}\label{eq:rbynvarepsilon}
  \frac{r}{n\varepsilon} \le \frac{(r+1)^2}{nM^{1/4}T_{n,r}^{\alpha(s)}} = \frac{(r+1)^{1+4^{-s-1}}}{n^{(5+4^{-s-1})/6}M^{1/4}} \le \frac{(r+1)^{5/4}}{n^{5/6}M^{1/4}} = \frac{T_{n,r}^{5/12}}{n^{5/8}M^{1/4}} \le \frac{T_{n,r}^{5/3}}{M^{1/4}} \le \frac{T_{n,r}^{2/3}}{4M^{1/4}},
  \end{equation}
  where in the last two inequalities we used $n^{-5/8}\le T_{n,r}^{5/4}$ and $T_{n,r} \le \mathfrak{B}_0^3 \le 1/(16\sqrt{2}\Phi_4^4\Pi) \le 1/4.$ Therefore,
  \[
  A_n(r) \ge r - \frac{4\varepsilon}{3}\quad\mbox{and}\quad B_n(r) \le r + \frac{8\varepsilon}{3} + \frac{T_{n,r}^{2/3}}{M^{1/4}}\varepsilon \le r + \varepsilon\left(\frac{8}{3} + \frac{1}{6B\log(en)}\right) \le r + \frac{17\varepsilon}{6},
  \]
  the last inequality above holds from~\eqref{eq:FnrbyVarepsilonBound} since $B \ge 1$. This yields
  \begin{align*}
  \mathbb{P}(A_n(r) \le \|U_{n-l+s-1,k}\| \le B_n(r)) &\le \mathbb{P}\left(r - {4\varepsilon}/{3} \le \|U_{n-l+s-1,k}\| \le r + {17\varepsilon}/{6}\right)\\
  &\le \Delta_{n-l+s-1,k}(r - 4\varepsilon/3) + \Delta_{n-l+s-1,k}(r + 17\varepsilon/6)\\
  &\qquad+ \mathbb{P}(r - 4\varepsilon/3 \le \|Y\| \le r + 17\varepsilon/6).
  \end{align*}
  We now want to bound $\Delta_{n-l+s-1,k}(\cdot)$ terms on the right hand side probability using the induction hypothesis and for this we need to show that $r - 4\varepsilon/3$ and $r + 17\varepsilon/6$ both belong to $I_{s-1,M}$ when $r\in I_{s,M}$. For $r - 4\varepsilon/3$, note that
  \[
  r \le -1 + \mathfrak{B}_sn^{1/6}\quad\Rightarrow\quad r - 4\varepsilon/3 \le -1 + \mathfrak{B}_sn^{1/6} \le -1 + \mathfrak{B}_{s-1}n^{1/6},
  \]
  where the last inequality follows since $\mathfrak{B}_s \le \mathfrak{B}_{s-1}$. This implies $r-4\varepsilon/3 \in I_{s-1,M}$. For $r + 17\varepsilon/6$, note that since $s\ge1$ and $r+1 \le \mathfrak{B}_sn^{1/6}$,
  \begin{align}
  &r + \frac{17\varepsilon}{6} + 1\label{eq:rplusvarepsilonBound}\\ &\qquad= (r+1)\left(1 + \frac{17\varepsilon}{6(r+1)}\right) \le \mathfrak{B}_sn^{1/6}\left(1 + \frac{17M^{1/4}T_{n,r}^{\alpha(s)}}{6(r+1)^2}\right)\\ 
  &\qquad\le \mathfrak{B}_sn^{1/6}\left(1 + \frac{17M^{1/4}}{6(r+1)^{17/16}n^{5/32}}\right) \le \mathfrak{B}_sn^{1/6}\left(1 + \frac{17M^{1/4}}{6(\mu+1)^{17/16}n^{5/32}}\right) = \mathfrak{B}_{s-1}n^{1/6}.
  \end{align}
  This implies $r + 17\varepsilon/6 \in I_{s-1,M}$. Hence from the induction hypothesis, we get
  \begin{equation}\label{eq:InductionApplication}
  \begin{split}
  \mathbb{P}(A_n(r) \le \|U_{n-l+s-1,k}\| \le B_n(r)) &\le MT_{n,r-4\varepsilon/3}^{\alpha(s-1)}\mathbb{P}(\|Y\| > r - 4\varepsilon/3)\\ 
  &\qquad+ MT_{n,r+17\varepsilon/6}^{\alpha(s-1)}\mathbb{P}(\|Y\| > r + 17\varepsilon/6)\\ 
  &\qquad+ \mathbb{P}(r - 4\varepsilon/3 \le \|Y\| \le r + 17\varepsilon/6)
  \end{split}
  \end{equation}
  It is clear that $T_{n,r-4\varepsilon/3}^{\alpha(s-1)} \le T_{n,r}^{\alpha(s-1)}$. To bound $T_{n,r+17\varepsilon/6}^{\alpha(s-1)}$ in terms of $T_{n,r}^{\alpha(s-1)}$, note that from~\eqref{eq:rplusvarepsilonBound},
  \[
  \frac{T_{n,r+17\varepsilon/6}^{\alpha(s-1)}}{T_{n,r}^{\alpha(s-1)}} = \left(1 + \frac{17\varepsilon}{6(r+1)}\right)^{3\alpha(s-1)} \le \left(1 + \frac{17M^{1/4}}{6(\mu+1)^{17/16}n^{5/32}}\right).
  \]
  Using~\eqref{eq:3p7}, we get
  \begin{align*}
  \mathbb{P}(r - 17\varepsilon/6 \le \|Y\| \le r + 17\varepsilon/6) &\le \frac{17\Phi_2\Phi_3\varepsilon(r+1)e^{17\Phi_4(r+1)\varepsilon/6}}{6}\mathbb{P}(\|Y\| > r).
  \end{align*}
  Substituting these bounds in~\eqref{eq:InductionApplication}, we get
  \begin{align*}
  \mathbb{P}(A_n(r) \le \|U_{n-l+s-1,k}\| \le B_n(r)) &\le MT_{n,r}^{\alpha(s-1)}\left(2 + \frac{17M^{1/4}}{6(\mu+1)^{17/16}n^{5/32}}\right)\mathbb{P}(\|Y\| > r - 4\varepsilon/3)\\ &\qquad+ 3\Phi_2\Phi_3\varepsilon(r+1)e^{3\Phi_4\varepsilon(r+1)}\mathbb{P}(\|Y\| > r),\\
  \frac{\mathbb{P}(A_n(r) \le \|U_{n-l+s-1,k}\| \le B_n(r))}{\mathbb{P}(\|Y\| > r)} &\le MT_{n,r}^{\alpha(s-1)}\left(2 + \frac{17M^{1/4}}{6(\mu+1)^{17/16}n^{5/32}}\right)\Phi_3e^{4\Phi_4\varepsilon(r+1)/3}\\ &\qquad+ 3\Phi_2\Phi_3\varepsilon(r+1)e^{3\Phi_4\varepsilon(r+1)}.
  \end{align*}
  Combining this with~\eqref{deltanls} yields
  \begin{align*}
  \frac{\Delta_{n-l+s,k}(r)}{\mathbb{P}(\|Y\| > r)} &\le \Phi_2\Phi_3\varepsilon(r+1)e^{\Phi_4\varepsilon(r+1)} + \frac{16C_2\varepsilon^{-2}\beta}{\Phi_0H^2n}\\
  &\quad+ \frac{{2}^{1/2}C_3\varepsilon^{-3}L_ne^{2\mathfrak{C}f_n(r)\log(ep)/\varepsilon}\Phi_3e^{4\Phi_4\varepsilon(r+1)/3}MT_{n,r}^{\alpha(s-1)}}{6n^{1/2}}\left(2 + \frac{3M^{1/4}}{(\mu+1)^{17/16}n^{5/32}}\right)\\
  &\quad+ \frac{3\sqrt{2}C_3\varepsilon^{-3}L_ne^{2\mathfrak{C}f_n(r)\log(ep)/\varepsilon}}{6n^{1/2}}\Phi_2\Phi_3\varepsilon(r+1)e^{3\Phi_4\varepsilon(r+1)}.
  \end{align*}
  To simplify the right hand side, we use $\Phi_4\varepsilon(r+1) = \Phi_4M^{1/4}T_{n,r}^{\alpha(s)} \le \Phi_4M^{1/4}T_{n,r}^{5/16}$ (since $s\ge1$). Also, since $r\in I_{s,M}$, we have $(r+1)^3 \le \mathfrak{B}_s^3n^{1/2} \le \mathfrak{B}_0^3n^{1/2}$ implying that $T_{n,r} \le \mathfrak{B}_0^3 \le G_n \le 1/(81\Phi_4^4M)^{4/5}$. Hence $\Phi_4\varepsilon(r+1) \le 1/3$. Further from~\eqref{eq:FnrbyVarepsilonBound}, $2\mathfrak{C}f_n(r)\log(ep)/\varepsilon \le 1/3$. Therefore,
  \begin{align*}
  \frac{\Delta_{n-l+s,k}(r)}{\mathbb{P}(\|Y\| > r)} &\le \Phi_2\Phi_3\varepsilon(r+1)e^{1/3} + \frac{16C_2\varepsilon^{-2}\beta}{\Phi_0H^2n} + \frac{3\sqrt{2}C_3\varepsilon^{-3}L_ne^{1/3}}{6n^{1/2}}\Phi_2\Phi_3\varepsilon(r+1)e\\
  &\quad+ \frac{{2}^{1/2}C_3\varepsilon^{-3}L_ne^{1/3 + 4/9}\Phi_3MT_{n,r}^{\alpha(s-1)}}{6n^{1/2}}\left(2 + \frac{3M^{1/4}}{(\mu+1)^{17/16}n^{5/32}}\right).
  \end{align*}
  Since $\varepsilon(r+1) = M^{1/4}T_{n,r}^{\alpha(s)}$, the above inequality is equivalent to 
  \begin{equation}\label{eq:AlmostFinalBound}
  \begin{split}
  \frac{\Delta_{n-l+s,k}(r)}{T_{n,r}^{\alpha(s)}\mathbb{P}(\|Y\| > r)} &\le \Phi_2\Phi_3e^{1/3}M^{1/4} + \frac{16C_2\varepsilon^{-2}\beta}{\Phi_0H^2nT_{n,r}^{\alpha(s)}}+ \frac{e^{4/3}C_3\varepsilon^{-3}L_n}{\sqrt{2}n^{1/2}}\Phi_2\Phi_3M^{1/4}\\
  &\quad+ \frac{{2}^{1/2}e^{7/9}C_3\varepsilon^{-3}L_n\Phi_3M}{6n^{1/2}T_{n,r}^{1-3\alpha(s)}}\left(2 + \frac{3M^{1/4}}{(\mu+1)^{17/16}n^{5/32}}\right).
  \end{split}
  \end{equation}
  We now bound each of the terms on the right hand side.
  
  \noindent {\bf Second Term of~\eqref{eq:AlmostFinalBound}:}
  \begin{equation}\label{inter11}
  \begin{split}
  \frac{16C_2\varepsilon^{-2}\beta}{\Phi_0 H^2 n T_{n,r}^{\alpha(s)}} &= \frac{16C_2(r+1)^2\beta}{\Phi_0H^2nT_{n,r}^{3\alpha(s)}M^{1/2}}= \frac{16C_2T_{n,r}^{2/3}n^{1/3}\beta}{\Phi_0H^2nT_{n,r}^{3\alpha(s)}M^{1/2}}= \frac{16C_2\beta T_{n,r}^{4^{-s-1}}}{\Phi_0 H^2 n^{2/3}T_{n,r}^{1/3}M^{1/2}}\\&\le \frac{16C_2\beta }{\Phi_0 H^2 n^{2/3}T_{n,r}^{1/3}M^{1/2}} \le \frac{16C_2\beta T_{n,r}}{\Phi_0 H^2 M^{1/2}} \le \frac{16C_2\beta}{(81)^{4/5}\Phi_0 \Phi_4^{16/5} H^2 M^{13/10}}.  
  \end{split}
  \end{equation}  
  In the last two inequalities we used $n^{-2/3} = \left(n^{-1/2}\right)^{4/3} \le T_{n,r}^{4/3}$ and $T_{n,r} \leq 1/(81 \Phi_4^4 M)^{4/5}$.
  \noindent {\bf Third and Fourth Terms of~\eqref{eq:AlmostFinalBound}:} Note from the definition of $\varepsilon$ that 
  \begin{equation}\label{useful1}
  \frac{\varepsilon^{-3}}{n^{1/2}} = \frac{M^{-3/4}T_{n,r}^{-3\alpha(s)}}{n^{1/2}(r+1)^{-3}} = M^{-3/4}T_{n,r}^{1-3\alpha(s)}.
  \end{equation}
  This yields
  \begin{equation}\label{inter22}
  \begin{split}
  \frac{e^{4/3}C_3\varepsilon^{-3}L_n\Phi_2\Phi_3M^{1/4}}{\sqrt{2}n^{1/2}} = \frac{e^{4/3}C_3L_n\Phi_2\Phi_3T_{n,r}^{1-3\alpha(s)}}{\sqrt{2}M^{1/2}} \le \frac{e^{4/3}C_3L_n\Phi_2\Phi_3}{\sqrt{2}M^{1/2}},
  \end{split}
  \end{equation} 
  and
  \begin{equation}\label{eq:FourthTerm}
  \begin{split}
  \frac{2^{1/2}e^{7/9}C_3\varepsilon^{-3}L_n\Phi_3M}{6n^{1/2}T_{n,r}^{1-3\alpha(s)}} = \frac{2^{1/2}e^{7/9}C_3L_n\Phi_3M^{1/4}}{6}.
  \end{split}
  \end{equation}
  Now combining the bounds for all the terms in~\eqref{eq:AlmostFinalBound}, we get
  \begin{align*}
  \frac{\Delta_{n-l+s,k}(r)}{T_{n,r}^{\alpha(s)}\mathbb{P}(\|Y\| > r)} &\le \Phi_2\Phi_3e^{1/3}M^{1/4} + \frac{16C_2\beta}{(81)^{4/5}\Phi_0 \Phi_4^{16/5} H^2 M^{13/10}} + \frac{e^{4/3}C_3L_n\Phi_2\Phi_3}{\sqrt{2}M^{1/2}}\\ &\qquad+ \frac{2^{3/2}e^{7/9}}{6}C_3L_n\Phi_3M^{1/4} + \frac{e^{7/9}C_3L_n\Phi_3M^{1/2}}{\sqrt{2}(\mu+1)^{17/16}n^{5/32}}.
  \end{align*}
 The right hand side is bounded by $M$ from the definition of $M$.
%
  %
  %
  %
  %
  %
  %
  %
  %
  %
  %
\end{proof}
We are now ready to prove the final large deviation result.
\begin{thm}
For all $n\ge4$, we have
\[
\left|\frac{\mathbb{P}(\|S_n\| > r)}{\mathbb{P}(\|Y\| > r)} - 1\right| \le 1.02MT_{n,r}^{1/3}\]
for all $r$ such that
\begin{equation}\label{finalcond3}
  T_{n,r} \le \mathfrak{B}_0\exp\left(-\frac{17M^{1/4}\log(en)}{6(\mu+1)^{17/16}n^{5/32}}\right). 
\end{equation}
\end{thm}
\begin{proof}
From Lemma~\ref{sixtph} with $s = l:=[\log(en)]$ it follows that
\[
\Delta_n(r) \le MT_{n,r}^{\alpha(l)}\mathbb{P}(\|Y\| > r)\quad\mbox{for}\quad r\in I_{l,M}.
\]
To prove the result, is enough to show that $MT_{n,r}^{\alpha(l)} \le 1.02MT_{n,r}^{1/3}$ and that~\eqref{finalcond3} implies $r \in I_{l,M}$. Note that for all $n\ge4$
\[
T_{n,r}^{-1/3}T_{n,r}^{\alpha(l)} = T_{n,r}^{-4^{-l-1}/3} \le T_{n,r}^{-4^{-\log(en)}/3} \le n^{4^{-\log(en)}/6} = \exp\left(\frac{\log n}{24n^{\log(4)}}\right) \le 1.02.
\]
Now observe that
\[
\mathfrak{B}_l \ge \mathfrak{B}_0\exp\left(-\frac{17M^{1/4}\log(en)}{6(\mu+1)^{17/16}n^{5/32}}\right).
\]
This completes the proof.
\end{proof}
\section{Proof of Corollary \ref{cor:Theorem31}}\label{appsec:corollary41}
The notation $\Theta$ from now on denotes a constant depending only on $\sigma_{\min}$ and $\sigma_{\max}$ and can be different in different lines. Note that the result holds for any $\Phi_0,\ldots,\Phi_4$ satisfying~\eqref{eq:3p7} and in particular, we can take $\Phi_2 = \Theta\log(ep), \Phi_4 = \Theta\log^2(ep)$. Note that $H \ge \Theta K_p^{-1} (\log p)^{-1/\alpha}$, and $C_j = C_0 (\log p)^{j-1}, j=1,2,3.$ This implies that
\begin{align*}
\tilde{\Pi} &\le \Theta(\log(ep))^{7/6},\\
\Pi &\le \Theta\max\left\{(\log(ep))^{7/6}, (\log(ep))^{4/3}, \frac{\log^2(ep)}{\log^2(ep)}, (\log^2(ep))^{4/7}, \frac{\log(ep)(\log(ep))^{2/\alpha}}{\log^{10}(ep)}\right\} = \Theta(\log(ep))^{4/3},\\
M &\le \Theta\max\left\{(\log(ep))^{4/3}, (\log(ep) + \log^2(ep))^{4/3}, \left(\frac{(\log(ep))^{1+2/\alpha}}{(\log(ep))^{32/5}}\right)^{\frac{10}{23}}, (\log(ep))^{2}, \frac{\log^{4}(ep)}{(\mu+1)^{17/8}n^{5/16}}\right\}\\
&= \Theta\max\left\{(\log(ep))^{8/3}, \frac{\log^4(ep)}{(\mu+1)^{17/8}n^{5/16}}\right\} = \Theta(\log(ep))^{8/3}\max\left\{1, \frac{(\log(ep))^{4/3}}{(\mu+1)^{17/8}n^{5/16}}\right\}. 
\end{align*}
Hence $n \ge (\log(ep))^{4(16/15)}(\mu+1)^{-34/5}$ implies $M \le \Theta(\log(ep))^{8/3}$. We now bound $\mathfrak{B}_0$. 
\begin{align*}
\mathfrak{B}_0 &\ge \Theta\min\left\{\frac{1}{(\log^8(ep)\log^{4/3}(ep))^{1/3}}, \frac{1}{(\log^8(ep)\log^{8/3}(ep))^{4/15}},\frac{1}{(\log^2(ep)\log^{1/\alpha}(ep)\log(en))^{1/3}},\right.\\
&\qquad\left.\frac{(\log^{4/3}(ep))^{1/9}}{(\log^{1/\alpha}(ep)\log(ep)\log(en))^{4/9}}, \frac{(\log(ep))^{1/3}}{(\log(ep)\log^{1/\alpha}(ep)\log(en))^{1/2}}\right\}\\
&= \Theta{(\log(e(p\vee n)))^{-28/9}}.
\end{align*}
Substituting these in Theorems \ref{thm:largedeviation} the results follow. Observe that
\[
\exp(-3M^{1/4}(\mu+1)^{-17/16}\log(en)n^{-5/32}) \ge \exp(-\Theta(\log(ep))^{2/3}(\mu+1)^{-17/16}\log(en)n^{-5/32}),
\]
which is bouded below by constant since $n \ge (\log(ep))^{64/15}(\log(en))^{32/5}(\mu+1)^{-34/5}$.
\noeqref{eq:SomeToAllIntegral}
\end{document}